\documentclass{article}

\usepackage[utf8]{inputenc}
\usepackage{graphicx,amsmath,amsfonts,amsthm,physics,mathrsfs,enumerate,tikz-cd}
\usepackage{amssymb}
\usepackage{xcolor}
\usepackage[margin=1in]{geometry}
\usepackage{titlesec}
\titleformat{\subsubsection}[runin]
  {\normalfont\normalsize\bfseries}
  {\thesubsubsection}
  {1em}
  {}

\usepackage{tcolorbox}

\usepackage[colorlinks,citecolor=cyan,linkcolor=teal]{hyperref}

\graphicspath{ {figures/} }

\newcommand{\cev}[1]{\reflectbox{\ensuremath{\vec{\reflectbox{\ensuremath{#1}}}}}}

\DeclareMathOperator{\Aut}{Aut}
\DeclareMathOperator{\Hom}{Hom}

\DeclareMathOperator{\End}{End}

\DeclareMathOperator{\PSL}{PSL}
\DeclareMathOperator{\PGL}{PGL}
\DeclareMathOperator{\SL}{SL}

\DeclareMathOperator{\hull}{hull}

\DeclareMathOperator{\diag}{diag}
\DeclareMathOperator{\supp}{supp}

\renewcommand{\P}{\mathbb{P}}
\DeclareMathOperator{\SO}{SO}
\DeclareMathOperator{\RP}{\mathbb{R}\textrm{P}}

\newcommand{\N}{\mathbb{N}}
\newcommand{\Z}{\mathbb{Z}}
\newcommand{\R}{\mathbb{R}}
\newcommand{\Q}{\mathbb{Q}}

\newcommand{\A}{\mathbb{A}}
\newcommand{\inverse}{^{-1}}

\newcommand{\proj}{\textrm proj}

\newcommand{\cT}{\mathcal{T}}

\newcommand{\sP}{\mathscr{P}}

\DeclareMathOperator{\GL}{GL}
\DeclareMathOperator{\Aff}{Aff}

\DeclareMathOperator{\T}{T}

\DeclareMathOperator{\im}{im}

\DeclareMathOperator{\diam}{diam}

\renewcommand{\proj}{\mathrm{proj}}
\newcommand{\uu}{\underline}

\newcommand{\sC}{\mathscr{C}}

\newcommand{\Tension}{\mathrm{Tension}}

\theoremstyle{definition}
\newtheorem{theorem}{Theorem}[section]
\newtheorem{maintheorem}{Main Theorem}
\newtheorem{maincorollary}[maintheorem]{Main Corollary}
\newtheorem*{thmrepeated}{Theorem \ref{thm: co-stable norm and optimal best maps}}

\newtheorem{definition}[theorem]{Definition}
\newtheorem{lemma}[theorem]{Lemma}

\newtheorem{corollary}[theorem]{Corollary}
\newtheorem{proposition}[theorem]{Proposition}
\newtheorem{remark}[theorem]{Remark}
\newtheorem*{remarks}{Remarks}
\newtheorem{question}[theorem]{Question}
\newtheorem{claim}[theorem]{Claim}

\newtheorem{hypothesis}{}

\newcommand{\Addresses}{{
  \bigskip
  \footnotesize
  \noindent Marit Bobb, \textsc{Max Planck Institute for Mathematics in the Sciences, Leipzig}\par\nopagebreak
  \textit{E-mail address}: \texttt{mbobbmath@gmail.com}
  
  \noindent James Farre, \textsc{Max Planck Institute for Mathematics in the Sciences, Leipzig}\par\nopagebreak
  \textit{E-mail address}: \texttt{james.farre@mis.mpg.de}
  }}

\DeclareMathOperator{\Lip}{Lip}

\newcommand{\cas}{\mathcal{X}^{d+1}}
\newcommand{\cag}{\textrm{Aff}^*(d+1,\mathbb R)}

\usepackage{comment}
\newcommand{\defeq}{\overset{\mathrm{def}}{=}}
\newcommand{\s}{\mathscr S}

\title{Boundary of Anosov coaffine representations}
\author{Marit Bobb and James Farre}
\date{}

\begin{document}

\maketitle

\begin{abstract}
    Let $\Gamma \le \SL(d+1,\R)$ be a torsion-free hyperbolic group dividing a properly and strictly convex domain $\Omega \subset \RP^d$, and let $\rho: \Gamma \to \SL(d+1,\R)$ be the inclusion.  We study representations of $\Gamma$ into the \emph{coaffine group} that are deformations of $\rho\oplus 1 \in \Hom(\Gamma, \GL(d+2,\R))$ and that lie on the boundary of the locus of projective Anosov representations.  In this setting, we identify a canonical limit set $\Lambda$ for the action and study its fine geometric features which are, in a suitable sense, stable.  We prove the existence of an  \emph{ending lamination} that encodes the failure of the Anosov condition and the non-existence of a Cannon--Thurston map. Under certain conditions, we compute the Hausdorff dimension and show that the corresponding measure is positive.
\end{abstract}

\setcounter{tocdepth}{1}
\tableofcontents

\section{Introduction}

Anosov representations form a rich class of discrete representations of hyperbolic groups into Lie groups, with robust geometric and dynamical properties. Their degenerations, however, remain poorly understood outside a few settings.  In this article, we study a special family of \emph{coaffine representations} on the boundary of the projective-Anosov locus, whose limit sets exhibit both similarities to and striking differences from those of geometrically infinite Kleinian groups.

Deformations and degenerations of finitely generated Kleinain groups have been extensively studied, resulting in a mature theory that has driven major advances in low-dimensional topology over the last half-century. Especially important are degenerations of \emph{convex cocompact} Kleinian surface groups and the corresponding complete hyperbolic manifolds, which played a central role in Thurston’s Geometrization Program for understanding the topology of 3-manifolds \cite{Thurston:bulletin_3-manifolds}.

These groups are examples of the much more general \emph{strongly convex cocompact} subgroups $\Gamma\le \PGL(d+1,\R)$ acting on $\RP^d$ introduced by Danciger, Guéritaud, and Kassel \cite{DGK}, which are projective-Anosov when $\Gamma$ is Gromov hyperbolic.
Benoist studied the special case that $\Gamma$ acts cocompactly, i.e., \emph{divides}, a properly and strictly convex domain $\Omega \subset \RP^d$ \cite{Benoist_CDI,Benoist:CDII,Benoist_CDIII}, which gave expansive new insight into compact manifolds admitting a properly convex projective structure.

A projective-Anosov representation $ \rho: \Gamma \to \GL(d+1,\R)$ admits continuous equivariant \emph{boundary maps} $\xi: \partial^\infty \Gamma \to \RP^{d}$ and $\xi^*: \partial^\infty \Gamma \to (\RP^d)^*$ that satisfy compatibility, transversality, and dynamical properties.
Both boundary maps are topological embeddings. The image $\im \xi$ identifies the \emph{proximal limit set} $\Lambda$ for the $\rho$-action of $\Gamma$ on $\RP^d$.  For a convex cocompact Kleinian surface group, $\Lambda$ is a Jordan curve in $\partial^\infty \mathbb H^3$.

At the boundary of the Anosov locus, limit sets can be remarkably rich. 
For a Kleinian surface group representation $\rho:\Gamma\to\mathrm{PSL}(2,\mathbb{C})$ whose conjugacy class is invariant under a pseudo-Anosov mapping class $\phi$,  
Cannon and Thurston gave an astoundingly beautiful description of $\Lambda$ \cite{CT}. Indeed, $\Lambda = \partial^\infty \mathbb H^3$ is the quotient of the circle $\partial^\infty\Gamma$ via an equivalence relation that collapses leaves of the two invariant laminations of $\phi$ to points \cite{CT}. Consequently, there is a group equivariant bounded-to-one Peano curve $\partial^\infty\Gamma \to \Lambda$.   
Such Cannon--Thurston maps for Kleinian (surface) groups have been analyzed extensively since this seminal work (cf. \cite{Mahan:CTsurface,Mahan:ICM}). 
\medskip

We study coaffine representations of groups dividing a strictly and properly convex domain in $\RP^d$ that lie on the boundary of the locus $\mathcal A \subset \Hom(\Gamma, \GL(d+2, \R))$ of projective-Anosov representations.  In this context, we describe the structure of a natural limit set $\Lambda$ and establish the following phenomena: 

\begin{itemize}
    \item \emph{Structural stability of $\Lambda$}: there are two distinct behaviors that are stable under suitable perturbations.
    \item \emph{Non-existence of Cannon--Thurston maps}: there is no continuous equivariant map $\partial \Gamma \to \Lambda$.
    \item \emph{Hausdorff dimension and measure of $\Lambda$}: in many cases, the Hausdorff dimension is as large as possible and the corresponding measure is positive.
    \item \emph{Fine structure of $\Lambda$}: for ``positive'' representations, $\Lambda$ is the epi-graph of a lower-semi-continuous height function $\mathcal H: \partial \Gamma \to \R\cup \{+\infty\}$ that is bounded below.
    \item \emph{Existence of ending laminations}: An oriented geodesic lamination  witnesses the failure of the Anosov condition, and the set of its positive ideal endpoints lies in $\mathcal H\inverse (\{+\infty\})$.
\end{itemize}
See Main Theorems \ref{mainthm: ending lamination intro}, \ref{mainthm: equality of limit sets intro},  \ref{mainthm: structure of limit sets}, \ref{mainthm: structure of hull of Lambda intro}, \ref{mainthm: Lebesgue  intro} and Main Corollaries \ref{maincor: no CT map} and \ref{maincor: hasudorff intro}.

Although there does not seem to be an established notion of ``geometrical infiniteness'' for representations of finitely generated discrete groups into an arbitrary Lie group, our investigation reveals some delicate new phenomena that may help develop this notion further. 

\subsection{Coaffine representations on the boundary of Anosov}\label{subsec: coaffine intro}

In order to get to our main results we first introduce some of the main players and important technical tools.

In this article, $\Gamma$ is always a torsion-free Gromov-hyperbolic group and $\rho: \Gamma \to \SL(d+1,\R)$ is an injective representation dividing a properly and strictly convex domain $\Omega\subset \RP^d$ with ideal boundary $\partial^\infty\Omega\cong \partial^\infty \Gamma$.   Let $M =\Omega/\rho$ be the corresponding convex projective manifold.  

We then study deformations of $\rho\oplus 1$ in the \emph{general coaffine group} 
\[\cag\defeq \left \{ \begin{pmatrix}
    A & \mathbf 0\\
    \alpha& 1
\end{pmatrix} : A \in \GL(d+1,\R), \alpha \in \left(\R^{d+1}\right)^*\right\}< \GL(d+2,\R),\]
which is isomorphic to the stabilizer of a point $[e_{d+2}]$ in $\RP^{d+1}$.
The \emph{linear part} is the element of $\GL(d+1,\R)$ in the upper left-hand block.  
The \emph{coaffine space}  is \[\cas \defeq \RP^{d+1}\setminus[e_{d+2}],\]
and  \emph{coaffine geometry} is the geometry of this pair.  
See \S\ref{subsec: coaffine space} and \S\ref{section: geometry of coaffine domain}.
\medskip

The coaffine space admits a canonical foliation by \textit{vertical} affine lines, those incident to $[e_{d+2}]$. The representation $\rho\oplus1$ preserves the set $\sC\Omega \subset \cas$ of vertical lines  passing through $\Omega\subset \RP^d\subset \cas$.  $\sC\Omega$ is naturally a trivial affine line bundle over $\Omega$, and is convex in any affine chart for $\RP^{d+1}$ that contains it. The vertical foliation of $\sC\Omega$ extends to the boundary $\partial^\infty \sC\Omega \cong \partial^\infty \Omega \times \R$.  
See Figure \ref{fig: intro_cones}.

Let $\chi: \Gamma \to \R_{>0}$ be a character of $\Gamma$, i.e., $\chi = e^{-\varphi}$ where $\varphi \in H^1(\Gamma,\R)$. A representation $\eta: \Gamma \to \cag$ which has linear part $L(\eta): \Gamma \to \GL(d+1,\R)$ equal to the product $\rho\chi$ preserves $\sC\Omega$. Thinking of $\rho$ as fixed, we will be exclusively interested in representations $\eta$ with linear part $\rho\chi$.

\begin{figure}[h]
    \centering
    \includegraphics[width=.7\linewidth]{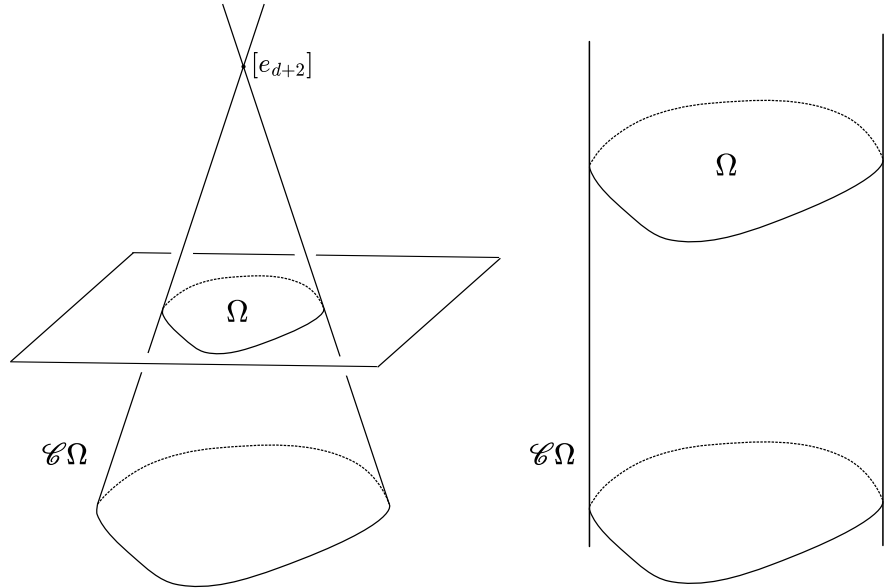}
    \caption{The cone $\sC \Omega$ in two affine charts. On the right, the exceptional point $[e_{d+2}]$ is at infinity. The case $d=2$ and $\sC \Omega \subset \RP^3$ is shown.}  
    \label{fig: intro_cones}
\end{figure}

Our first result supplies an \emph{ending geodesic lamination} for such coaffine representations in $\partial \mathcal A$.  A \emph{geodesic lamination} $\lambda \subset \Omega$ is a closed set that is foliated by complete projective line segments.  If $\lambda$ is $\rho$-invariant, then it descends to a geodesic lamination $\uu\lambda \subset M$.

For $\eta \in \Hom(\Gamma,\GL(d+2,\R))$ and $\gamma \in \Gamma$, let $\lambda_i^\eta(\gamma)$ be the logarithm of the modulus of the $i$th largest eigenvalue of $\eta(\gamma)$.

\begin{maintheorem}[Existence of ending laminations]\label{mainthm: ending lamination intro}
    Suppose $\eta: \Gamma \to \cag$ lies in $\partial \mathcal A$ and has linear part $L(\eta)$ preserving $\Omega$, as above.  Then, there is a non-empty oriented chain recurrent geodesic lamination $\uu \lambda \subset M$ that witnesses the failure of the $1$-Anosov condition in a strong sense:
    there is a sequence of oriented closed curves $\gamma_n\subset M$ such that 
    \begin{equation}\label{eqn: length to zero}
        (\lambda_1^\eta - \lambda_2^\eta)(\gamma_n) \to 0, \text{ as } n\to \infty,
    \end{equation}
     every limit point of any sequence satisfying \eqref{eqn: length to zero} 
     is contained in $\uu\lambda$  compatibly oriented,
     and $\uu\lambda$ is the union of such limit points.  Convergence is in the Hausdorff topology on closed subsets of $M$.
\end{maintheorem}

An orientation of $\uu\lambda$ is a continuous orientation of its leaves.  
Chain recurrence is a dynamical condition that constrains the behavior of $\uu\lambda$.  That $\uu\lambda$ is chain recurrent is equivalent to being approximated in the Hausdorff topology by unions of closed geodesics; see \S\ref{subsec: geodesic laminations}.  
\begin{remarks}\label{rmk: first theorem}\leavevmode\par
    \begin{enumerate}
        \item $\lambda_1^\eta(\gamma) = \lambda_1^{L(\eta)}(\gamma)$ and $\lambda_2^\eta(\gamma) =\max\{0, \lambda_2^{L(\eta)}(\gamma)\}$, so the failure of the Anosov condition depends only on the linear part of $\eta$; this is related to the fact that $\cag$ is not semi-simple.
        \item If $\im \rho \le \SO_\circ(d,1)$, then $(\lambda_1^\rho-\lambda_2^\rho)(\gamma)$ is the length of $\gamma$ in $M$ with its hyperbolic metric.
        \item $\uu\lambda$ need not be connected, and it need not carry a transverse measure of full support. 
        \item We have not asked that $\gamma_n$ be primitive, so $\gamma_n = \gamma^n$ for some $\gamma \in \Gamma$ is one possibility.   
        \item $\uu\lambda$ must satisfy a homological positivity condition.\footnote{Namely, that every almost loop in $\uu\lambda$ is non-zero in homology.}  See Remark \ref{rmk: construct tight maps} for a realization result in $d=2$.
    \end{enumerate}   
\end{remarks}

We highlight an analogy to Kleinian groups, from which we have borrowed the notion of an ending lamination.
The ends of a complete hyperbolic $3$-manifold with finitely generated fundamental group are homeomorphic to a surface times an interval and are either \textit{geometrically finite} or \textit{simply degenerate} \cite{Thurston:notes, Bonahon:ends, Agol:tameness, CG}. Simply degenerate ends admit a sequence of closed curves of bounded length exiting the end. The \textit{ending lamination} is a limit of any such sequence of curves in the measure topology \cite{Thurston:notes}, and Kleinian surface groups are completely classified by their end invariants \cite{ELTI,ELTII}.

See \S\ref{sec: failure of anosov} for another version of Main Theorem \ref{mainthm: ending lamination intro} in terms of the measure topology on curves that makes the analogy between this setting and Kleinian groups stronger. 

\subsection{Best forward Lipschitz maps and the co-stable norm}\label{subsec: best lipschitz intro}
There is a simple criterion for $\eta$ to be in $\partial \mathcal A$. We found Main Theorem \ref{mainthm: ending lamination intro}  by interpreting it geometrically. Here is the criterion, which is very closely related to \cite{Lahn:I,Lahn:II} and an observation made in \cite{Barbot:stablenorm}; see Corollary \ref{cor: failure of anosov} for a  proof.

\begin{proposition}\label{prop: stable conorm anosov intro}
    Let $\eta$ be a coaffine representation  with linear part $L(\eta) = \rho \chi$, where $\chi = e^{-\varphi}$.  Then  $\eta$ is $1$-Anosov if and only if 
    \[\|\varphi\|_\rho\defeq \sup_{\gamma \in \Gamma}\frac{\varphi(\gamma)}{\lambda_1^\rho(\gamma)}<1.\]
\end{proposition}

The geometric point of view arises by considering a certain $\rho$-invariant asymmetric Finsler metric $d_\rho$ on $\Omega$ called the \emph{domain shape} metric.\footnote{There is really a family of such metrics, but we choose one.} 
It is ``half of the Hilbert metric.''  Domain shape metrics were introduced by Danciger and Stecker and used by Reid in studying the asymptotic geometry of Hitchin representations into $\SL(3,\R)$ \cite{Reid:finsler}.
It has the following important features:
\begin{itemize}
    \item $d_\rho$ is uniquely geodesic, and the geodesics are (directed) projective line segments in $\Omega$.
    \item The minimal translation length of a non-trivial element $\gamma \in \Gamma$ is $\lambda^\rho_1(\gamma)$.
    \item The geodesic flow $g_t: T^1\Omega \to T^1\Omega$ is well defined, $C^1$,  and Anosov.  
\end{itemize}
If $\partial \Omega$ is an ellipsoid, then $d_\rho$ is just the hyperbolic metric.  See \S\ref{sec: domain shape}.

The relationship between $\|\cdot \|_\rho$ as a function on $H^1(\Gamma, \R)$ and $d_\rho$ is:  $\|\cdot\|_\rho$ is an (asymmetric) norm called the \emph{co-stable norm} with respect to $d_\rho$.  The co-stable norm and the dual stable norm on $H_1(\Gamma,\R)$ were first studied independently by Gromov and Federer (see \cite{Bangert:equiv}).   They are useful in a variety of contexts, e.g., \cite{Bangert:minimalgeodesics, AL:stablenorm, McSR}.  

The co-stable norm supplies us with a powerful tool, which we summarize here. Theorem \ref{thm: co-stable norm and optimal best maps} gives more details.
When dealing with asymmetric metric spaces, it is important to specify a \emph{directionality} for Lipschitz maps.
A map $f: \Omega \to \R$ is $L$-forward Lipschitz if for all $x, y \in \Omega$,
\[f(y) - f(x)\le Ld_\rho(x,y).\]

\begin{theorem}\label{thm: tight map and tension intro}
    Let $d_\rho$ be the domain shape metric on $\Omega$ and let $\varphi \in H^1(\Gamma,\R)$ with $\|\varphi \|_\rho =1$.  There is a $1$-forward Lipschitz $\varphi$-equivariant map 
    \[\tau: \Omega \to \R\] 
    such that the set of directed lines in $\Omega$ along which $\tau$ is isometric is a non-empty $\rho$-invariant geodesic lamination $\lambda$. 
    
    Furthermore, $\tau$ is \emph{best} in the sense that it minimizes the forward Lipschitz constant for any $\varphi$-equivariant map $\Omega \to \R$, and $\tau$ is \emph{optimal} in the sense that no line segment that is not in $\lambda$ is mapped isometrically to $\R$ by $\tau$.  
\end{theorem}

 The quotient $\uu\lambda \subset M$ is the lamination from Main Theorem \ref{mainthm: ending lamination intro}.   We call $\lambda$ the \emph{tension locus} or \emph{stretch locus} of $\tau$, and we refer to $\tau$ as an \emph{optimal tight map}.

\begin{remark}\label{rmk: construct tight maps}
    Using the tools of \cite{CF,CFII}, for hyperbolic surfaces, it is possible to construct tight maps with prescribed tension locus (if you are willing to change the hyperbolic metric).  See \cite[\S5]{FLM:I} and \cite[\S7]{FLM:II}.
\end{remark}

A consequence of Theorem \ref{thm: tight map and tension intro} is that for all $\varphi \in H^1(\Gamma,\R)$, $\|\varphi\|_\rho$ measures the minimum forward Lipschitz constant for a $\varphi$-equivariant map $\Omega \to \R$.  Compare with \cite[Theorem 8.5]{Thurston:stretch} and \cite[Theorem 1.3]{GK:stretch}.


\subsection{Main results: limit sets and stability}\label{subsec: main results}

We take $\rho: \Gamma \to \SL(d+1, \R)$ dividing $\Omega \subset \RP^d$ and $\varphi\in H^1(\Gamma,\R)$ satisfying $\|\varphi\|_\rho=1$.  Main Theorem  \ref{mainthm: ending lamination intro} gives us a chain recurrent oriented geodesic lamination $\uu \lambda \subset M = \Omega/\rho$.  
We work with two sets of hypotheses:
\begin{hypothesis}\label{hyp A}
    $d\ge 2$ and $\uu\lambda$ is a multicurve.\footnote{Our methods also apply to the case that $d\ge 2$ and $\uu\lambda$  has finitely many leaves but not all closed.  We do not treat this case in the text. We note that the setting of hypotheses \ref{hyp A} forms an open subset of the sphere $\{\|\cdot\|_\rho=1\}\subset H^1(\Gamma,\R)$.}
\end{hypothesis}
\begin{hypothesis}\label{hyp B}
  $d=2$ and $\uu\lambda$ has an infinite leaf.  
\end{hypothesis}

The coaffine representations that we consider in this paper are of the form
\[\eta(\gamma) = \begin{pmatrix}
    \rho(\gamma)\chi(\gamma) & \mathbf 0\\ Z(\gamma) &1
\end{pmatrix},\]
where $Z: \Gamma \to (\R^{d+1})^*$ is a function satisfying a cocycle condition.  The coaffine representations with linear part $\rho\chi$ are parameterized by the vector space $Z^1(\Gamma, (\R^{d+1})_{\rho\chi}^*)$ of such cocycles.  The cohomology $H^1(\Gamma, (\R^{d+1})_{\rho\chi}^*)$ parametrizes coaffine representations with linear part $\rho\chi$ up to conjugation by the translational subgroup of $\cag$; see \S\ref{section: geometry of coaffine domain}.

\begin{definition}\label{def: cone intro}
    The \textbf{positive cone} $\sP_{\rho\chi} \subset H^1(\Gamma,(\R^{d+1})^*_{\rho\chi})$ is an open convex cone cut out by  linear inequalities associated to the minimal components of $\uu\lambda$.  See Definitions \ref{def: positive cone and cocylces} (hypotheses \ref{hyp A}) and  \ref{def: positive cone} (hypotheses \ref{hyp B}).
    
    We write $\eta \in \sP_{\rho\chi}$ if $L(\eta) =\rho\chi$ and if the cohomology class of the cocycle defining $\eta$ lies in $\sP_{\rho\chi}$.
\end{definition}

 If $\uu\lambda$ is a mulitcurve with $k$-components, then $\sP$ is a polyhedral cone with at most $k$-faces, i.e., it is the  intersection of $k$ linear half-spaces. 
 If $\uu \lambda$ is not finite, then this cone is an intersection of finitely many families of Cantor sets worth of linear functionals.   
With the exception of the case that $\uu\lambda$ is a multi curve with a small number of components, it is not clear if $\sP_{\rho\chi}$ is non-empty.
\medskip 

To a coaffine representation $\eta$ and point $o\in \sC\Omega$, we associate the \textbf{orbital limit set}
\[\Lambda_{orb}^o \defeq \overline{\eta(\Gamma).o}\cap \partial^\infty\sC\Omega. \]
 An element $\gamma\in \Gamma$ is $\eta$-proximal if $\lambda_1^\eta(\gamma) > \lambda_2^\eta(\gamma)$.  The unique attracting eigenline of $\eta(\gamma)$ is denoted $\eta(\gamma)^+$.  Define the \textbf{proximal limit set}
\[\Lambda_{prox} \defeq \overline{\{\eta(\gamma)^+\in\RP^{d+1} \colon \gamma\in \Gamma \textrm{ is  $\eta$-proximal} \}} \subset \partial^\infty\sC\Omega.\]

By an easy dynamical argument, we always have $\Lambda_{prox} \subset \Lambda_{orb}^o$ for all $o \in \sC\Omega$. Equality holds for Anosov representations, but need not otherwise \cite{DGK}.
The proof of the following result is surprisingly involved.  
\begin{maintheorem}\label{mainthm: equality of limit sets intro}
    Under hypotheses \ref{hyp A} or \ref{hyp B}, if  $\eta \not \in \partial (\sP_{\rho\chi}\cup - \sP_{\rho\chi})$, then for all $o \in \sC\Omega$, we have 
    \[\Lambda_{prox}=\Lambda_{orb}^o.\]
\end{maintheorem}

 Therefore, we say that $\eta$ is \textbf{exceptional} if $\eta \in  \partial (\sP_{\rho\chi}\cup - \sP_{\rho\chi})$ and \textbf{unexceptional} otherwise.  For unexceptional $\eta$, Main Theorem \ref{mainthm: equality of limit sets intro} allows us to make the following definition:
\[\Lambda \defeq \Lambda_{prox} = \Lambda_{orb}^o.\]

\begin{remark}
    Under hypotheses \ref{hyp A}, if $\eta$ is exceptional, then the \emph{full orbital limit set} $\bigcup_{o\in \sC\Omega} \Lambda_{orb}^{o}$ is equal to $\partial^\infty \sC\Omega$; see also Remark \ref{rmk: characterizing exceptional with a ray or line}.
    Under hypotheses \ref{hyp B}, if $\eta$ is exceptional, we do not know if the full orbital limit set is \emph{ever} different from $\Lambda_{prox}$.
\end{remark}

Denote by $\lambda^+\subset \partial^\infty\Omega$ the forward endpoints of the directed lines constituting $\lambda$.  

\begin{maintheorem}\label{mainthm: structure of limit sets}
    Under hypotheses \ref{hyp A} or \ref{hyp B}, for unexceptional $\eta$ with $L(\eta) =\rho\chi$, we have the following alternatives:
    \begin{enumerate}  
        \item If $\eta \not \in \overline{(\sP_{\rho\chi}\cup - \sP_{\rho\chi})}$ then we have an equality between the limit set and the boundary of the cone
        \[\Lambda=\partial^\infty \sC\Omega.\]  
        \item If $\eta \in \sP_{\rho\chi}$ then  there is a lower-semi-continuous ``height'' function $\mathcal H: \partial^\infty \Omega \to \R\cup \{+\infty\}$, which is bounded below, such that
        \[\Lambda= \operatorname{epi-graph}(\mathcal{H}). \]
        Furthermore, $\mathcal H\inverse(+\infty) \supset \lambda^+$.
        \item If $\eta \in -\sP_{\rho\chi}$ then there is an upper-semi-continuous ``height'' function $\mathcal H: \partial^\infty \Omega \to \R\cup \{-\infty\}$, which is bounded above, such that
        \[\Lambda= \operatorname{hypo-graph}(\mathcal{H}). \]
        Furthermore, $\mathcal H\inverse(-\infty) \supset \lambda^+$.
    \end{enumerate}
\end{maintheorem}

\begin{remarks}\leavevmode\par
\begin{enumerate}
    \item The function $\mathcal H$ depends on a choice of trivialization $\partial^\infty\sC\Omega\cong \partial^\infty \Omega\times \R$.  It is not $\Gamma$-invariant. 
    \item In alternative 2 of Main Theorem \ref{mainthm: structure of limit sets},  $\mathcal H\inverse (+\infty)$ properly contains $\lambda^+$. It is always uncountable. 
    See Remarks \ref{rmk: lots of rays going to infinity} and \ref{rmk: infty is more than lambda} and Question \ref{q: Hinverse}. 
\end{enumerate}
\end{remarks}

 We emphasize that openness of $\sP_{\rho\chi}$ gives that all three alternatives are \emph{stable} under perturbing the cocycle defining $\eta$. We will see that a combination of robust hyperbolic dynamics and parabolic dynamics are responsible for this phenomenon. Compare with the notion of a \emph{blender} from smooth dynamical systems \cite{Blender}.
\medskip

Alternatives 2 and 3 are especially interesting.  We focus on the second: $\Lambda$ is a closed set that looks like the cone over a very erratic height function.   Indeed, the semi-continuity property gives that for any sequence of points $x_n \to x \in \lambda^+$, $\mathcal H(x_n) \to +\infty$.  
Since  $\lambda^+$ is $\Gamma$-invariant, it is a dense set in $\partial^\infty \Omega$.

\begin{maincorollary}\label{maincor: no CT map}
    Under hypotheses \ref{hyp A} or \ref{hyp B}, if $\eta$ is unexceptional, then there is no continuous equivariant surjection $\partial^\infty \Omega \to \Lambda$, i.e., there is no Cannon--Thurston map.
\end{maincorollary}

See \S \ref{subsec: proof of corollary} for the uncomplicated proof.
\medskip

In alternative 2 of Theorem \ref{mainthm: structure of limit sets}, the convex hull of $\Lambda$ is particularly interesting. Combining Proposition \ref{prop: topology of Cmin},  previous work of the authors in \cite{BF}, and the lower bound on $\mathcal H$ in Theorem \ref{mainthm: structure of limit sets} gives the following result.

\begin{maintheorem}\label{mainthm: structure of hull of Lambda intro}
    In alternative 2, the convex set $\hull(\Lambda)$ in $\sC\Omega$ is properly convex.   Also, $\hull(\Lambda)$ has a single boundary component on which $\eta$ acts cocompactly.   This boundary component accumulates onto all of $\Lambda$ (and to $[e_{d+2}]$ in $\RP^{d+1}$).  When $d=2$, this invariant surface is \emph{bent} along its \emph{bending geodesic lamination} according to an \emph{affine transverse measure}.
\end{maintheorem}

See  \cite{BF} for precise notions of affine measure and bending lamination. 
We remark that our previous work implies that the affine transverse measure actually varies continuously for $\eta\in \sP_{\rho\chi}$ \cite{BF}.

\medskip

These domains bear some resemblance to Kleinian surface groups with exactly one simply degenerate end, which also have one convex core boundary component that is bent along a measured geodesic lamination.
The limit set of such a Kleinian surface group is topologically the dendrite obtained as the quotient of the circle by the equivalence relation that identifies pairs of endpoints of leaves of the ending lamination \cite{Mahan:CTsurface}; compare with the statement that $\mathcal H\inverse(\pm\infty)\supset \lambda^+$.  The dendrite in the Kleinian setting has Hausdorff dimension 2, but vanishing $2$-dimensional Lebesgue measure \cite{Canary:ends,Agol:tameness, CG}. 
\medskip

The sphere $\partial^\infty \Omega$ is a $C^1$-submanifold of $\RP^d$, preserved by the smooth action of $\rho$, so there is a $(d-1)$-dimensional measure in the class of Lebesgue on $\partial^\infty \Omega$ which is quasi-invariant under the action of $\rho$.  Similarly, there is a $d$-dimensional measure in the Lebesgue class on $\partial^\infty \sC \Omega$ which is quasi-invariant by the action of  $\eta$.  
The limit set $\Lambda$ is quite large, at least under hypotheses \ref{hyp A}: 

\begin{maintheorem}\label{mainthm: Lebesgue  intro}
    Under hypotheses \ref{hyp A}, when $\eta$ is unexceptional, the $d$-dimensional Lebesgue measure of $\Lambda$ is positive.
\end{maintheorem}
This immediately implies the following.
\begin{maincorollary}\label{maincor: hasudorff intro}
    Under hypotheses \ref{hyp A}, when $\eta$ is unexceptional, the Hausdorff dimension of $\Lambda$ is $d$. 
\end{maincorollary}

Obviously, when $\Lambda = \partial ^\infty \sC \Omega$, the $d$-dimensional Lebesgue measure of $\Lambda$ is positive and it thus has Hausdorff dimension $d$.
\begin{remark}
    We expect the conclusion of Main Theorem \ref{mainthm: Lebesgue  intro} to hold under hypotheses \ref{hyp B} and to be accessible using the techniques developed in this paper. However, such a proof appears to require substantial additional technical work and record-keeping, which we do not pursue here.
\end{remark}

\subsection{Context and relation to other works}

\subsubsection{The Barbot component}

Barbot examined a class of Anosov representations of surface groups with image in $\SL(3,\R)$ obtained as small deformations of $\mu \oplus 1$ where $\mu:\pi_1S\to \SL(2,\R)$ is a lift of a hyperbolization of a closed surface  into $\PSL(2,\R)$  \cite{Barbot:stablenorm}. 
The limit curve of a Barbot representation is contained in no affine chart in $\RP^2$. 
A subclass of Barbot's representations are coaffine, in particular, those with image in $\operatorname{Aff}^*(2,\R)$, see in particular \S 4.1.2 of \cite{Barbot:stablenorm}, wherein he discusses the stable norm.

The main difference between Barbot's setting and the present work is the absence of the cone $\mathscr C \Omega$. We use heavily that $\partial \mathscr C \Omega \cong \partial \Omega \times \R$, the trivial line bundle over $\partial \Omega \cong S^1$. 
The analogous statement for a Barbot coaffine representation is that the circle $\partial \pi_1 S$ is identified via the representation $\mu$ with $\RP^1$, and the quotient $Q: \RP^2 \setminus [e_3]\to \RP^1 \cong \partial \pi_1 S$ induces the structure of the twisted affine line bundle $S^1 \widetilde{\times}~ \R$, the M\"obius bundle.

As a result of the non-orientability of this bundle, positivity is not defined in this context. However, one may still, by way of the vanishing or non-vanishing of the linear conditions $Z(\gamma)(\mu(\gamma)_+)$ define a notion of exceptional and unexceptional representations.

It is possible to follow the techniques from \S\S \ref{sec: heights bounded multicurve}-\ref{sec: prox is orb} to produce rays in the proximal limit set.  Minimality of the action on the circle and non-orientability of the bundle then produce \emph{lines} in the proximal limit set.  Thus any unexceptional boundary Anosov coaffine Barbot representation has $\Lambda = \RP^2$.

\subsubsection{Convex cocompact surface subgroups of $\PGL(4,\R)$}

There are three classical settings of convex cocompact surface groups acting on homogeneous (pseudo)-Riemannian $3$-manifolds, all of which are sub-geometries of $(\SL(4,\R), \RP^3)$:
hyperbolic space by isometries in $\SO_\circ(3,1)$;
co-Minkowski space by isometries in $\SO_\circ(2,1)\ltimes (\R^{3})^*$; and  anti-de Sitter space by isometries in $\SO_\circ (2,2)$.
All three are of independent interest.

Kleinian surface groups admit interesting degenerations, while those in the anti-de Sitter setting do not \cite{Mess,Nots_on_Mess,Barbot:deformations}.  
The present article and our previous work \cite{BF} constitute the authors' first attempts to understand the structures and phenomena that arise for projective Anosov subgroups of $\PGL(4,\R)$ beyond these classical geometries.

\subsubsection{Other behaviors on the boundary of Anosov representations}
Besides finitely generated Kleinian groups, discussed above, there are some other contexts where the geometry of the limit sets of representations on the boundary of the locus of Anosov representations has been studied.

In another rank-one setting, Schwartz identified representations of triangle reflection groups into $\mathrm{SU}(2,1)$ that lie on the boundary of the Anosov locus.  He studied the complement of the limit set in the $3$-sphere at infinity in the complex hyperbolic plane and proved that some hyperbolic manifolds can be realized as the quotient \cite{Schwartz:GPconj, Schwartz:whitehead}; see also \cite{GP,FP:moduli}.

In \cite[\S7.5]{GGKW17}, the authors give an example of a free group mapping into $\SO(2,1)\times \SO(2,1)$ on the boundary of the projective-Anosov locus that has a perfectly nice continuous, equivariant, transverse boundary map.  

Lee--Lee--Stecker examined certain triangle group representations into $\SL(3,\R)$ which are deformations of the reducible representations into $\SL(2,\R)\oplus 1$ \cite{SL:triangle}. The Anosov representations form a pair of rays, the endpoints of which are discrete and faithful, but not Anosov. In this case, they found that the representations on the boundary of Anosov do admit equivariant limit curves, but those curves fail to be transverse.

In \cite{Barbot-Lee-Valerio, Valerio, Schwartz_Pappus}, the authors study actions of the modular group on $\RP^2$ and the symmetric space $\SL(3,\R) / \SO(3)$ which are on the boundary of Anosov. In the most recent iteration, Schwartz identifies the closure of a component of Anosov representations and gives a satisfying description of its boundary in terms of ``Pappus representations''. In these cases there is a continuous equivariant (but not injective) map from the boundary of the group into the space of flags in $\R^3$.

The boundary of a component of Anosov representations has been studied or identified in a number of settings; see \cite{Coxeter,weakly_maximal,BP:kpositive,Lahn:I,Lahn:II, Barbot:stablenorm}.

\subsubsection{Slack in other contexts}
In recent and forthcoming work of Danciger, Gu\'eritaud, and Kassel \cite{DGK:slackcalculus,DGK:slacksurface}, a notion of ``slack'' corresponding to crossing curves in a surface is used to study a different extremal problem: that of understanding the shape of the  boundary of the Benoist limit cone.  Our notion of slack relates curves to a function, while their notion relates two curves to each other.  Both quantities can be imagined by the excess length of a strand of rope if the (ideal) endpoints are fixed: hence the name slack.

In recent work of the second named author together with  Dal'bo, Landesberg,  and Minsky \cite{FLM:I, FLM:II, FLM:III, DFLM}, equivariant $1$-Lipschitz maps and the geometry of their stretch laminations are an essential tool to understand horocycle orbit closures (and leaves of the horospherical foliation) on certain geometrically infinite hyperbolic surfaces (and higher dimensional hyperbolic manifolds).  
It is curious that this common tool has played such an important role in these investigations, which do not seem closely related to one another.

\subsection{Outline}
\S\ref{sec: ingredients proof} outlines some of the ideas for the proofs of our Main Theorems and provides the proof of Main Corollary \ref{maincor: no CT map}.  It also serves as a guide to some of the more technical sections of the paper.

The first part of the paper, \S\S\ref{sec: background}--\ref{sec: sections and height}, assembles many basic ingredients and establishes existence of ending laminations.
In \S\ref{sec: background}, we recall some preliminaries on linear algebra, convex projective geometry, coaffine space, Anosov representations, and geodesic laminations.
\S\ref{section: geometry of coaffine domain} describes some basic structural features of $\sC\Omega$, its boundary, the minimal convex invariant domain, and the various limit sets.
In \S\ref{sec: domain shape}, we introduce the domain shape metric, recall its basic properties, and prove a result (Proposition \ref{prop: singular values and DS metric}) relating the distance between points in a $\Gamma$-orbit to singular values.  \S\ref{sec: stable norm} introduces the co-stable norm on cohomology and optimal tight maps.  Proofs are deferred to Appendix \ref{appendix}.
In \S\ref{sec: slack}, we develop some basic continuity and finiteness properties of slack as well as quantify its growth away from the tension lamination.  In \S\ref{sec: failure of anosov}, we prove Main Theorem \ref{mainthm: ending lamination intro} and some related results using slack and geometric features of $\tau$.  In \S\ref{sec: sections and height}, we construct an $\eta$-equivariant surface in $\sC\Omega$ and define height.

The second part, \S\S\ref{sec: heights bounded multicurve}--\S\ref{sec: measure}, is the heart of the paper in which we prove the remaining Main Theorems.  
Under hypotheses \ref{hyp A}, \S\ref{sec: heights bounded multicurve} defines the positive cone $\sP_{\rho\chi}$ and bounds the height of the invariant surface from below when $\eta \in \sP_{\rho\chi}$.  
We then repeat these steps in \S\S\ref{sec: positive cone lamination}--\ref{sec: heights bounded below general} to prove the same theorem under hypotheses \ref{hyp B}; these sections are more technical.  \S\ref{sec: positive cone lamination} defines the positive cone in this setting, while \S\ref{sec: heights bounded below general} uses this as input to prove that heights are bounded below in the positive cone.  
In \S\ref{sec: prox is orb}, we revisit the structure of the minimal convex domain and prove Main Theorems \ref{mainthm: equality of limit sets intro} and \ref{mainthm: structure of limit sets} by finding a ray in the proximal limit set.  
Finally, in \S\ref{sec: measure}, we prove Main Theorem \ref{mainthm: Lebesgue  intro}.  
We collect some questions in \S\ref{sec: questions}.

\subsection*{Acknowledgments}
The authors would like to thank Teddy Weisman and Jeff Danciger for several enlightening conversations related to this work.  The authors also thank Fabi Lander and Steve Trettle for software that helped us visualize the limit sets studied in the paper.

This material is based partially upon work supported by the National Science Foundation under Grant No. DMS-1928930, while the second named author was in residence at the Simons Laufer Mathematical Sciences Institute in Berkeley, California, during the Spring of 2026.
The second named author acknowledges the support of the Institut Henri Poincaré (UAR 839 CNRS-Sorbonne Université) and LabEx CARMIN (ANR-10-LABX-59-01).

\subsection*{AI disclosure}
The vast majority of the work was done before the rupture in mathematics brought about by LLMs took place.
Generative AI is not responsible for any of the content of this article apart from speeding the proof of Claim \ref{clm: AI claim} and a helpful suggestion for the proof of Proposition \ref{prop: found a ray} during the editing process, which it helped expedite.   Additionally, the first named author obtained some poor estimates on shipping volume from an LLM, and the second used the same LLM to compare puppy food to purchase, among other mundane tasks.

\section{About the proofs}\label{sec: ingredients proof}

We now outline some of the main ingredients and ideas underlying the proofs of our Main Theorems. This section serves as a guide to the technical sections that follow, and we invite the reader to return to it when useful for understanding how the various arguments fit together.
\medskip

Recall from Theorem \ref{thm: tight map and tension intro} that the domain shape asymmetric metric $d_\rho$ on $\Omega$ endows $H^1(\Gamma,\R)$ with the co-stable norm $\|\cdot \|_\rho$.  Given $\varphi \in H^1(\Gamma,\R)$ with $\|\varphi\|_\rho=1$, the geometry of a $\varphi$-equivariant \emph{optimal tight map} $\tau: \Omega \to \R$ from  the same Theorem provides an indispensable tool throughout the paper.  

\begin{definition}[Slack]\label{def: slack intro}
    Given a rectifiable directed path $\alpha : [s,t] \to \Omega$, the \textbf{$\tau$-slack} of $\alpha$ is 
\[\s(\alpha)\defeq \ell(\alpha) - (\tau(\alpha(t))-\tau(\alpha(s))),\]
where $\ell(\alpha)$ is its length in $\Omega$. 
\end{definition}

Since $\tau$ is $1$-forward Lipschitz, $\ell(\alpha) +\ell(\overline \alpha) \ge\s(\alpha)\ge0$ for all paths $\alpha$. Slack can be defined for lines or rays and take values in $[0,+\infty]$ by taking a supremum over finite segments.  

One indication that slack is an important quantity is given by the following result (proved in \S\ref{subsec: equivariant section of cone}), which is used constantly for the proof of Main Theorem \ref{mainthm: structure of limit sets}.
For two functions $f$ and $g$, we use the notation $f \asymp g$ to mean that there is a constant $C$ such that $\frac fC \le g \le Cf$.  

\begin{proposition}\label{prop: basic slack estimate intro}
    Choose a point $o\in \Omega$.  For any $\gamma \in \Gamma$, we have that 
    \[\|\rho\chi(\gamma)\| \asymp e^{\s[o,\gamma o]},\]
    where $[o,\gamma o]\subset \Omega$ is the directed projective line segment joining $o$ to $\gamma.o$.  The implicit constant depends on a choice of norm $\|\cdot \|$ for $\R^{d+1}$ and the point $o \in \Omega$, but does not depend on $\gamma \in \Gamma$.  
\end{proposition}

\subsection{The positive cone}\label{subsec: positive cone expl}

Let $Q : \sC \Omega \cup \partial^\infty\sC\Omega \to \Omega \cup \partial^\infty \Omega$ be the natural projection.  Its fibers are vertical affine lines.
We define a height function on these lines that gives a trivialization $\sC \Omega \cup \partial^\infty \sC\Omega 
\cong 
\left(\Omega \cup \partial^\infty \Omega \right)\times \R$.  The important property is that it detects when sequences of points go toward $[e_{d+2}]$ in one direction or the other.

For illustration, let $d=2$. Suppose that the tension lamination $\uu \lambda$ contains a simple closed curve $\delta$. In this case, up to conjugating the representation within $\Aff^*(3,\R)$, $\eta(\delta)$ may be assumed to have the form 
\begin{equation}\label{eqn: matrix unipotent intro}
    \eta(\delta) = 
\begin{pmatrix}
    1 & 0 & 0 & 0\\
    0 & e^{\lambda_2^\rho(\delta)-\varphi(\delta)} & 0 & 0 \\
    0 & 0 & e^{\lambda_3^\rho(\delta)-\varphi(\delta)} & 0\\
    r & 0 & 0 & 1
\end{pmatrix}.
\end{equation}

\begin{figure}[h]
    \centering
    \includegraphics[width=.35\linewidth]{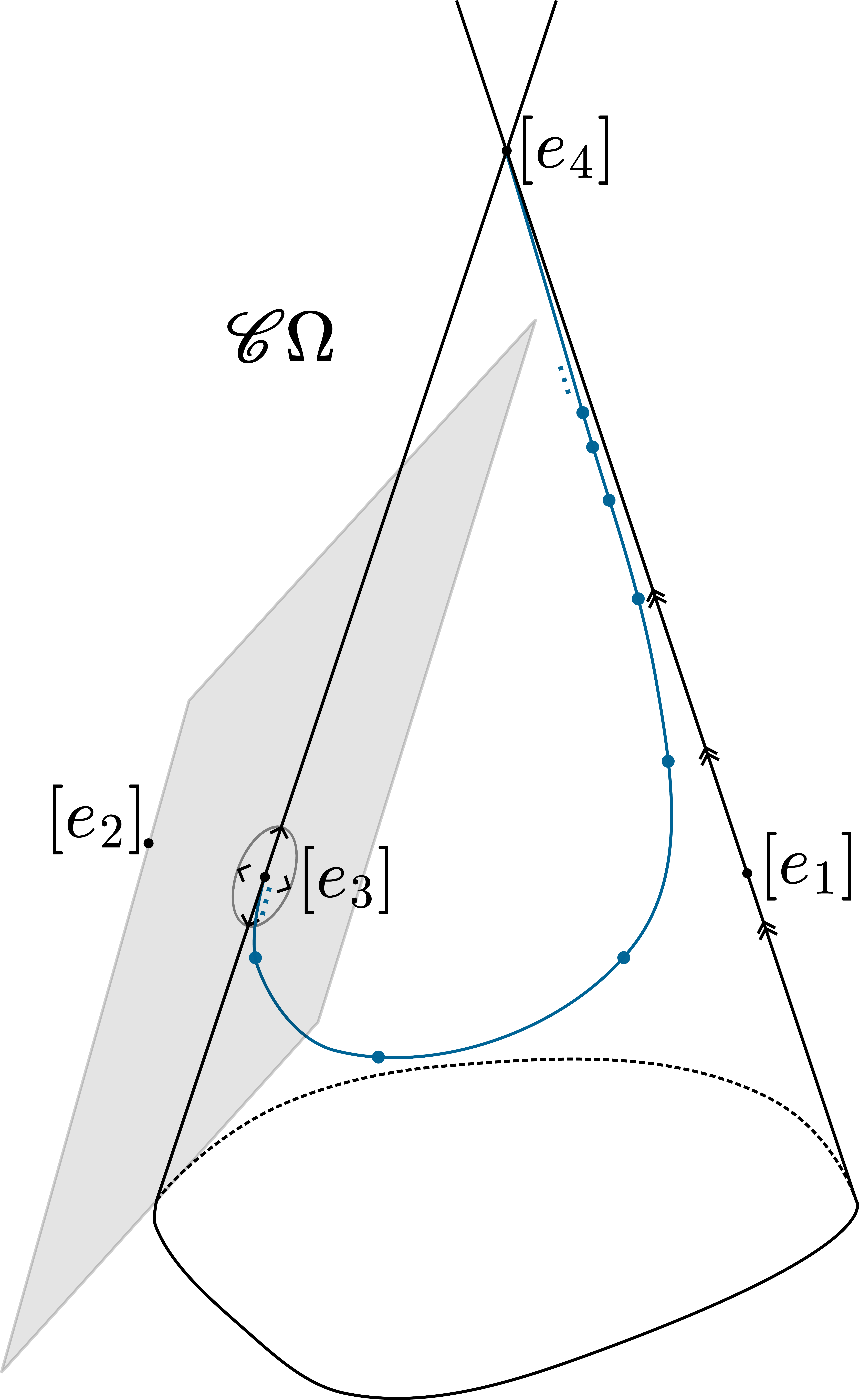}
    \caption{Mixed dynamics of the element $\eta(\delta)$ from Equation \ref{eqn: matrix unipotent intro}. The repelling plane is in grey, the line $Q\inverse(\rho(\delta)^+) = \langle [e_1],[e_4]\rangle$ on which the unipotent acts is on the right, and the orbit of a typical point in $\mathscr C \Omega$ is in blue. A basis compatible with the matrix representation is shown.}  
    \label{fig: mixed_dynamics}
\end{figure}

In particular, $Q\inverse(\rho(\delta)^+) = \langle [e_1],[e_4]\rangle$ in these coordinates, and the (1, 4)-block is a unipotent, where $\rho(\delta)^+$ is the attracting fixed point of $\rho(\delta)$ in $\partial \Omega \subset (\RP^d)$. The number $r$ is not an invariant quantity, but its sign is. That is, the unipotent block acts on the affine fiber of $Q$, $\langle [e_1],[e_4]\rangle \setminus [e_4]$ as a translation, and since there is a coherent orientation on these fibers, it makes sense to say that translation is positive or negative. 

Let $Z$ be the cocycle defining $\eta$ with linear part $\rho\chi$.  The useful coordinate-free formulation of the sign of $r$ (as a function of $Z$) is
\[
\operatorname{sign} (r) =\operatorname{sign}\left( Z(\delta)(\rho(\delta)^+_0) \right), 
\]
where the subscript $0$ is a chosen, continuous section $\overline \Omega \to \R^{d+1}\setminus \{0\}$. The sign of this linear functional is well-defined on cohomology (Lemma \ref{lemma: positivity lemma}).

Under hypotheses \ref{hyp A}, the positive cone $\mathscr P_{\rho\chi} \subset H^1(\Gamma,(\R^{d+1})^*_{\rho \chi})$ is exactly the intersection of the $k$ linear half-spaces defined in this way. It is clear that this is a convex open cone with at most $k$ faces.
For $\eta \in \mathscr P_{\rho\chi}$,  every element of $\eta(\Gamma)$ is either proximal, with attracting fixed point in $\partial \mathscr C \Omega$, or is conjugate to a transformation of the form \eqref{eqn: matrix unipotent intro} with positive unipotent translation (when $d>2$ the middle block may be more complicated). This is the intuition for the lower bound on the height of an invariant domain: the only novel dynamics in this case, compared to the convex cocompact case, is the presence of some \textit{positive} unipotent translations.
\medskip

Under hypothesis \ref{hyp B}, it is less clear that the story will come out this way. In dimension $2$, the structure theorem for geodesic laminations gives that $\uu \lambda$ has finitely many minimal components $\uu \mu_i$ and finitely many leaves spiraling between them. In the case that none of the $\uu \mu_i$ are closed, then \textit{every} element of $\eta(\Gamma)$ is proximal --- no element has a unipotent block, regardless of the cocycle $Z$.

The analogue of the positivity condition is to examine averages of the cocycle $Z$ along the leaves of $\uu \mu_i$ and collect the accumulation points of the analogous linear functionals. In this way, for each minimal component and invariant measure we obtain a Cantor set of accumulations --- one for each leaf of $\mu_i \subset \Omega$ that meets a small transversal arc (where $\mu_i$ is the full preimage of $\uu\mu_i$ in $\Omega$). This is the content of
Proposition \ref{prop: limit points of averages}, which occupies essentially all of \S \ref{sec: positive cone lamination}.

The positive cone in this case is defined as the simultaneous positivity of these linear conditions. This definition reduces to the previously defined positive cone under hypothesis \ref{hyp A} in dimension $2$, when $\uu \lambda$ is a multicurve.

\subsection{Heights bounded from below}

The first significant step in proving Main Theorem \ref{mainthm: structure of limit sets} is to show that the height of an invariant hypersurface in $\mathscr C \Omega$ is bounded below for $\eta \in \sP_{\rho\chi}$.
In spite of the differences between hypotheses \ref{hyp A} and \ref{hyp B}, we were able to unify the structure of the proof of bounded height for representations in the positive cone.

We first construct an $\eta$-equivariant hypersurface inside of $\sC \Omega$ in \S\ref{subsec: integrating forms and equivariant surfaces} by integrating a closed one-form  representing $Z$ valued in the flat bundle
\[E_{\rho\chi}^* \defeq \Omega\times_{\rho\chi}(\R^{d+1})^*.\]

The structure of the proof of the Bounded Height Theorems \ref{thm: heights bounded below scc} and \ref{thm: heights bounded below general}  is to examine the height achieved by rays in this invariant hypersurface. For instance, leaves of the lamination $ \lambda$ meeting a given compact set have height which increases at a uniform linear rate -- this is clear for closed leaves due to the invariance of the hypersurface.  This is the content of the Low Slack Lemmas \ref{lem: main low slack cor} and \ref{lem: low slack minimal}.

In contrast, the High Slack Lemma \ref{lemma: high-slack lemma} says that rays that accumulate slack at a definite linear rate have a uniformly bounded height. For example, a ray that stays at a definite distance from $\lambda$ (or crosses it at a definite angle) has this behavior.

Using the Partition Lemmas \ref{lemma: partition lemma multicurve case} and \ref{lem: partition lemma general}, we break an arbitrary ray into pieces that are ``low-slack" (rays that fellow travel with a leaf of $\lambda$) and ``high-slack'' segments, and combine quantitative estimates on the two behaviors to get uniform lower bounds on the height of every ray. This shows that our equivariant surface (hence $\Lambda_{orb}$) is bounded below.  This approach also yields useful upper and lower bounds on the height of every ray, throughout its entire life. We apply these estimates liberally.  These estimates constitute the heart of the paper. 

The main subtlety of the proof is that along a given low-slack segment, the height increases linearly relative to its arc length. However, the \textit{rate} of this linear increase \textit{decreases} exponentially relative to the total accumulated slack up to that point.

\subsection{Proximal and orbital limit sets coincide}

The Bounded Height Theorems \ref{thm: heights bounded below scc} and \ref{thm: heights bounded below general} and the estimates used to prove them permit, in \S \ref{sec: prox is orb}, another round of geometric analysis of invariant domains for representations in the positive cone. 

There is a minimal preserved convex subset of $\sC \Omega$:
\[
C_{min} = \hull \Lambda_{prox}
\]
(see Lemma \ref{lemma: minimal domain}).

For any point $o \in C_{min}$, the $\eta$-orbit of $o$ accumulates onto $\Lambda_{prox}$, and we therefore have inclusions
\[
\Lambda_{prox} \subset \Lambda_{o}\subset \partial^\infty C_{min}.
\]
There are two interesting points that we prove about these limit sets when $\eta$ is in the positive cone. 
\medskip

Firstly, that the orbital limit set is independent of $o \in \sC \Omega$. The first step is that this limit set is the same for a pair of points in $C_{min}$, which is basic convex geometry. Then, to upgrade this equality to all of $\sC \Omega$, we first show that $C_{min}$ is incident to every vertical line in $\mathscr C \Omega$, i.e.\ for every $p\in \mathscr C \Omega$, there exists $q\in C_{min}$ such that $p$ and $q$ share the same vertical line (Lemma \ref{lemma: C_min is a nontrivial fibration}). Then we consider sequences of height differences
\[
h(\eta(\gamma_n).p)-h(\eta(\gamma_n).q))
\]
for $\gamma_n\in\Gamma$ a divergent sequence.
The key observation is that in the case that the slack $\s[Q(q),\rho(\gamma_n).Q(q)]$ tends to infinity, the height differences above tend to zero, and the two sequences have a common accumulation in $\partial^\infty \sC \Omega$. In the case that the slack along these rays is bounded, the height of the ray tends to infinity, and hence both $\gamma_n.p$ and $\gamma_n.q$ tend to $[e_{d+2}]$, the global fixed point. This proves the required independence (Proposition \ref{prop: positive orbital sets coincide}).
\medskip

The second interesting point about limit sets is that for unexceptional $\eta$, $\Lambda_{prox} = \partial^\infty C_{min}$, i.e.\ it is ``as big as possible''. Due to the definition of $C_{min}$ as the hull of $\Lambda_{prox}$, this amounts to showing that if $x,y\in \Lambda_{prox}$ are distinct and are in the same vertical fiber of $\partial^\infty \sC \Omega$, then the segment $[x,y]$ is completely contained in  $\Lambda_{prox}$ as well.

In fact we prove a structural result (Proposition \ref{prop: found a ray}) that implies this `no gaps' condition:

\begin{proposition}
    For every unexceptional coaffine representation, either there is a vertical ray in $\Lambda_{prox}$  over every proximal limit point or there is a vertical line over every proximal limit point.  
\end{proposition}

To produce a vertical ray in $\Lambda_{prox}$ (a key step in the proof), we explicitly compute the height of the attracting fixed point of an $\eta$-proximal element (Lemma \ref{lem: compute height}).  Then we use a specific $\N^2$-indexed family of elements $\gamma^n\beta_m\in \Gamma$, where $\gamma$ is not in $\uu\lambda$ and $\beta_m$ follows a leaf of $\uu\lambda$ for time proportional to $m$.  
Then $\gamma^n$ provides exponential contraction, while $\beta_m$ supplies linear growth in the formula for height.   The analogy is that the set $\{e^{-n}m: n, m>0\}$ accumulates to $\R_{\ge0}$.

From this we obtain Main Theorem \ref{mainthm: equality of limit sets intro}, that for all unexceptional representations
\[
\Lambda_{prox} = \Lambda_{orb} = \partial^\infty C_{min}.
\]
Existence of the function
\[\mathcal H : \partial^\infty \Omega\to \R\cup \{\pm\infty\}\]
from Main Theorem \ref{mainthm: structure of limit sets} follows.
\medskip

Since the closed curves in $\uu\lambda$ are exactly the $\eta$-non-proximal ones, under hypotheses \ref{hyp B} where $\uu\lambda$ has no closed leaves, \emph{every} non-trivial element of $\Gamma$ is $\eta$-proximal.  This contributes significantly to the difficulty of the proof with hypotheses \ref{hyp B}.  However, the authors were surprised that the \emph{conclusions} of Main Theorem \ref{mainthm: structure of limit sets} coincided under the two sets of hypotheses --- rather, we expected to have interesting qualitative differences in the structure of $\Lambda$.  

\subsection{Proof of Main Corollary \ref{maincor: no CT map}}\label{subsec: proof of corollary}

We now show that no Cannon--Thurston map can exist.  Recall that lower-semi-continuity of $\mathcal H$ for $\eta \in \sP_{\rho\chi}$ gives: for any sequence $x_n \in \partial^\infty \Omega$ tending to a point $x \in \lambda^+$ (the positive endpoints of the leaves of $\lambda$), we have $\mathcal H(x_n) \to +\infty$.

    \begin{proof}
    First consider alternative 2 from Main Theorem \ref{mainthm: structure of limit sets}.
    For dynamical reasons, any equivariant map would have to send the attracting fixed point $\gamma^+\in \partial^\infty\Gamma$ of an $\eta$-proximal element to $\eta(\gamma)^+ \in \partial^\infty\sC\Omega$.  By the semi-continuity of $\mathcal H$, any sequence of $\eta$-proximal elements $\gamma_n \in \Gamma$ with $\gamma_n^+ \to x \in \lambda^+$, has $\eta(\gamma_n)^+ \to [e_{d+2}] \in \RP^{d+1}$ as $n\to \infty$.  So any such continuous equivariant map is constant on a dense set of $\partial^\infty\Gamma$, hence constant.  This is a contradiction, since there are plenty of proximal elements $\gamma\in \Gamma$ with $\eta(\gamma)^+\not=[e_{d+2}]$.
    The situation in alternative 3 is analogous.  
    
    The proof of alternative 1 from Main Theorem \ref{mainthm: structure of limit sets} also furnishes that any sequence $\gamma_n \in \Gamma$ with $\gamma_n^+ \to x \in \lambda^+$ has $\eta(\gamma_n)^+ \to [e_{d+2}]$, so we conclude in the same way (see Theorems \ref{thm: heights bounded below scc} and \ref{thm: heights bounded below general}). 
    \end{proof}

\subsection{Measure of the limit set}

For unexceptional $\eta$ not in the positive or negative cone, $\Lambda = \partial^\infty \sC\Omega$, which clearly has dimension $d$ under hypotheses \ref{hyp A} or \ref{hyp B}.

Under hypotheses \ref{hyp A}, for $\eta$ in the positive cone, we observe from the Bounded Heights Theorem \ref{thm: heights bounded below scc} that rays with unbounded height are relatively rare.  Indeed, we show that if $x \in \mathcal H\inverse (+\infty)$, then the corresponding ray $r_x$ must spend a very high proportion of its life near the leaves of $\lambda$, at least at infinitely many points in time (Proposition \ref{prop: average slack and height}).  
The intuition behind this is as follows: 
suppose $r_x: [0,\infty) \to \sC\Omega$ wants to have \emph{both} its height increase without bound \emph{and} make infinitely many deviations from leaves of $\lambda$.
The estimates obtained in the Bounded Heights Theorem \ref{thm: heights bounded below scc} then say that $r_x$ must fellow travel near a leaf of $\lambda$ for an \emph{exponentially long} period of time relative to the total slack accumulated by its excursions to make a unit increase of height.  

\begin{remark}
    The lower bound in the Bounded Heights Theorem \ref{thm: heights bounded below general} is not  sharp, and is precisely the barrier to obtaining Main Theorem \ref{mainthm: Lebesgue  intro} under hypotheses \ref{hyp B}.  The lower bound can be made sharp, but requires  tedious bookkeeping.  
\end{remark}

A dynamical argument using the Sinai--Ruelle--Bowen measure for the (Anosov) geodesic flow on $T^1M$ for the domain shape metric allows us to conclude that the behavior of such rays is indeed null for the Lebesgue measure class on $\partial^\infty \Omega$.  Applying Fubini's Theorem then gives the desired result.
\medskip

To reduce the question of positivity of Lebesgue measure of $\Lambda$ to nullity of $\mathcal H\inverse (+\infty)$, we apply a pleasing ``zooming in'' argument (Lemmas \ref{lem: properties of the union} and \ref{lem: measure of union}) to the point $[e_{d+2}]$.  Indeed, we use the global dilation 
\[I_t = \diag(I_{d+1}, e^{-t})\]
of every vertical ray and a nesting property 
to show that $\Lambda$ has positive Lebesgue measure if and only if $\mathcal H\inverse (\R) \times [0,\infty)$ has positive Lebesgue measure.

\section{Preliminaries}\label{sec: background}

We now collect some basic background material.

\subsection{Vector spaces and projective space}\label{subsec: vector spaces}

We work with finite-dimensional vector spaces over $\R$ throughout. 

\subsubsection{Vector spaces}
For a vector space $V$, the dual space is denoted $V^*$.
For $A\in \GL(V)$ and $\alpha\in V^*$, we denote by 
\[A.\alpha = \alpha \circ A\inverse\]
the unique action compatible with the evaluation pairing $V \times V^* \to \R$.

Given a representation $r:\Gamma \rightarrow \GL(V)$, there is thus an action of $\Gamma$ on both $V$ and $V^*$. We denote these vector spaces with their $\Gamma$-actions as $V_r$ and $(V^*)_r$ respectively.

\subsubsection{Projective space}

We define the \textbf{real projective space} $\RP^d = \P(\R^{d+1})$ as usual to be the quotient of $\R^{d+1}\setminus \{0\}$ by the center of $\GL(d+1,\R)$ (that is, scalar transformations). The dual projective space is $(\RP^d)^* = \P(\R^{d+1})^*$. For $v \in \R^{d+1}\setminus \{0\}$, we write its projective class $[v]$, and more generally $[W]$ for the image of any vector subspace of $\R^{d+1}$ in $\RP^d$.

There is a correspondence between the points of the dual projective space $(\RP^d)^*$ and the projective hyperplanes in $\RP^d$. Namely for each $[\alpha] \in (\RP^d)^*$, the hyperplane $[\ker \alpha]$ consists of the projective class of every vector on which $\alpha$ vanishes. The hyperplane and the point in the dual space identify one another uniquely. The natural action of $\GL(d+1,\R)$ described above is compatible with this identification. This is one aspect of what is generally referred to as projective duality-- the hyperplanes in a projective space may be viewed as points in a dual projective space. 

An element of the general linear group $A \in \GL(V)$ acts on $\P(V)$ (and its dual) via
\[
A.[v] = [Av], \qquad  A.[\alpha] = [A.\alpha] = [\alpha\circ A\inverse]
\]
which factors through the projective general linear group $\PGL(V).$

\subsubsection{Singular values and eigenvalues}\label{subsub: sing and eig}

Choose an inner product on $V$ with associated norm $\|\cdot\|$, and let $A \in \GL(V)$.  Let $K < \GL(V)$ be the subgroup preserving this inner product.  The \textbf{singular values}
\[0< e^{\sigma_{d+1}(A)}\le \ldots \le e^{\sigma_1(A)}\]
of $A$ with respect to $\|\cdot \|$ are the square roots of the eigenvalues of the positive definite symmetric matrix $A^t A$.  It is standard in linear algebra that there are elements $k_1, k_2 \in K$ such that 
\[A = k_2 \diag (e^{\sigma_1(A)}\dots e^{\sigma_{d+1}(A)}) k_1,\]
where the diagonal matrix is given in terms of an ordered orthonormal basis $e_1, \dots, e_{d+1}$ for the $(d+1)$-dimensional vector space $V$.  
In particular, $\|A\| = e^{\sigma_1(A)}$ and $\|A\inverse\| = e^{-\sigma_{d+1}(A)}$. 
\medskip

For $A\in \GL(V)$ with $\sigma_1(A)>\sigma_2(A)$, define $U(A)\le V$ to be the line spanned by $k_2e_1$.  This is the major axis of the ellipsoid $\{Av: \|v\|=1\} \subset V$.  It is also the direction in $V$ whose norm decreases by the largest proportion after applying the map $A\inverse$.  
Also define $S(A)$ to be the hyperplane spanned by $k_1\inverse \langle e_2, ..., e_{d+1}\rangle$.  This is the hyperplane in $V$ for which the norm of any element under $A$ changes by at most $e^{\sigma_2(A)}$.  Compare \cite[\S2.2]{BPS19}.

\begin{figure}[h]
    \centering
    \includegraphics[width=.6\linewidth]{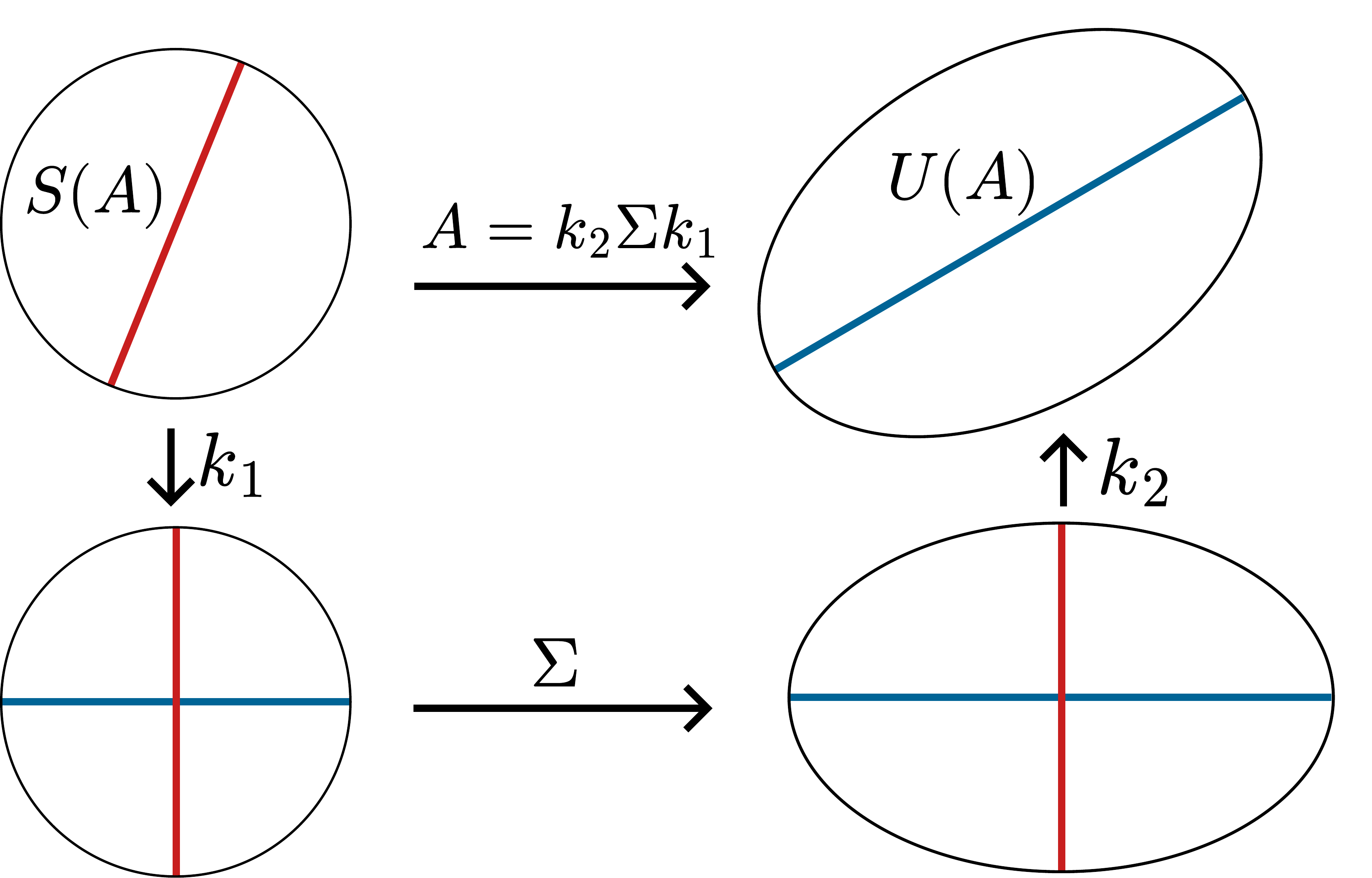}
    \caption{For a proximal element $A$, the stable and unstable subspaces are shown in red and blue in the top of the figure. In the bottom left, the blue axis is the span of the first basis vector, and the red is its orthogonal complement.}  
    \label{fig: S and U subspaces}
\end{figure}

The moduli of the eigenvalues of $A$ are denoted $\{e^{\lambda_i(A)}\}$ and are listed in non-increasing order, i.e., $e^{\lambda_{d+1}(A)}\le ... \le  e^{\lambda_{1}(A)}$.  Recall first that an element $A\in \GL(V)$ is \textbf{proximal} if $\lambda_1(A)>\lambda_2(A)$.

\subsection{Coaffine space}\label{subsec: coaffine space}
In this paper, we will be interested in coaffine structures fibering over $d$-manifolds with $1$-dimensional fiber.

The \textbf{(general) coaffine group} is\footnote{This is actually the index-two subgroup preserving the orientations on foliation by vertical lines.}
\[
\cag =
\left\{
\begin{pmatrix}
    A & 0 \\
    \tau & 1 
\end{pmatrix}
\colon
    A \in \GL_+(d+1,\R),
    \tau \in (\R^{d+1})^*
\right\}<\GL_+(d+2,\R)
\]
which is by definition the group of orientation preserving automorphisms of 
the \textbf{coaffine space}
\[\cas \defeq \left(\RP^{d+1}\setminus [e_{d+2}]\right)\subset \RP^{d+1}.\]
We define the \textbf{linear part}
\[
L
\begin{pmatrix}
    A & 0 \\
    \tau & 1 
\end{pmatrix}
=
A
\]
and the \textbf{cotranslation part} to be $\tau$. The kernel of $L$ is the subgroup of cotranslations. 

The vector space quotient
\[
\hat Q : \R^{d+2}\to  \R^{d+2} / [e_{d+2}]\cong \R^{d+1}
\]
obtained by forgetting the last coordinate passes to a projection 
\[
\begin{aligned}
Q : \cas &\to \RP^{d} \\
\begin{bmatrix}
v\\
1
\end{bmatrix}
&\mapsto
\begin{bmatrix}
v
\end{bmatrix}
\end{aligned}
\]
which interacts with $\cag$ via the linear part
\[
Q\circ
g
=
L(g) \circ Q.
\]

That is, the coaffine group acts on the quotient by $Q$ by the linear part. The linear part's action on $\RP^d$ factors through $\PGL_+(d+1,\R)$, which we omit from the notation.

Geometrically the image of $Q$ is the projective space formed by the \textbf{vertical lines}, the one-dimensional projective subspaces of $\RP^{d+1}$ which are incident to $[e_{d+2}]$.

We identify 
\[
\im \hat Q = \R^{d+1} \textrm{ and } \im Q = \RP^{d}.
\]

We choose to view the group $\cag$ as a subgroup of $\GL(d+2,\R)$ as it makes the eigenvalue decomposition straightforward. Note that its action on $\cas$ factors through its image in $\PGL(d+2,\R)$, but that it is also isomorphic to its image.

\subsection{Convex and divisible domains in projective space}\label{subsec: convex and divisibles}

\subsubsection{Convex sets}
We review some of the essential background about convex sets in projective space. For a general reference see \cite{Rockafellar}.

An \textbf{affine chart} in the projective space $\RP^d$ is the complement of a hyperplane 
\[
\A_\alpha = \RP^d \setminus [\ker \alpha]
\]
where $[\alpha] \in (\RP^d)^*$. Each affine chart is isomorphic to the standard, real, $d$-dimensional affine space and hence has a notion of convexity, i.e., for every pair of points $p,q \in \A_\alpha$ there exists a unique projective segment
\[
[p,q]\subset \A_\alpha
\]
with endpoints $p$ and $q$.

A set $X$ in real projective space $\RP^d$ is \textbf{convex} if there exists an affine chart $\A_\alpha$ such that $X \subset \A_\alpha$ and $X$ is convex in $\A_\alpha$. It is \textbf{properly convex} if further there is $\A_\alpha$ such that the closure $\overline X\subset A_\alpha$. 
The definition of (proper) convexity is independent of the chart, so long as $\A_\alpha$ contains $X$ (or $\overline X$). 
\medskip
 
Every closed convex set $X\subset \RP^d$ decomposes into a partition of open faces. An \textbf{open face} of a convex subset $X$ is a class of the equivalence relation
$x\sim y \iff$ there exists an open segment $(x',y')\subset X$ such that $x,y \in (x',y').$ In short, two points are in the same face if the segment between them is extensible in $X$ in both directions. The closed faces (or just \textbf{faces} for brevity) of $X$ are the closures of the open faces. The (closed) faces form a poset under inclusion. If $X$ is not closed, we define its faces to be the faces of its closure, $\overline X$. Being convex sets themselves, faces have an intrinsic dimension, which is the dimension of the projective subspace which they span under arbitrary convex combination.

This definition of a face is desirable because, firstly, every face $K\subset X$ is convex, and furthermore if $L\subset K\subset X$, and $L$ is a face of $K$ and $K$ is a face of $X$, then $L$ is a face of $X$ (being a face is transitive).

We will also use the following elementary fact later in \S \ref{sec: prox is orb}. As usual, a convex hull is defined for us either as an intersection of half-spaces containing a set, or as a union of points obtained via finite convex combination.

\begin{lemma}[\cite{Rockafellar}]\label{lem: 0-faces of convex sets}
    Let $S$ be a closed subset of affine space $\A$, and let $\hull (S)$ be its convex hull. Then the set of $0$-dimensional faces of $\hull(S)$ is a subset of $S$.
\end{lemma}

\subsubsection{Divisible convex sets}

For a more complete overview of the rich landscape of divisible domains, we direct the reader to a pair of nice survey papers, \cite{Benoist_survey} and \cite{Marquis_survey}.

If $X\subset \RP^d$ is a properly convex set, then its (projective) automorphism group is 
\[
\Aut(X) \defeq \{g\in \PGL(d+1,\R) \colon g.X = X.\}
\]
Generically, $\Aut(X) = 1$. If $X$ is a round ball, then 
$\Aut(X) \cong \mathrm{PO}(d,1)$; this is the projective (Klein) model of the real hyperbolic space.

Benoist, following Benz\'ecri, studied properly convex subsets of $\RP^d$ admitting compact manifold (or orbifold) quotients.
\begin{definition}[Divisible convex domain, \cite{Benoist_CDI}]\label{def: convex divisibles}
    An open properly convex subset $\Omega\subset \RP^d$ is a \textbf{convex divisible domain} if there exists a discrete subgroup
    $\Gamma <\Aut(\Omega)$ such that the quotient
    $\Gamma \backslash \Omega$ is compact.
    In this case, $\Gamma$ \textbf{divides} $\Omega$.
\end{definition}

The fact that $\Gamma \backslash \Omega$ may be assumed to be a manifold (or orbifold) follows from the fact that properly convex sets are metrizable by an $\Aut(\Omega)$-invariant metric, in particular via the Hilbert metric. See also \S\ref{sec: domain shape}.

Benoist also gives the following result. A convex set is $C^1$ if every point in the boundary has a unique supporting hyperplane, and it is strictly convex if its boundary contains no segments. 

\begin{lemma}[\cite{Benoist_CDI}]\label{lem: strictly con and C1 is hyperbolic}
    A divisible domain $\Omega \subset \RP^d$ has $C^1$ and strictly convex boundary if and only if every group $\Gamma$ dividing $\Omega$ is Gromov hyperbolic.
\end{lemma}

Indeed, in this case the Hilbert metric is hyperbolic. 
\medskip

In dimension $2$, the story of strictly and properly convex divisible domains is particularly nice. Dividing groups come from Hitchin representations of surface groups \cite{Choi-Goldman}, that is, $\Gamma$ dividing $\Omega \subset \RP^2$ is a deformation of a discrete surface group in $\SO(2,1)$.

The following fact is useful when working with convex divisible domains.

\begin{lemma}[\cite{Benoist_CDI}]\label{lem: lift to SL}
    If $\Gamma < \PGL(d+1,\R)$ divides $\Omega\subset \RP^d$, a properly convex set, then, up to finite index, $\Gamma$ is torsion-free and admits a lift to $\SL(d+1,\R)$.
\end{lemma}

Note that the removal of torsion and the lift are different claims, the former follows immediately from Selberg's Lemma, the latter is less instantaneous.

\subsection{Projective Anosov representations}\label{subsec: projective Anosov}

Anosov representations were introduced by Labourie \cite{Lab06} and developed for general hyperbolic groups by Guichard--Wienhard \cite{GW12}.
We follow the more recent approaches, due to Kapovich--Leeb--Porti \cite{KLP17,KLP18a, KLP18b},  Guéritaud–Guichard–Kassel–Wienhard \cite{GGKW17}, and Bochi--Potrie--Sambarino \cite{BPS19}, which have the advantage of having simple, quantitative formulations that also have direct consequences for the action of an Anosov subgroup/representation on suitable flag manifolds.  Thus our exposition is ahistorical.  

In this paper, we will only be interested in \emph{projective-Anosov} (or \emph{$1$-Anosov}) subgroups of  $\GL(d+1,\R)$ (or equivalently, the corresponding faithful representations) in terms of singular values (with respect to an inner product on $\R^{d+1}$).

\begin{definition}[Projective Anosov representation]\label{def: Anosov rep}
    Let $\Gamma$ be a finitely generated group, and choose a finite symmetric generating set giving rise to a word metric $| \cdot |$ on $\Gamma$.  A representation $\rho : \Gamma \to \GL(d+1,\R)$ is \textbf{projective Anosov} if there are constants $a, b >0$ such that
    \[\sigma_1(\rho(\gamma)) - \sigma_2(\rho(\gamma))\ge a|\gamma| -b,\]
    for all $\gamma \in \Gamma$.  
\end{definition}

A projective Anosov representation comes with a useful package of information describing the the dynamics of the action on projective space and the dual projective space.   The relationship between the following Theorem and Definition \ref{def: Anosov rep} is contained in the papers cited above. 

\begin{theorem}\label{thm: anosov limit curve}
    Let $\rho : \Gamma \to \GL(d+1,\R)$ be a projective Anosov representation.  Then $\Gamma$ is Gromov hyperbolic, the kernel of $\rho$ is finite, and there are continuous, $\rho$-equivariant maps $\xi: \partial^\infty \Gamma \to \RP^d$ and $\xi^*: \partial^\infty \Gamma\to (\RP^d)^*$ satisfying
    \begin{itemize}
        \item (compatibility) for all $x \in \partial^\infty \Gamma$, $\xi(x)\le \ker \xi^*(x)$
        \item (transversality) for all $x \not=y \in \partial^\infty \Gamma$, we have $\xi(x) \not\in \ker \xi^*(y)$
        \item (dynamics preserving) for any infinite order element $\gamma \in \Gamma$ with attracting fixed point $\gamma_+ \in \partial ^\infty \Gamma$, $\xi(\gamma_+)$ and $\xi^*(\gamma_+)$ are the attracting fixed points of the natural action of $\rho(\gamma)$ on $\RP^d$ and $(\RP^d)^*$, respectively.  
    \end{itemize}
    Moreover, the (projective) Anosov property is an open condition in $\Hom(\Gamma, \GL(d+1,\R))$.  
\end{theorem}

Another characterization of the projective Anosov property for Gromov hyperbolic groups in terms of a uniform eigenvalue gap, due to Kassel--Potrie \cite{KP:eigen}, is given at the beginning of \S\ref{sec: failure of anosov}.

Denote by $\sigma_i^\rho(\gamma) \defeq \sigma_i(\rho(\gamma))$ and $\lambda_i^\rho(\gamma) \defeq \lambda_i(\rho(\gamma))$ the logarithms of the $i$th largest singular value and eigenvalue (\S\ref{subsub: sing and eig}), respectively.  By Definition \ref{def: Anosov rep}, 
for any element $\gamma \in \Gamma$ of sufficiently large word length, $\sigma_1^\rho(\gamma) >\sigma_2 ^\rho(\gamma)$.  So, for all but finitely many elements $\gamma$, the major axis $U(\rho(\gamma))$ of the ellipsoid $\{\rho(\gamma)v: \|v\|=1\}$ is well defined.

Consider the complete metric  $d_{\mathbb P}([v],[w]) =  |\sin \angle (v,w)|$ on projective space.

The boundary maps $\xi$ and $\xi^*$ for an Anosov representation $\rho : \Gamma \to \GL(d,\R)$ can be defined as a limit of the stable subspaces for elements of $\Gamma$.

\begin{lemma}[{\cite[Lemma 4.7]{BPS19}}]\label{lem: uniform convergence to boundary map}
    Given a $1$-Anosov representation $\rho: \Gamma \to \GL(d,\R)$ and constants $L\ge 1$ and $C\ge 0$, there are positive $c$ and $d$ such that the following holds.  Suppose $\gamma \in \Gamma$ is large enough for $U(\rho(\gamma))$ and $S(\rho(\gamma))$ to be defined and is on a $(L,C)$-quasi-geodesic ray from the identity $id$ to $x \in \partial^\infty \Gamma$.  Then
    \[d_{\mathbb P} (U(\rho(\gamma)), \xi(x)) \le de^{-c|\gamma|}.\]
    Furthermore, if $\gamma\inverse$ is on a $(L,C)$-quasi-geodesic ray from the identity to $y\in \partial^\infty \Gamma$, then 
    \[d_{\mathbb P}(S(\rho(\gamma)), \xi^*(y))\le de^{-c|\gamma|}.\]
\end{lemma}
One very weak but useful consequence of Lemma \ref{lem: uniform convergence to boundary map} is that $\im \xi$ is well approximated by the family $\{U(\rho(\gamma)): \gamma \in \Gamma\}$ in the following sense. 
For every $\epsilon>0$, there is an $N$ such that if $|\gamma|>N$, then 
\begin{equation}\label{eqn: uniform convergence of xi}
    d_{\mathbb P}(U(\rho(\gamma)), \im \xi)
    <\epsilon
    \text{ and } 
    d_{\mathbb P }(S(\rho(\gamma)), \im \xi^*)
    <\epsilon.  
\end{equation}

\medskip

The following states that $\rho(\gamma)$ is, up to scale, very close to a projection.  It is quite easily deduced from the definitions and the fact that moduli of the eigenvalues and singular values are comparable for elements $\gamma$ whose ideal endpoints vary in compact sets of $\partial^\infty\Gamma\times \partial^\infty\Gamma\setminus \Delta$.  Compare with, e.g.,  \cite[\S2.1]{BCLS}.

\begin{lemma}\label{lem: rank one projectors and Anosov}
    Let $\rho: \Gamma \to \GL(d,\R)$ be $1$-Anosov and suppose $U, V \subset \partial^\infty\Gamma$ are such that $\overline U \cap \overline V = \emptyset$. There are $c,d >0$ such that the following holds.  For every infinite order $\gamma \in \Gamma$ with $\gamma_+ \in U$ and $\gamma_- \in V$, consider the rank one projection $P_\gamma$ to $\xi(\gamma_+)$ parallel to $\xi^*(\gamma_-)$.  Then 
    \[ \left\|\frac{1}{e^{\lambda_1^\rho(\gamma)}}\rho(\gamma) -P_\gamma\right\|\le de^{-c|\gamma|}.\]
\end{lemma}

\subsubsection{Strong convex cocompactness}\label{subsub: strong cc}

Following Danciger--Gu\'eritaud--Kassel \cite{DGK}, we have:

\begin{definition}\label{def: convex cocompact}
    Say that a faithful representation $\rho: \Gamma \to \GL(d+1,\R)$ with infinite discrete image is \textbf{strongly convex cocompact} if there is a $\rho$-invariant properly convex open set $\mathcal D \subset \RP^d$ whose boundary $\partial \mathcal D$ is of class $C^1$ and strictly convex and a point $x \in \mathcal D$ such that the action of $\rho$ on the convex hull of $\overline{\rho(\Gamma).x}\cap \mathcal D$ is cocompact. 
\end{definition} 

They prove, among other things, that strong convex cocompactness is an open condition on $\Hom(\Gamma, \GL(d+1,\R)$ and  provide the following elegant characterization when $\Gamma$ is Gromov hyperbolic.

\begin{theorem}[{\cite[Theorem 1.4]{DGK}}]
    Suppose $\rho : \Gamma \to \GL(d+1,\R)$ is an injective representation of a non-elementary hyperbolic group with discrete image that preserves a properly convex open subset of $\RP^d$.  Then $\rho$ is strongly convex cocompact if and only if $\rho$ is projective Anosov.
\end{theorem}

Although it follows from this theorem, it can be deduced from work of Benoist \cite{Benoist_CDI} that the natural inclusion of a group dividing a properly convex domain with $C^1$ boundary in $\RP^d$ is projective  Anosov.  
Similarly,  if $\rho: \Gamma \to \GL(d+1,\R)$ is $1$-Anosov, then every element $\gamma \in \Gamma$ of infinite order has that the action of $\rho(\gamma)$ on $\RP^d$ is proximal.

\subsection{Geodesic laminations and currents}\label{subsec: geodesic laminations}
\subsubsection{Basic definitions}
Let $M$ be a closed $d$-manifold equipped with a sufficiently regular Finsler metric so that every unit tangent vector $v\in T^1M$ is tangent to a unique complete geodesic in $M$.  Then the geodesic flow $g_t : T^1M \to T^1M$ is defined.\footnote{The \emph{domain shape} metric of \S\ref{sec: domain shape} is such a Finsler metric.}  

\begin{definition}[Geodesic lamination]
    A \textbf{geodesic lamination} $\lambda \subset M$ is a closed set foliated by complete geodesics called its \textbf{leaves}. A sublamination $\lambda'\subset \lambda$ is a closed union of leaves that inherits its foliation structure from $\lambda$. 
\end{definition}

Geodesic laminations play an important role in the theory of surfaces, where they enjoy a good structure theory, but the definition makes sense in any dimension $d\ge 2$.

\begin{remark}
    In dimension $2$, the condition that two complete geodesics do not intersect is quite restrictive, while in higher dimensions $d \ge 3$, generic geodesics do not intersect.  Thus, we do not expect geodesic laminations in dimension $d\ge 3$ to have a good structure theory.  
\end{remark}

A geodesic lamination $\lambda$ is \textbf{connected} if it is connected as a subspace of $M$, and $\lambda$ is called \textbf{orientable} if its leaves admit a continuous global orientation.
If $\lambda$ is connected, another way to say this is that the set $T^1\lambda \subset T^1M$ of unit vectors tangent to the leaves of $\lambda$ has two connected components if and only if $\lambda$ is orientable.  In this case, a choice of orientation determines a canonical map $\lambda \to T^1M$ which is a homeomorphism onto its image of points $T^1_+\lambda$ that are forward tangent to leaves of $\lambda$.  Note that the geodesic flow preserves $T^1\lambda$ and $T^1_+\lambda$ if $\lambda$ is oriented.

\subsubsection{Chain recurrence}
In the context of geodesic laminations, the following definition of chain recurrence was adapted from Conley by Thurston. For this, choose a metric on $T^1M$ that generates the topology.  

We do not use this somewhat technical definition directly in the body of the text, so the uninterested reader can safely skip it, if desired.  

\begin{definition}[Chain recurrence]
    A geodesic lamination is called \textbf{chain recurrent} if for every $\epsilon>0$ and every point $x \in T^1\lambda$, there is a closed $\epsilon$-trajectory of $\lambda$ through $x$, i.e., a closed oriented loop $\gamma$ in $M$ such that for every length one segment in $\gamma$, there is a length one segment of an oriented leaf of $\lambda$ such that the canonical lifts are $\epsilon$-close in $T^1M$.  
\end{definition}

Chain recurrence is a closed condition on geodesic laminations (in any dimension).
Suppose also the geodesic flow $g_t:T^1M\to T^1M$ is Anosov.  In dimension $2$, any chain recurrent geodesic lamination can be approximated by a simple multicurve in the Hausdorff topology on closed sets \cite{Thurston:stretch}. 
A chain recurrent geodesic lamination in any dimension can be approximated by multicurves (perhaps not simple ones) in the Hausdorff topology---this is essentially a consequence of the definition and the Anosov Closing Lemma. 

\subsubsection{The measure topology}
Since $M$ is compact, any closed flow invariant subset supports a flow invariant probability Borel measure, and the space of such measures is compact in the weak-$*$ topology.  For a closed oriented geodesic $\gamma \subset M$, the canonical lift $T^1_+\gamma$ is the support of a unique flow invariant probability measure, $\delta_\gamma$ on $T^1M$.  

Say that a sequence of oriented closed geodesics $\gamma_n$ accumulates to $\lambda$ \emph{in the measure topology} if all weak-$*$ accumulation points of $\{\delta_{\gamma_n}\}$ in the space of probability measures on $T^1M$ have support contained in $T^1\lambda$.  If $\lambda$ is oriented, it also makes sense to ask that all weak-$*$ accumulation points of $\{\delta_{\gamma_n}\}$ have support contained in $T_+^1\lambda$.  
Compare with Bonahon's theory of geodesic currents \cite{Bonahon:ends}.

\section{Basic geometry of coaffine representations}\label{section: geometry of coaffine domain}

We  give the data of a coaffine representation and define and study some basic properties of associated limit sets and convex invariant domains.  

\subsection{Coaffine representations}\label{subsec: coaffine reps}

The coaffine representations that we study in this paper are of the following specific form. 
We need three kinds of data:

\begin{enumerate}
    \item Let $\rho: \Gamma \to \SL(d+1,\R)$ be an injective representation of a torsion-free group dividing a properly and strictly convex domain $\Omega \subset \RP^d = Q(\cas)$.   The quotient $M \defeq \Omega/\rho$ is a manifold.\footnote{this does not further restrict the choice of $\Omega$ following Lemma \ref{lem: lift to SL}.} 

    \item Let $\varphi \in \Hom(\Gamma, \R)$ and $\chi \defeq e^{-\varphi}$.  
    \item Let $Z$ be a \textbf{$\rho\chi$-twisted cocycle}, i.e., a function $Z: \Gamma \to (\R^{d+1})^*$ that satisfies
\[Z(\gamma_1 \gamma_2) 
=
Z(\gamma_1)\rho(\gamma_2)\chi(\gamma_2) + Z(\gamma_2), \forall \gamma_1,\gamma_2 \in\Gamma.\]
We let $Z^1(\Gamma, (\R^{d+1})^*_{\rho\chi})$ denote the vector space of such $\rho\chi$-twisted cocycles.
\end{enumerate}

Notationally, $\rho\chi: \Gamma \to \GL(d+1,\R)$
is defined by $\rho\chi(\gamma) = \rho(\gamma)\chi(\gamma)$, the second factor lying in the center of $\GL(d+1,\R)$.

Finally, given all of this chosen data, let 
\[
\eta = \eta(\rho,\chi,Z):\Gamma \to \cag
\]
be the coaffine representation of $\Gamma$ defined by
\begin{equation}\label{eqn: eta}
    \eta(\gamma) = 
\begin{pmatrix}
    \rho(\gamma)\chi(\gamma) & \textbf{0}\\
    Z(\gamma) & 1
\end{pmatrix}
\end{equation}
the final column being filled with $0$'s with the exception of the diagonal $1$.

\begin{tcolorbox}[
    colback=blue!5,
    colframe=black!50,
    boxrule=0.5pt,
    arc=2pt
]
In short, we choose:
\begin{itemize}
    \item $\rho:\Gamma\to \SL(d+1,\R)$ a representation dividing a properly and strictly convex domain $\Omega$.
    \item $\chi=e^{-\varphi}$ a character of $\Gamma$; $\varphi\in H^1(M,\R)$.
    \item $Z\in Z^1(\Gamma,(\R^{d+1})^*_{\rho\chi})$ a twisted cocycle.
\end{itemize}
Then  $\eta=\eta(\rho,\chi,Z)$ is the representation of $\Gamma$ to $\cag$ as in Equation \ref{eqn: eta}.
\end{tcolorbox}

A cocycle $Z$ is called a \textbf{coboundary} if there exists $\alpha \in (\R^{d+1})^* $ such that $Z(\gamma) = \alpha \circ (\rho\chi(\gamma)-I)$. The space of coboundaries is written
$
B^1(\Gamma, (\R^{d+1})^*_{\rho\chi})
$, and it is a vector subspace of the cocycles; the quotient is the \textbf{twisted cohomology }
\[
H^1(\Gamma, (\R^{d+1})_{\rho\chi}^*) \defeq 
Z^1(\Gamma, (\R^{d+1})^*_{\rho\chi}) / 
B^1(\Gamma, (\R^{d+1})^*_{\rho\chi}))
\]

Conjugation of $\eta$ by a cotranslation affects the representation by adding a coboundary to $Z$.
Thus $\eta$ preserves a hyperplane $\ker(\alpha)$ which does not intersect $[e_{d+2}]$ if and only if $Z$ is a coboundary. See \S 2.2 of \cite{BF} for details.

We will say that $\eta$ not preserving any hyperplane (equivalently, defined with $Z$ not a coboundary) is \textbf{non-semi-simple}.

We remark that our convention for the cocycle property differs slightly from the standard definition. The resulting spaces are isomorphic, and this condition is the one that makes $\eta$, as defined, a representation.

\subsection{The cone $\mathscr C \Omega$}

Under the assumption that $\rho$ preserves a properly convex domain $\Omega \subset \RP^d$, the geometry of $\cas$ becomes consequently richer as a fiber bundle over $\RP^d$.

Let 
\[
\mathscr C \Omega = Q\inverse (\Omega) \subset \cas.
\]

The domain $\mathscr C \Omega \subset \cas$ is the main geometric setting of the paper.
We direct the reader to \cite{BF} for the essential properties of $\mathscr C \Omega$ and their proofs, but list the most salient here
\begin{lemma}\label{lem: cone basic props}
    The domain $\mathscr C \Omega$ enjoys the following properties
    \begin{enumerate}
        \item It is an affine line bundle over $\Omega$ via the projection $Q$; the fibers are \textbf{vertical lines}
        \item It is convex: for all $p,q\in \mathscr C \Omega$ distinct, there exists a unique projective segment $[p,q]\subset \mathscr C\Omega$ between them.
    \end{enumerate}
\end{lemma}

The proofs of the claimed properties are not difficult.

\subsection{Limit sets and invariant domains in $\mathscr C \Omega$}

Here we explore some of the essential geometry of coaffine representations and the domains on which they act, given in general what we know about projective and coaffine geometry, convex domains, and the definitions.

 Recall that an element $g\in \GL_+(d+2,\R)$ is \textbf{proximal} if its action on $\RP^{d+1}$ has a unique attracting fixed point $\gamma_+$. It is standard that for every non-identity element $\gamma\in \Gamma\le \GL_+(d+2,\R)$, the transformation $\rho(\gamma)$ is proximal when $\rho$ divides a strictly convex projective domain, such as $\Omega$ \cite{Benoist_CDI}.  As usual, we define the \textbf{proximal limit set}
\[\Lambda_{prox}= \overline {\{\eta(\gamma)^+\in\RP^{d+1} \colon \gamma\in \Gamma \textrm{ $\eta$-proximal}\}},\]
where the closure is taken in $\cas$.  $\Lambda_{prox}$ is invariant.

\begin{remark}
    A common definition of the proximal limit set is that it is the minimal closed invariant subset of a representation acting on $\RP^{d+1}$. In this setting, that would be the singleton $[e_{d+2}]$. If one instead considers the action of $\eta$ on $\cas$, then one has the desired property: $\Lambda_{prox}$ is the minimal closed subset under the action of $\eta$.
\end{remark}
In the setting at hand, it is not hard to see that
\[
Q(\eta(\gamma)_+) = \rho(\gamma)_+
\]
for all $\gamma$ which are $\eta$-proximal.
Therefore,
\[\Lambda_{prox}\subset \partial \mathscr C \Omega.\]

We also consider limit sets obtained from orbits of points. For a point $o\in \mathscr C \Omega$, define
\[\Lambda_{orb}^o = \overline{\eta(\Gamma). o } \cap \partial \mathscr C \Omega. \]
In the case of convex cocompact actions of hyperbolic groups (for example, when $\eta$ is Anosov) this set is independent of the point $o$ \cite{DGK}.
\medskip

The containment of the proximal limit set in every orbital limit set is straightforward:

\begin{lemma}\label{lem: proximal contained in orbital}
    For every $o\in \mathscr C\Omega$,
    \[
        \Lambda_{prox}\subset \Lambda_{orb}^o.
    \]
\end{lemma}

\begin{proof}
    Let $\gamma\in\Gamma$ be such that $\eta(\gamma)$ is proximal, and
    let $\eta(\gamma)^+$ denote its attracting fixed point. The repelling hyperplane of $\eta(\gamma)$ is
    disjoint from $\mathscr C\Omega$. Consequently,
    \[
        \eta(\gamma)^n.o\longrightarrow \eta(\gamma)_+
    \]
    for every $o\in\mathscr C\Omega$. Thus every attracting fixed point
    of a proximal element belongs to $\Lambda_{orb}^o$.

    Since $\Lambda_{orb}^o$ is closed, it contains the closure of the
    set of attracting fixed points of proximal elements, namely
    $\Lambda_{prox}$.
\end{proof}

The following is standard in the setting of a group with proximal elements acting on a convex space. 
\begin{lemma}\label{lemma: minimal domain}
    There exists a minimal non-empty relatively closed convex $\eta$-invariant subset $C_{min} \subset \mathscr C \Omega$, and 
    \[C_{min} = \hull \Lambda_{prox}.\]
    Additionally, if $o\in C_{min}$, then 
    \[\hull \Lambda_{orb}^o = C_{min}.\]
\end{lemma}

\begin{proof}
    If $C$ is a relatively closed  non-empty convex  $\Gamma$-invariant subset of $\mathscr C \Omega$, it contains a point $p\in \mathscr C \Omega$. The orbit of this point is both contained in $C$ by invariance, and accumulates to $\Lambda_{prox}$ by Lemma \ref{lem: proximal contained in orbital}. Since $C$ is closed and convex, it must therefore contain $C_{min}$.

    For the second statement, note that the orbit $\eta(\Gamma).o\subset C_{min}$ since $o\in C_{min}$ and $C_{min}$ is $\eta(\Gamma)$-invariant. The claim follows.
\end{proof}

The conclusion of the following lemma is that the minimal domain $C_{min}$ (and as a consequence any invariant convex domain) is a fibration over $\Omega$ with fibers that are intervals.

\begin{lemma}\label{lemma: C_min is a nontrivial fibration}
    Every vertical line in $\mathscr C \Omega$ intersects $C_{min}$ nontrivially. For any $o\in \Omega$, the set $Q\inverse o \cap C_{min}$ is a closed convex subset of the line $Q\inverse o$ (i.e.\ an interval, ray, or the entire line).
\end{lemma}

\begin{proof}
   We recall a consequence of the (finite-dimensional) Hahn-Banach theorem, which is that if $P\subset \mathbb \RP^{d+1}$ is a projective subspace and $C$ is a closed properly convex domain satisfying $P\cap C = \emptyset$, then there exists a half-space $H$ such that $C\subset H$ and $P\subset \partial H$. 

    To prove the lemma, we suppose that some vertical line $L$ is disjoint from $C_{min}$ and apply Hahn-Banach to obtain a hyperplane $\partial H$ containing $L$ and bounding $C_{min}$ on one side. Because $C_{min}$ is the hull of the proximal limit set, we have an open subset of $\partial \Omega$ with the property that the $Q$-fibers over these points are disjoint from the proximal limit set. Since the proximal limit set is $\eta(\Gamma)$-invariant and the action of $\Gamma$ on $\partial\Omega$ is minimal, we have that $\Lambda_{prox}$ is empty in this case. This is a contradiction-- a representation of the form under consideration has abundant proximal elements, see for example, Benoist \cite{Benoist_CDI}.

    This shows that every vertical line intersects $C_{min}$ (and hence any non-empty closed invariant convex set) non-trivially. The result follows.
\end{proof}

\subsection{Structure of canonical invariant domains}

It is not immediately clear what the boundary of the minimal domain looks like when the representation is not Anosov. Recall from \cite{BF} that if $\Gamma$ is a surface group and $\eta$ is Anosov and not semi-simple, then $C_{min}$ fibers over $\Omega$ with fiber a closed interval; its two boundary components are disks equipped with piecewise-projective homeomorphisms to $\Omega$ and carry affine measured laminations. In the case studied in the previous paper, since $\eta$ is projective Anosov, the proximal limit set is a topological circle in $\partial \mathscr C \Omega$.
\medskip

We begin by giving the following structural result for $C_{min}$ which requires no special technology.

\begin{lemma}\label{lem: rays or lines}
    Let $\partial^\infty C_{min}=\partial C_{min}\cap \partial \mathscr C \Omega$. The non-empty pre-images of the map $Q\vert_{\partial^\infty C_{min}}$ are all simultaneously (for a given representation) exactly one of:
    \begin{itemize}
        \item closed positive rays,
        \item closed negative rays,
        \item vertical lines (the entire pre-image of the projection), or
        \item closed segments 
    \end{itemize}
\end{lemma}
\begin{proof}
    Convexity of $C_{min}$ gives that the fibers are connected and they are closed and convex by definition. All that remains is the exclusivity of the four cases. Suppose that $p$ and $q$ are points in $\partial C_{min}\cap \partial \mathscr C \Omega$, and that the positive ray $[p,[e_{d+2}])\subset \partial C_{min}$. Consider a sequence of points $p_n$ on this ray tending to $[e_{d+2}]$, and observe that the segment $[p_n,q]\subset \overline{C_{min}}$ by convexity. These segments accumulate to the positive ray $[q,[e_{d+2}])$ which must therefore be a subset of $\partial C_{min}$. Thus if there exists any point in $\partial C_{min}$ with a positive ray over it, then in fact \emph{every} point with non-empty fiber has this property. The symmetric fact about negative rays completes the proof. 
\end{proof}

Using the lemma above, we may describe the topology of $C_{min}$ and the part of its boundary in $\mathscr C \Omega$.  

\begin{proposition}\label{prop: topology of Cmin}
    When $\eta$ is not semi-simple, then 
    \begin{itemize}
        \item 
        if $\partial^\infty C_{min}$ is a union of positive (resp. negative) rays, the closure of the minimal domain $\overline{C_{min}}\subset \RP^{d+1}$ is homeomorphic to a $(d+1)$-ball and the boundary $\partial C_{min}\cap \mathscr C \Omega$ is a single disk.
        \item 
        if 
        $\partial^\infty C_{min} = \partial \mathscr C \Omega$,
        then 
        $C_{min} = \mathscr C \Omega$, and the boundary $\partial C_{min}\cap \mathscr \mathscr C \Omega$ is empty.
        \item 
        if $\partial^\infty C_{min}$ is a union of (possibly degenerate) segments, then $\overline{C_{min}}$ is homeomorphic to a $(d+1)$-ball and the boundary $\partial C_{min}\cap \mathscr C \Omega$ is a union of two open $d$-disks.
    \end{itemize}
    In the first and last cases the boundary disks have no zero-faces (as subsets of $\partial C_{min}$, considered as a convex set) and when $(d=2)$, the set of $1$- and $2$-faces form the leaves and complementary $2$-strata of a geodesic lamination on $\Omega$. The projection $Q$ restricted to any stratum of this lamination is the restriction of a projective isomorphism, i.e., the disks carry a `stratified projective structure', which we do not make precise.
\end{proposition}

\begin{proof}
    In the second case $C_{min} = \mathscr C\Omega$ and there is nothing to prove. 

    In the first and third cases, we have already shown in Lemma \ref{lemma: C_min is a nontrivial fibration} that every vertical line intersects the minimal domain.
    Suppose that there were a vertical line $L$ for which $L\cap C_{min} = \{p\}$, a singleton. Then consider a pair of sequences in $L$, $x_n$ and $y_n$ tending to $p$ from above and below, respectively. Since $x_n$ and $y_n$ are projective subspace of $\RP^3$ disjoint from $C_{min}$, there must exist half-spaces $H^x_n$ and $H^y_n$ containing $x_n$ and $y_n$ in their boundaries and containing $C_{min}$ (by Hahn-Banach). These sequences of half-spaces subconverge to $H_+$ and $H_-$ both of which are halfspaces with $p$ in their boundary and $C$ in their intersection. We obtain, similar to the argument in the proof of Lemma \ref{lemma: C_min is a nontrivial fibration}, that either $C_{min}$ is empty (a contradiction), or that $\eta$ preserves a hyperplane, i.e.\ that it is semisimple, which it is assumed not to be. We conclude that in the final case of the Proposition, every vertical line meets $C_{min}$ in a closed segment which is not a singleton. The endpoints of these segments form the pair of disks required by the lemma. They are disks because they are the graphs of convex (or concave) functions $\Omega \to \R$, after identifying the vertical lines in $\mathscr C \Omega$ with $\R$.

    In the first case, we similarly obtain that every vertical line intersects $C_{min}$ in a ray-- the endpoints of these rays inside of $\mathscr C \Omega$ form a disk as required. 

    To see that there are no zero faces on the boundary disks, we recall the following fact from convex geometry: if $S$ is a closed subset of affine space, then the set of extreme points ($0$-faces) of $\hull S$ is itself a subset of $S$ \cite{Rockafellar}.
    
    This essentially proves the claims about laminations on the boundary disks, and we refer the reader to \cite{BF} for further details.
\end{proof}

We require the following technical lemma in \S\ref{sec: prox is orb} in order to show a refined relationship between proximal and orbital limit sets.

\begin{lemma}\label{lem: proximal rays, proximal endpoints}
   Suppose that $\partial^\infty C_{min}$ is a union of vertical rays. Then the boundary points of the rays are a subset of $\Lambda_{prox}$.
\end{lemma}

\begin{proof}
    Without loss of generality, suppose that $\partial^\infty C_{min}$ is a union of positive rays. Each ray $[p,[e_{d+2}])$ is realized as the intersection of a hyperplane with $C_{min}$ (the hyperplane being the preimage under the projection $q$ of a support plane to $\Omega$), and hence is a face of $C_{min}$. Since the point $p$ is itself a face of this face, the singleton $p$ is also a face of $C_{min}$ (see \S\ref{subsec: convex and divisibles}), i.e.\ it is an extreme point. We conclude that $p\in \Lambda_{prox}.$
\end{proof}

\section{The domain shape metric}\label{sec: domain shape}
Let $\Omega\subset \RP^d$ be a strictly and properly convex domain.
We recall a construction for a complete, asymmetric $\rho$-invariant Finsler metric $d_\rho$ on $\Omega$ that behaves in many ways like a metric of negative curvature.  We follow \cite[Section 2]{Reid:finsler}, although the construction is due to Jeff Danciger and Florian Stecker.
One of the main features that interests us about this metric is that 
\begin{equation}\label{eqn: periods are lambda1}
    \inf_{p \in \Omega} d_\rho(p,\gamma.p) = \lambda_1^\rho(\gamma)
\end{equation}
for all $\gamma \in \Gamma$, where $e^{\lambda_1^\rho(\gamma)}$ is the largest eigenvalue of $\rho(\gamma)$. This agrees with the hyperbolic metric when $\Omega$ is an ellipsoid. 

In order to construct $d_\rho$, we first choose a smooth equivariant convex lift $f: \Omega \to \R^{d+1}$, which is asymptotic to $\Omega$ in the sense that the line $\overline{f(p)f(q)}$ is never in $\overline{\Omega}$.  The \emph{hyperbolic affine sphere} is an example; see \cite{Cheng-Yau}.
Now, for points $p,q\in \Omega$, let $r\in \partial^\infty \Omega$ be the point where the projective line starting at $q$ through $p$ first meets $\partial^\infty\Omega$.  Let $[\beta]\in (\RP^2)^*$ be such that $[\ker\beta]$ supports $\overline{\Omega}$ at $r$.  
Then define
\begin{equation}\label{eqn:def metric}
    d_\rho(p,q) = \log\frac{\beta(f(q))}{\beta(f(p))}.
\end{equation}

The function $d_\rho$ depends on the chosen lift $f$, but we suppress this from the notation.
The following theorem summarizes the features we need.
\begin{theorem}\label{thm: rhodist}
    The function $d_\rho$ defined in \eqref{eqn:def metric} is a complete uniquely geodesic asymmetric metric.  In addition: 
    \begin{enumerate}
        \item The geodesics are projective line segments.
        \item $d_\rho$ is the path metric for a $C^1$ Finsler metric $F_\rho: T\Omega \to \R_{\ge0}$.  The unit sphere in $T_p\Omega$ is $C^1$ and strictly convex.  
        \item If $p\in \Omega$ is a point on the projective line left invariant by $\rho(\gamma)$, then $d_\rho(p,\gamma.p) = \lambda_1^\rho(\gamma)$.
        
    \end{enumerate} 
\end{theorem}

Reid calls $d_\rho$ the \textbf{domain shape metric}.  The function $F_\rho$ has the explicit description
\[F_\rho(p,x) = \frac{\beta(df_p(x))}{\beta(f(p))},\]
where $[\ker \beta]$ supports $\overline{\Omega}$ at the point where the line starting at $p$ in the direction $-x$ meets $\partial \Omega$.

\begin{proof}
     $\Omega$ is strictly convex, so \cite[Lemmas 2.7 and 2.8]{Reid:finsler} give all but item 3 of the Theorem.  
     
     Every element $\gamma \in \Gamma$ of infinite order is proximal, so there is a splitting of $\R^{d+1} = [e_1] \oplus V \oplus [e_{d+1}]$, where $\rho(\gamma)e_i = e^{\lambda_i^\rho(\gamma)}e_i$, for $i = 1$ and $d+1$, while $\rho(\gamma)V = V$ is the span of all of the other generalized eigenspaces.   
    Let $\beta \not=0 \in (\R^{d+1})^*$ be such that $\beta( [e_{d+1}]\oplus V) = 0$.  Then $[\ker \beta ]$ supports $\overline\Omega$ at the repelling fixed point $[e_d]\in \partial^\infty\Omega$ of $\rho(\gamma)$, whereas $[e_1] \in \partial^\infty \Omega$ is the attracting fixed point.  
     
     Then $f(p) = ae_1+be_{d+1}$ for some non-zero $a, b \in \R$, and $f(\gamma.p) = \rho(\gamma)f(p) = e^{\lambda_1^\rho(\gamma)}ae_1 + e^{\lambda_{d+1}^\rho(\gamma)}be_{d+1}$. Thus
     \[d_\rho(p,\gamma.p) = \log\frac{\beta (e^{\lambda_1^\rho(\gamma)}ae_1 +e^{\lambda_{d+1}^\rho(\gamma)}be_{d+1})}{\beta(ae_1+be_{d+1})} = \lambda_1^\rho(\gamma),\]
     which is what we wanted to show.
\end{proof}

The following result relates the largest singular value of $\rho$ with $d_\rho$ along the orbit of a point.  Compare with \cite[Proposition 10.1]{DGK}.
\begin{proposition}\label{prop: singular values and DS metric}
For all $p\in \Omega$, there is a number $C$ such that for all $\gamma\in \Gamma$,
    \[\abs{d(p,\gamma p)-\sigma_1(\rho(\gamma))}<C\]
\end{proposition}

\begin{proof}
    Let $e_1,\dots, e_{d+1}$ be an orthonormal basis, and $e^1,\dots, e^{d+1}$ the standard orthonormal dual basis. 
    
    Let $\epsilon$ be small enough that $d_{\mathbb P}(o, \im \xi)>2\epsilon$.  For all but finitely many $\gamma \in \Gamma$, we have $\sigma_1^\rho(\gamma) >\sigma_2^\rho(\gamma)$.  By Lemma \ref{lem: uniform convergence to boundary map}, we know that for all but a finite set of  $\gamma \in \Gamma$, $d_{\mathbb P} (o, U(\rho(\gamma)))> \epsilon$ and $d_{\mathbb P} (o, S(\rho(\gamma)))>\epsilon$.  

    So, for $\gamma \in \Gamma$ not in this finite set, we write
    \[\rho(\gamma) = k_2 \diag(e^{\sigma_1}, ..., e^{\sigma_{d+1}}) k_1,\]
    where we abbreviate $\sigma_i = \sigma_i^\rho(\gamma)$.  

    Write
    \[f(o) = \sum_{i = 1}^{d+1}a_i k_1\inverse e_i,\]
    and note that the coefficients $a_i$ depend on $\gamma$. 
    Then 
    \begin{align*}
        f(\gamma.o) = \rho(\gamma)f(o) &= e^{\sigma_1} k_2 \sum_{i=1}^{d+1}a_i e^{\sigma_i-\sigma_1}e_i\\
        & = e^{\sigma_1} k_2 \left( a_1e_1 +O\left(e^{-c|\gamma|}\right)\right).
    \end{align*}
    Since $o$ is not close to $S(\rho(\gamma))$, $a_1$ has size that is bounded below by some $\delta>0$, independently of $\gamma$ (once $\gamma$ is large enough); see also \cite[Lemma A.3]{BPS19}.
    This proves, in particular, that $\gamma.o$ is close to $\langle k_2e_1\rangle = U(\rho(\gamma))$.  

\begin{figure}[h]
    \centering
    \includegraphics[width=.6\linewidth]{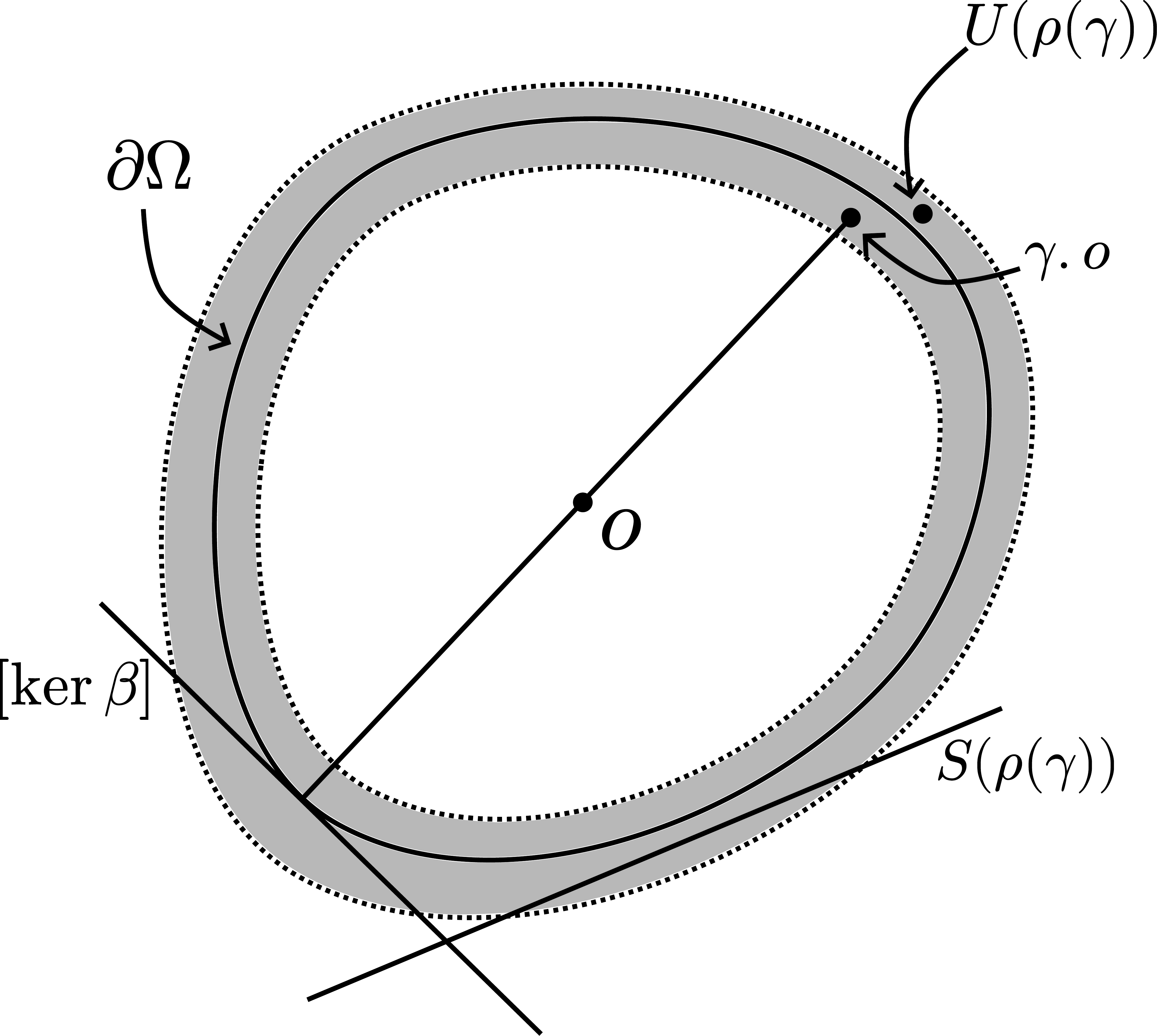}
    \caption{The domain and a neighborhood of its boundary in grey. The key pieces of the argument are labeled.}  
    \label{fig: sigma 1 prop}
\end{figure}

    Continuing the segment from $\gamma.o$ to $o$, we encounter a point $p \in \partial^\infty \Omega$, which has a unique supporting hyperplane $[\ker \beta]$; observe that $\gamma.o$ is not close to $[\ker\beta]$, by construction; see Figure \ref{fig: sigma 1 prop}.  Choose $\beta$ so that $\|\beta\|=1$ and $\beta(f(\Omega))>0$. 

    Then 
    \[\log\beta(f(\gamma.o)) = \sigma_1+ \log \beta\left(k_2a_1e_1 +O(e^{-c|\gamma|}) \right).\]
    Because $[\ker \beta]$ is antipodal to $U(\rho(\gamma)) = \langle k_2e_1\rangle$, and $|a_1|>\delta$, we conclude that $\beta(a_1k_2e_2)> \delta'$, where $\delta'>0$ is independent of (large enough) $\gamma$. This claim follows from the compactness of $\partial \Omega$ and continuity of the antipodal map induced by $o$.

    Since $\|\beta\|=1$, we have 
    \[\beta(k_2a_1e_1 +O(e^{-c|\gamma|})) \ge \delta' + O(e^{-c|\gamma|}).\]
    Thus
    \[\log\beta(f(\gamma.o)) \ge \sigma_1 + \log\delta' +O(e^{-c|\gamma|}/\delta'). \]

    On the other hand, using that $\|\beta\|=1$ and $\|f(\gamma.o)\|\le e^{\sigma_1}\|f(o)\|$, 
    \[\log\beta(f(\gamma.o))\le \sigma_1 + \log \|f(o)\|.\]

    Combining these estimates gives
    \[\left| \log\frac{\beta(f(\gamma.o))}{\beta(f(o)))} -\sigma_1\right|\le 2|\log \|f(o)\|| + \log(\delta') + O(e^{-c|\gamma|}/\delta') = O(1),\]
    where the constant term does not depend on $\gamma$, once $\gamma$ is large enough.

    Since there are only finitely many other small elements of $\Gamma$, the Proposition follows.
\end{proof}

\subsection{Properties of the geodesic flow and relation to the Hilbert metric}
Since our metric is complete, uniquely geodesic and $C^1$, the geodesic flow $g_t: T^1\Omega \to T^1\Omega$ is a well defined $C^1$ homeomorphism for all $t\in \R$.  
By invariance, the flow descends to the quotient
\[\uu g_t: T^1\Omega/\rho \to T^1\Omega/\rho.\]
Basic properties of geodesic flows in Hilbert and Finsler geometries can be found in the survey \cite{Crampon:Hilbertflow}.
The following is due to Benoist.

\begin{theorem}\label{thm: flow Anosov}
    The geodesic flow is an Anosov flow.
\end{theorem}

\begin{proof}
    Benoist proved that the Hilbert geodesic flow on a strictly and properly convex divisible domain is Anosov \cite{Benoist_CDI}.  The geodesic flow for the domain shape metric is a $C^1$-reparameterization of the Hilbert geodesic flow, and is hence Anosov (e.g. \cite[Proposition 17.4.5]{KH:dynamicalsystems}).
\end{proof}

\section{The stable norm and geometry of best forward Lipschitz maps}\label{sec: stable norm}
As in \S\ref{subsec: coaffine reps}, assume $\rho: \Gamma \to \SL(d+1,\R)$ divides a strictly and properly convex domain $\Omega \subset \RP^d$ with quotient manifold $M =\Omega /\rho$, and let $d_\rho$  be a domain shape metric on $\Omega$ from Section \ref{sec: domain shape}.  
Let $T_\rho^1M$ be the tangent vectors of unit length for the (asymmetric) Finsler norm $F_\rho$ corresponding to $d_\rho$; see Theorem \ref{thm: rhodist}.

The goal of this section is to  use \emph{equivariant best forward Lipschitz} maps to understand the geometry of a cohomology class $ \varphi \in H^1(M, \R)$ with respect to $d_\rho$.  This is an indispensable tool for the rest of the article.  

For this, we study the \emph{stable norm} and its dual \emph{co-norm} on homology and cohomology, respectively.  If the metric on $M$ of interest were smooth Riemannian or Finsler, then our main Theorem \ref{thm: co-stable norm and optimal best maps} could be obtained by combining existing results in the literature.  
However, since the domain shape metric is neither smooth nor symmetric in general, special care is needed to deduce the correct statements and adapt relevant arguments from the literature.  For this reason we have chosen to give the (mostly elementary and self-contained) proof of Theorem \ref{thm: co-stable norm and optimal best maps} in Appendix \ref{appendix} so that it does not distract from the main focus of this article.

First we discuss some preliminary material on directional Lipschitz maps and the notion of chain recurrence for geodesic laminations.

\subsection{Forward Lipschitz maps}
We first recall some preliminaries on directional Lipschitz maps.

For an asymmetric metric space $(X,d)$, a map $f: X \to \R$ is called \textbf{$L$-forward Lipschitz}\footnote{There is also a notion of \emph{backward Lipschitz continuity}, but we will not need it.} if for all $x,y\in X$, we have
\[f(y)-f(x) \le Ld(x,y).\]
For a forward Lipschitz function $f: X \to \R$, the \textbf{forward Lipschitz constant} is defined
\[\vec \Lip(f) \defeq \sup_{x\not= y}\frac{f(y)-f(x)}{d(x,y)}.\]
For $x \in X$ and $\epsilon>0$, let $B_\epsilon(x) = \{y : d(x,y)<\epsilon \text{ and } d(y,x)<\epsilon\}$, and define the \text{local forward Lipschitz constant} by 
\[\vec\Lip_f(x) \defeq \lim_{\epsilon\searrow 0}\vec\Lip\left(f|_{B_{\epsilon}(x)}\right).\]
The function $x \mapsto \vec\Lip_f(x)$ is upper semi-continuous being the decreasing limit of upper semi-continuous functions, i.e., $\vec\Lip_f\inverse [a, \infty)$ is closed for all $a \in \R$. 

The \textbf{tension locus} of a $1$-forward Lipschitz map $f: X \to \R$ is
\begin{equation}\label{eqn: def tension}
    \Tension(f) \defeq \{x\in X : \vec\Lip_f(x) =1\}.
\end{equation}
By upper semi-continuity, $\Tension(f) \subset X$ is a closed set.

\subsection{Cohomology classes and best forward Lipschitz maps}

Let $\varphi \in H^1(M,\R)$ and define
\begin{equation}\label{eqn: def costable}
    \|\varphi\|_\rho \defeq \sup_{\gamma \in \Gamma} \frac{\varphi(\gamma)}{\lambda_1^\rho(\gamma)},
\end{equation}
and let 
\[\mathcal B_\rho \defeq \{\varphi \in H^1(M,\R) : \|\varphi\|_\rho \le 1\}.\]
One goal we have is to relate $\|\varphi\|_\rho$ with the smallest forward Lipschitz constant in a certain class of $\varphi$-equivariant maps $\Omega \to \R$, i.e., 
those forward Lipschitz maps $f: \Omega \to \R$ that satisfy 
\begin{equation}\label{eqn: varphi equiv}
    f(\gamma.p) = f(p)+\varphi(\gamma)
\end{equation}
for all $p \in \Omega$ and $\gamma \in \Gamma$.  

\begin{lemma}\label{lem: lower bound lip}
We have \[\|\varphi\|_\rho\le \min_{f\sim \varphi} \vec\Lip(f)\] 
where $f\sim\varphi$ means that $f: \Omega \to \R$ is a $\varphi$-equivariant forward Lipschitz map.
\end{lemma}
\begin{proof}
    First we show that the set over which the infimum is taken  is non-empty.  
    For this, let $\omega$ be a closed smooth $1$-form on $M$ representing $\varphi$, i.e., $\varphi$ is the cohomology class determined by integrating Lipschitz $1$-cycles against $\omega$.
Define 
\[\|\omega\|_\infty \defeq \max_{v \in T_\rho^1M} \omega(v).\]
Choose a point
$o \in \Omega$ and define a function $f_\omega : \Omega \to \R$ by the rule
    \[f_\omega(p) =\int_o^p\omega.\]
A direct computation involving properties of the integral and the definitions verifies that $f_\omega$ is $\|\omega\|_\infty$-forward Lipschitz and satisfies \eqref{eqn: varphi equiv}.

Now, if $f: \Omega \to \R$ is a $\varphi$-equivariant $L$-forward Lipschitz map, observe that 
\[\varphi(\gamma) = f(\gamma.p)-f(p) \le Ld_\rho(p,\gamma.p)\]
holds for all $p\in \Omega$.  Using item 3 of Theorem \ref{thm: rhodist}, we see that
\begin{equation}\label{eqn: norm less than lipschitz}
    \varphi(\gamma) \le L\lambda_1^\rho(\gamma).
\end{equation}

Up to adding constants,  $L$-forward Lipschitz maps are compact in the topology of uniform convergence on compact sets.  The infimum is thus attained as a minimum, proving the Lemma.  
\end{proof}

The following theorem shows, in particular, that there is an equivariant forward Lipschitz function whose Lipschitz constant achieves this lower bound and also gives some geometric information about any such extremal map. Recall the definition of the tension locus from \eqref{eqn: def tension}.
\begin{theorem}\label{thm: co-stable norm and optimal best maps}
    Let $\varphi\in \partial \mathcal B_\rho \subset H^1(M,\R)$, i.e., $\|\varphi\|_\rho =1$.
    There is a $\varphi$-equivariant $1$-forward Lipschitz map $\tau:\Omega \to \R$ with the following properties.
    \begin{enumerate}

        \item $\Tension (\tau)\subset \Omega$ is a non-empty $\rho$-invariant geodesic lamination, that descends to an oriented chain recurrent geodesic lamination $\Tension(\tau)/\rho\subset M$.
        \item $\tau$ is \textbf{optimal} in the sense that 
        \[\Tension (\tau) = \Tension (\varphi) \defeq \bigcap_{f} \Tension(f),\]
        where the intersection is taken over all $\varphi$-equivariant $1$-forward Lipschitz maps. 
        \item If $\{\gamma_n\}$ is any sequence in $\Gamma$ satisfying 
        \[\lim_{n \to \infty} \frac{\varphi(\gamma_n)}{\lambda_1^\rho(\gamma_n)} =1,\]
        then every accumulation point of $\{\gamma_n\}$ in the measure topology has support that is contained in $\Tension(\varphi)/\rho \subset M$.  
    \end{enumerate}
    In particular $\|\varphi\|_\rho = \min_{f\sim \varphi}\vec \Lip(f) = \vec\Lip(\tau)$.
\end{theorem}

See \S\ref{subsec: geodesic laminations} for further explanation of what item (3) means.   

\medskip

Parts of the Theorem are well-known in some more familiar settings, e.g., $C^2$--Finsler metrics or negatively curved Riemannian metrics.  
We provide a fairly self-contained proof in Appendix \ref{appendix} adapting arguments from the literature, namely from \cite{McShane,GK:stretch, Bangert:equiv, Burago:periodic, BIK:stablenorm,AL:stablenorm, FLM:II}.

\section{Slack of paths}\label{sec: slack}

Assume $\rho: \Gamma \to \SL(d+1,\R)$ divides a strictly and properly convex domain $\Omega \subset \RP^d$ with quotient manifold $M =\Omega /\rho$.
Let $\varphi \in H^1(M,\R)$ be such that $\|\varphi\|_\rho =1$.  Using  Theorem \ref{thm: co-stable norm and optimal best maps}, we find a $\varphi$-equivariant $1$-forward Lipschitz optimal tight map 
\[\tau: \Omega \to \R\]
with tension locus $\Tension(\tau) = \Tension (\varphi)$.  
Recall that $\Tension(\tau)$ is an oriented geodesic lamination and $\Tension(\tau)/\Gamma \subset M$ is chain recurrent.

The goal of this section is to define a useful quantity called \textbf{slack} associated to paths, following \cite{FLM:II}.  Let $\alpha : [a,b] \to \Omega$ be any rectifiable path, where $a<b\in \R$, and define
\[\s(\alpha) \defeq \ell(\alpha) - (\tau(\alpha(b)) - \tau(\alpha(a))),\]
where $\ell(\alpha)\ge 0$ is the \emph{directed} length of $\alpha$.
We record here two basic properties of slack that follow directly from the definitions and properties of $\tau$.  
\begin{lemma}\label{lem: slack non-negative}
    For any $\alpha: [a,b] \to \Omega$, we have $\s(\alpha)\ge0$ with equality if and only if $\alpha$ is a constant map or $\im(\alpha)\subset \Tension (\tau)$ and $\alpha$ is a non-decreasing parameterization of its image in the forward orientation, i.e., $\alpha$ is a (directed) reparameterization of the (directed) leaf of $\Tension(\tau)$ containing it.
\end{lemma}
Clearly, slack is additive.  That is, if $\alpha = \alpha_1*\alpha_2$, then $\s(\alpha) = \s(\alpha_1) +\s(\alpha_2)$.  
For two points $p,q \in \Omega$, we write $\s[p,q]$ to be the slack of the directed geodesic segment $[p,q]$ joining $p$ to $q$ in $\Omega$.

We extend the definition to (bi)-infinite paths as follows.  If $I \subset \R$ is connected and $\alpha: I \to \Omega$ is rectifiable, then we define
\[\s(\alpha)\defeq \sup_{J\subset I}\s(\alpha|_J)\in [0,\infty].\]
By additivity and non-negativity, if $\alpha: \R\to \Omega$, we have  
\[\s(\alpha) = \lim_{t\to \infty} \ell(\alpha|_{[-t,t]}) - \tau(\alpha(t)) +\tau(-t).\]
We record the following easy continuity property of slack. 
\begin{lemma}\label{lem: slack continuous finite}
    Let $A, B \subset \Omega$ be subsets with $\diam A \le \epsilon_1$ and $ \diam B \le \epsilon_2$.\footnote{By $\diam A\le \epsilon_1$, we mean that $d_\rho(a_1, a_2)\le \epsilon_1$ for all $a_1, a_2 \in A$.}  For any points $a_1, a_2 \in A$ and $b_1, b_2 \in B$, we have 
    \[\s[a_1, b_1]-\s[a_2, b_2] \le 2(\epsilon_1+\epsilon_2).\]
\end{lemma}
\begin{proof}
    By hypothesis that $d_\rho(a_1, a_2)\le \epsilon_1$ and $d_\rho(b_1, b_2) \le \epsilon_2$, by the triangle inequality, we have $\ell[a_1,b_1]  - \ell[a_2,b_2] = d_\rho(a_1, b_1) -d_\rho(a_2, b_2) \le \epsilon_1 +\epsilon_2$.  Since $\tau$ is $1$-forward Lipschitz, we have $\tau(a_1) - \tau(a_2) \le d(a_2,a_1)\le \epsilon_1$ and $\tau(b_2)-\tau(b_1)\le d_\rho(b_1, b_2)\le \epsilon_2$.  Thus
    \[\s[a_1, b_1]-\s[a_2,b_2] = \ell[a_1,b_1] -\ell[a_2,b_2] +(\tau(b_2)-\tau(b_1)) + (\tau(a_1) -\tau(a_2)) \le 2(\epsilon_1+\epsilon_2),\]
    completing the proof.
\end{proof}

If $x, y \in \partial\Omega$ are distinct, then it makes sense to talk about $\s[x,y]\in [0,\infty]$.
Indeed, consider the set $\mathcal G$ of pairs of leaves of $\Tension(\tau)$.  

\begin{lemma}\label{lem: slack finite and continuous}
    If $x,y \in \partial \Omega$ are such that $x$ is the backward endpoint of a leaf $l_x$ of $\Tension(\tau)$ and $y$ is the forward endpoint of a leaf $l^y$ of $\Tension(\tau)$, then $\s[x,y]<\infty$.  Moreover, the function $(l_x,l^y) \in \mathcal G \mapsto \s[x,y]$ is continuous.
\end{lemma}

\begin{proof}
    First we prove the finiteness statement.
    Given $\epsilon>0$, we claim that there are points $a_1, b_1 \in [x,y]$, $a_2 \in l_x$ and $b_2 \in l^y$ such that the following hold.  Every point on $[x,a_1]\subset [x,y]$ is distance at most $\epsilon$ from $[x,a_2]\subset l_x$, and vice versa.  Also, every point on $[b_1,y]\subset [x,y]$ is distance at most $\epsilon$ from $[b_2, y]\subset l^y$, and vice versa.  
    Indeed, as in \S\ref{sec: domain shape}, the Hilbert geodesic flow on $\Omega$ is Anosov, so asymptotic lines have distance tending to zero in the Hilbert metric.  Since $d_\rho(a,b) \le 2d_H(a,b)$ holds for all $a,b \in \Omega$, this implies the claim.

    By Lemma \ref{lem: slack non-negative}, we have $\s[x,a_2]=0$ and $\s[b_2,y] = 0$.
    For any $c_1 \in (x, a_1)$, there is a $c_2\in (x,a_2)$ within distance $\epsilon$ of $c_1$.  Using Lemma \ref{lem: slack continuous finite} we have  
    \[\s[c_1,a_1] \le \s(c_2,a_2) + 4\epsilon = 4\epsilon.\]
    Thus $\s[x,a_1]\le 4\epsilon$.  A similar argument gives $\s[b_1,y]\le 4\epsilon$.
    Putting this all together, we have \[\s[x,y] = \s[x,a_1]+\s[a_1,b_1]+\s[b_1,y] \le 8\epsilon +\s[a_1,b_1]<\infty,\]
    proving the finiteness claim.
    \medskip
    
    Now we prove continuity.
    The topology on geodesics is such that $l$ and $l'$ are close if they are close in the Hausdorff topology on large metric balls around some basepoint $o\in \Omega$.
    
    Let  $(l_x, l^y)\in \mathcal G$ and $\epsilon>0$ be given. Consider the geodesic $[x,y]$ and find points $a_1 , b_1 \in [x,y]$ and points $a_2 \in l_x$ and $b_2 \in l^y$, 
    as in  the first part of the proof.  If $x'$ is close enough to $x$ and $y'$ is close enough to $y$, then the geodesic $[x',y']$ is $\epsilon$-close to $[x,y]$  in a large ball around $o$ that contains $[a_1,b_1]\subset [x,y]$. 
    In particular, there are $a_1', b_1' \in [x',y']$ such that $\diam\{a_1, a_1'\}\le \epsilon$ and $\diam \{b_1, b_1'\}\le \epsilon$.  
    
    Now, if $l_{x'}$ is close enough to $l_x$ and $l^{y'}$ is close enough to $l^y$, then  also there are points $a_2' \in l_{x'}$ and $b_2' \in l^{y'}$ such that $\diam\{a_1', a_2'\}\le \epsilon$ and $\diam\{b_1', b_2'\}\le \epsilon$.
    We have 
    \[\s[x,y] = \s[a_1,b_1]+O(\epsilon),\]
    and
    \[\s[x',y'] = \s[a_1',b_1']+O(\epsilon).\]
    Applying Lemma \ref{lem: slack continuous finite} gives 
    \[\s[a_1,b_1]=\s[a_1',b_1']+O(\epsilon).\]
    Thus, $\s[x,y] = \s[x',y']+O(\epsilon)$ if $(l_x,l^y)$ is close enough to $(l_{x'}, l^{y'})$ in $\mathcal G$.    
\end{proof}

For notational convenience, we ignore the difference between $\tau$ and $\pi^*\tau$, where $\pi: T^1\Omega \to \Omega$ is the projection.  Also, we use $\lambda$ to denote $\Tension (\tau) \subset \Omega$ and $\underline \lambda$ to denote $\Tension(\tau)/\rho\subset M$.  Using its orientation, let $T^1_+\lambda\subset T^1M$ be the directions tangent to the leaves of $\lambda$ in the forward direction.
Choose a distance $d$ on $T^1M$ for which lifts of geodesics from $M$ have the same $d$-length and the projection is $1$-Lipschitz.

Given $\delta>0$, recall from \eqref{eqn: def L_delta} the (continuous) function 
\[L_\delta: T_\rho^1M \to \R \]
defined with respect to $\tau$ by
\[L_\delta(x) = \frac{\tau(g_\delta\tilde x) - \tau(\tilde x)}{\delta},\]
where $\tilde x \in T^1_\rho\Omega$ is any point in the preimage of $x\in T^1M$.
By optimality of $\tau$, the set $L_\delta\inverse(1)$ is exactly the set $T^1_+\underline \lambda \subset T^1M$.
The following is a fairly straightforward consequence of the definitions and continuity.

\begin{lemma}\label{lem: slack grows}
    For any $\delta>0$ and any $\epsilon>0$, there is a $\vartheta>0$ such that if $T\ge 2\delta$ and  $\alpha: [0,T] \to \Omega$ is a directed geodesic segment for which $\alpha': [0,T] \to T^1\Omega$ avoids the $\epsilon$-neighborhood of $T^1_+\lambda$, then 
    \[\s(\alpha)\ge \vartheta\ell(\alpha).\]
\end{lemma}

\begin{proof}
    We claim that there is an $\epsilon'>0$ such that $L_\delta\inverse ([1-\epsilon', 1])$ is contained in the $\epsilon$-neighborhood $\mathcal N_\epsilon(T_+^1\underline \lambda)$ of $T_+^1\underline \lambda$ in $T^1M$.
    Indeed, this follows directly from continuity of $L_\delta$, the fact that $L_\delta\inverse (1) =T^1_+\underline \lambda$, and compactness of $T^1_+\underline \lambda$.  So if $\alpha'$ avoids $\mathcal N_\epsilon(T_+^1\underline \lambda)$, then $L_\delta(\alpha'(t))\le1-\epsilon'$, for all $t \in [0,T-\delta]$.
    Thus 
    \[1-\epsilon'\ge L_\delta(\alpha'(t)) = \frac{\tau(\alpha(t+\delta)) - \tau(\alpha(t))}{\delta},\]
    so that rearranging the expression and noting that $\ell(\alpha|_{[t,t+\delta]}) =\delta$ gives
    \[\epsilon'\delta \le \s(\alpha|_{[t,t+\delta]}).\]
    Let $N = \lfloor T/\delta\rfloor$, so that $T-\delta\le N\delta\le T$.  Using additivity and non-negativity of slack, we have 
    \[\s(\alpha)\ge \sum_{i =0}^{N-1}\s(\alpha|_{[i\delta, i\delta +\delta]}) \ge N\epsilon'\delta\ge \epsilon'(T-\delta).\]
    Since $T\ge 2\delta$, it follows that $T-\delta\ge T/2$.  So the lemma follows with $\vartheta = \epsilon'/2$. 
\end{proof}

\section{The Anosov property fails along the tension lamination}\label{sec: failure of anosov}

In this section, $\rho: \Gamma\to \SL(d+1, \R)$ divides a properly convex domain $\Omega\subset \RP^d$, so that $M = \Omega/ \rho$ is a compact $d$-manifold, and $\chi = e^{-\varphi}: \Gamma \to \R_{>0}$ has  $\|\varphi\|_\rho=1$.
We choose an optimal $\varphi$-equivariant $1$-forward Lipschitz  map \[\tau: \Omega \to \R\] provided by Theorem \ref{thm: co-stable norm and optimal best maps} with corresponding tension locus $\lambda = \Tension(\tau)$.  Then $\uu\lambda = \lambda/\rho \subset M$ is an oriented chain recurrent geodesic lamination.

The geodesic flow $\underline g_t : T^1M \to T^1M$ preserves the set $T^1_+\underline \lambda$ of directed tangent vectors to $\underline \lambda$.  Since the flow is continuous, and $T^1_+\underline\lambda$ is compact, there is a non-empty simplex of flow invariant Borel probability measures $\Delta(\underline\lambda)$ with support contained in $\T^1_+\underline \lambda$; see \S\ref{subsec: geodesic laminations}.
\medskip

For a finitely generated group $\Gamma$, let $(X,d)$ be a \textbf{quasi-isometric model} for $\Gamma$, i.e., $(X,d)$ is a metric space on which $\Gamma$ acts properly discontinuously and cocompactly by isometries.  Denote the \textbf{stable length} by
\[\|\gamma\|_\infty := \lim_{n\to \infty}\frac{d(x,\gamma^nx)}{n},\]
which is independent of $x$. 
By work of Kassel--Potrie \cite{KP:eigen}, the $1$-Anosov property for a representation $\rho: \Gamma \to \GL(d,\R)$, with $d\ge 2$, is equivalent to the following.  First $\Gamma$ is Gromov hyperbolic, and second there are $a,b >0$ such that 
\[ (\lambda_1^\rho - \lambda_2^\rho)(\gamma)\ge a\|\gamma\|_\infty-b\]
for all $\gamma \in \Gamma$, where $\lambda_i^\rho(\gamma)$ is the logarithm of the modulus of the $i$th largest eigenvalue of $\rho(\gamma)$.  
\medskip

Let $\eta: \Gamma \to \PGL(d+2,\R)$ be any coaffine representation with linear part $\rho\chi$.  
\cite{Lahn:I, Lahn:II, Barbot:stablenorm}  gave  criteria when certain reducible representations are on the boundary of the locus of Anosov representations (see Corollary \ref{cor: failure of anosov}, below, for a concise explanation in our context).  
We give a description of some qualitative features of sequences of curves which are witnesses to the failure of the Anosov condition, which is new.  

Recall from \S\ref{section: geometry of coaffine domain} that a choice $Z \in Z^1(\Gamma,(\R^{d+1})^*)$ determines a coaffine representation $\eta= \eta(\rho,\chi,Z) : \Gamma \to \cag$ with linear part $\rho\chi$.

\begin{theorem}\label{thm: failure of anosov along lamination}
    With $\eta=\eta(\rho,\chi,Z)$, as above, there exist sequences $\gamma_n$ with $\|\gamma_n\|_\infty \to \infty$ with the property that 
    \[(\lambda_1^{\eta} - \lambda_2^\eta)(\gamma_n) \to 0, ~n\to \infty.\]
    Moreover, for any sequence $\gamma_n$ with $\|\gamma_n\|_\infty \to \infty$, if 
    \[\frac{(\lambda_1^{\eta} - \lambda_2^\eta)(\gamma_n)}{\|\gamma_n\|_\infty} \to 0\]
    then the only accumulation points of the sequence of projective classes of oriented currents corresponding to $\{\gamma_n\}$ are contained in the simplex of (oriented) measures with support contained in $\underline \lambda$.
\end{theorem}

Morally, the second statement of the theorem follows directly from Theorem \ref{thm: co-stable norm and optimal best maps} item (3).  We give details here. 

\begin{proof}
    First, observe that, as functions on $\Gamma$,
    \[\lambda_2^\eta = \max \{ \lambda_2^\rho - \varphi, 0\},\]
    so that 
    \[\lambda_1^\eta - \lambda_2^\eta = \min\{ \lambda_1^\rho - \lambda_2^\rho, \lambda_1^\rho - \varphi\}.\]
    
    We sketch a proof of the first statement of the Theorem.  
    By a classical argument of Birkhoff, there is a (dynamically) minimal component $\underline\mu \subset \underline \lambda$; we assume this component is not just a  periodic orbit (otherwise take a sequence of non-primitive closed curves, which all have eigenvalue gap identically zero).  Let $x \in T_+^1\underline \mu$ and $k$ be a small transversal hypersurface in $T^1M$ to $T_+^1\underline\mu$ meeting $x$.  This supplies a continuous first return map $P: k\cap T_+^1\underline \mu \to k\cap T_+^1\underline \mu$.  Since $T_+^1\underline \mu$ is minimal, $ k\cap T_+^1\underline \mu $ is perfect.  If $x$ is isolated, then $P^n(x) = x$, for some $n\ge 1$.  The $\underline g_t$ orbit of $x$ defines a closed geodesic loop $\gamma_x$ for which $(\lambda_1^\eta -\lambda_2^\eta)(\gamma_x^k) = (\lambda_1^\rho -\varphi)(\gamma_x^k) =0$, for all $k\ge 1$, because the change in $\tau$ is equal to directed arclength in the domain shape metric along $\lambda$.  Then $\{\gamma_x^k\}_{k\ge1}$ gives a sequence satisfying the first statement in the theorem, in the case that $x$ is isolated.
    
    If $x$ is not isolated, then $x$ is a non-trivial accumulation point of  $\{P^k(x) : k \ge1\}$.  The geodesic flow $\tilde {\underline g}_t$ for the Hilbert metric is an Anosov flow (see \S \ref{sec: domain shape} or \cite{Benoist_CDI}), so using the Anosov Closing Lemma \cite[\S6.4]{KH:dynamicalsystems}, the canonical lift of a long geodesic segment joining $x$ to $P^{n_k}(x)$ is  close to a closed geodesic segment $\gamma_k$ in $M$.  
    Since $Cd_H\le d_\rho\le 2d_H$, we conclude also that this closed geodesic segment is close also for the domain shape metric (not just the Hilbert metric).  
    Lifting the situation to $\Omega$, from the definitions,  we see that 
    \[\s[p_k,\gamma_k.p_k] = (\lambda_1^\rho - \varphi)(\gamma_k),\]
    where $p_k$ is any point on the axis of $\gamma_k$.  
    Choose a lift $\tilde x \in T^1\Omega$ of $x$ and let $q$ be the projection to $\Omega$.  Let $q_k \in \Omega$ be the point corresponding to $P^{n_k}(x)$  on the line in $\Omega$ tangent to $x$.  
    By conjugation, we may assume that the axes of $\gamma_k$ pass very close to $q$, and choose $p_k$ tending to $q$.  
    Using Lemma \ref{lem: slack continuous finite}, we have 
    \[\s[p_k,\gamma_k.p_k] = \s[q, q_k] +o(1) = o(1),\]
    where $\s[q,q_k] = 0$, because $[q,q_k]\subset\lambda$.  
    Then $(\lambda_1^\eta - \lambda_2^\eta)(\gamma_k) = (\lambda_1^\rho-\varphi)(\gamma_k) = \s[p_k,\gamma_k.p_k]\to 0$, as $k\to \infty$, so $\{\gamma_k\}_{k \ge1}$ supplies a sequence satisfying the first statement of the theorem.
    \medskip

    Now we prove the second statement, which proceeds along the lines of the proof of Lemma \ref{lem: Tension is non-empty}.
    Let $\gamma_n$ be a sequence in $\Gamma$ for which the gap between the logarithms of the first and second largest eigenvalues of $\eta$ grows sublinearly.  Let $\delta_n$ be the corresponding oriented closed geodesics in $M$, and $\mu_n$ the flow-invariant Lebesgue probability measures supported on the canonical lifts $T_+^1\delta_n$ of $\delta_n$ to $T^1M$.  
    Let $\mu$ be a flow invariant probability measure  on $T^1M$ that is a weak-$*$ accumulation point of $\{\mu_n\}$, and let $\mathcal \delta_\infty\subset M$ be the projection of the support of $\mu$.  
    
    Suppose for sake of contradiction that $\delta_\infty \not\subset \lambda$.  Then there is a point $x \in \supp(\mu) \setminus  T^1_+ \underline\lambda$ and $\epsilon>0$ such that $B_\epsilon(x)\cap T^1_+\underline \lambda = \emptyset$ and $\mu(B_\epsilon(x))>0$.  Recall the function 
    \[L_\delta(x) = \frac{\tau(g_{\delta}\tilde x) - \tau(\tilde x)}{\delta}\]
    defined on $T^1M$ from \S \ref{sec: stable norm} and again from  \S\ref{sec: slack}, which is continuous and satisfies $L_\delta\inverse (1) = T_+^1\underline \lambda$ by the definitions and optimality of $\tau$; in particular, we know that $L_\delta(y)\le 1-\epsilon'$ for all $y \in B_\epsilon(x)$ and some $\epsilon'>0$. 
    It is an exercise 
    to verify that
    \[1- \frac{\s[p_n,\gamma_np_n]}{\lambda_1^\rho(\gamma_n)} = \int L_\delta ~d\mu_n,\]
    where $p_n\in \Omega$ is a point on the axis of $\gamma_n$.
    Since $\mu$ is a weak-$*$ accumulation point of $\{\mu_n\}$, and $L_\delta$ is continuous, we have (up to passing to a subsequence that we do not rename)
    \begin{align*}
        \int L_\delta ~d\mu_n \to \int L_\delta~d\mu &= \int_{B_\epsilon(x)} L_\delta ~d\mu +\int_{B_\epsilon(x)^c} L_\delta~d\mu\\
        & \le (1-\epsilon') \mu(B_\epsilon(x)) + \mu(B_\epsilon(x)^c) \le 1-\epsilon''<1,
    \end{align*}
    where $\epsilon''$ is positive. 

    Together with the description of $\int L_\delta ~d\mu_n$, we conclude that 
    \[\epsilon'' \lambda_1^\rho(\gamma_n) \le  \s[p_n,\gamma_np_n] = (\lambda_1^\eta-\lambda_2^\eta)(\gamma_n).\]
    Since $(\Omega, d_H)$ is a quasi-isometric model for $\Gamma$ and $\lambda_1^\rho$ is bounded below uniformly by a positive multiple $C$ of the Hilbert geodesic length, we have $C\|\gamma \|_\infty\le \lambda_1^\rho$.
    So this contradicts the assumption that the gap between the first two logarithmic eigenvalues grows sublinearly in the stable lengths of the loops $\gamma_n$, and completes the proof of the theorem.    
\end{proof}
Theorem \ref{thm: failure of anosov along lamination}  says, informally, that the only \emph{witnesses} to the failure of the Anosov property are the measured components of $\underline \lambda$.

We now obtain, as an easy consequence, a characterization of when a coaffine representation is Anosov; see also \cite{Lahn:I, Lahn:II, Barbot:stablenorm}.

\begin{corollary}\label{cor: failure of anosov}
    For a coaffine representation $\eta= \eta(\rho,\chi,Z): \Gamma \to \PGL(d+2,\R)$, we have that $\eta$ is 1-Anosov if and only if 
    \[\|\varphi\|_\rho = \sup_{\gamma \in \Gamma} \frac{\varphi(\gamma)}{\lambda_1^\rho(\gamma)}<1.\]
\end{corollary}
\begin{proof}
    The ``only if'' part follows directly from (the proof of) Theorem \ref{thm: failure of anosov along lamination}, which exhibits explicit sequences of curves for which the first eigenvalue gap fails to grow at any linear rate.
    The ``if'' part goes as follows.  
    If $\|\varphi\|_\rho\le 1-\epsilon$, for some $\epsilon>0$, then 
    \[\epsilon\lambda_1^\rho \le \lambda_1^\rho - \varphi.\]
    Thus there is an eigenvalue gap for $\eta$ supplied by the Anosov property of $\rho$. 
    Indeed, $\lambda_1^\rho(\gamma) = \|\gamma\|_\infty$ for the $d_\rho$ metric, which is a quasi-isometric model of $\Gamma$, and $\eta(\gamma)$ has $1$ as an eigenvalue for all $\gamma \in \Gamma$.  
\end{proof}

Now we prove Main Theorem \ref{mainthm: ending lamination intro} from the introduction, which we reproduce here for convenience of the reader.
Neither Theorem \ref{mainthm: ending lamination intro} nor Theorem \ref{thm: failure of anosov along lamination} seem to follow from one another.  
\begin{theorem}\label{thm: ending lamination hausdorff}
    With hypotheses as in Theorem \ref{thm: failure of anosov along lamination}, for every sequence of oriented  closed curves $\gamma_n\subset M$ (not necessarily primitive) such that 
    \begin{equation}\label{eqn: length to zero 2}
        (\lambda_1^\eta - \lambda_2^\eta)(\gamma_n) \to 0, \text{ as } n\to \infty,
    \end{equation}
     every limit point of any sequence satisfying \eqref{eqn: length to zero 2} 
     is contained in $\uu\lambda$  compatibly oriented,
     and $\uu\lambda$ is the union of such limit points.  Convergence is in the Hausdorff topology on closed subsets of $M$.
\end{theorem}

\begin{proof}
    If $(\lambda_1^\eta-\lambda_2^\eta) (\gamma)$ is small enough, then $(\lambda_1^\eta-\lambda_2^\eta) (\gamma) = (\lambda_1^\rho - \varphi)(\gamma)$, which is the slack of a fundamental domain for the action of $\gamma$ on its axis in $\Omega$.  As a consequence of \S\ref{sec: slack}, a sequence of curves with slack tending to zero must be getting closer and closer to $\uu\lambda$.  This proves that any Hausdorff accumulation point of such a sequence is contained in $\uu\lambda$.  That $\uu\lambda$ is the union of such limit points is a consequence of chain recurrence.  
\end{proof}

\section{Equivariant sections of $\mathscr C \Omega$ and height}\label{sec: sections and height}

We have identified some basic objects of interest and developed some tools to study them.  
In order to understand basic structural features of limit sets of coaffine representations, we introduce, in this section, a family of equivariant hypersurfaces in $\mathscr C \Omega$.   
We also define a ``height'' function on $\mathscr C \Omega \cup \partial \mathscr C \Omega$.  Our equivariant surfaces have the property that they are graphs of height functions, and much of the remainder of the paper is devoted to studying these.

\subsection{Equivariant sections of $\mathscr C\Omega$}\label{subsec: equivariant section of cone}
In this section, $\rho: \Gamma\to \SL(d+1, \R)$ divides a properly convex domain $\Omega\subset \RP^d$, so that $M = \Omega/ \rho$ is a compact $d$-manifold, and $\chi = e^{-\varphi}: \Gamma \to \R_{>0}$ has  $\|\varphi\|_\rho=1$.
We choose an optimal $\varphi$-equivariant $1$-forward Lipschitz  map \[\tau: \Omega \to \R\] provided by Theorem \ref{thm: co-stable norm and optimal best maps} with corresponding tension locus $\lambda = \Tension(\tau)$.  
\medskip

There exists a smooth convex $\rho$-equivariant section \[\Phi:\Omega\rightarrow \widehat \Omega\subset \R^{d+1}\] where $\widehat \Omega$ is the preimage of $\Omega$ under the canonical projection defining $\RP^d$ (for example, the hyperbolic affine sphere \cite{Cheng-Yau}).

Define 
\[\Psi(p) = e^{-\tau(p)}\Phi(p) \text{ and } \chi = e^{-\varphi}.\]
Since $\varphi(\gamma)=\tau(\gamma.p)-\tau(p)$, it follows that $\Psi: \Omega\to \widehat \Omega$ is $\rho\chi$-equivariant. It is continuous because it is a product of continuous maps.

The following result relates the sizes of $\Phi$ and $\Psi$ with important geometric features of $\Omega$ and $\varphi$.  Recall that $d_\rho$ is the domain shape metric on $\Omega$ (built with respect to $\Phi$) from \S\ref{sec: domain shape}.
\begin{lemma}\label{lemma: basic slack estimates}
Given $o\in \Omega$ and $R>\textrm{diam}(\Omega/\rho,d_\rho)$, for all $p,p'\in B_R(o)$ and $q\in\Omega$,
\[\|\rho \chi(\gamma) \|\asymp e^{\s[p,\gamma.p']} \text{ and } \|\Psi(q)\|\asymp e^{\s [p,q]}.\]
The implicit constants depend on $o$ and $R$, but not on $\gamma$.
\end{lemma}

\begin{proof}
    By definition of singular values, we have $\|\rho(\gamma)\| = e^{\sigma_1(\rho(\gamma))}$, and so $\|\rho\chi(\gamma)\| = e^{\sigma_1(\rho(\gamma)) - \varphi(\gamma)}$, because $\chi(\gamma)$ is central.
    Using Proposition \ref{prop: singular values and DS metric} and the definition of the slack, we have
    \[\|\rho\chi(\gamma)\| = e^{d_\rho(o, \gamma.o) -(\tau(\gamma.o) - \tau(o)) +O(1)} \asymp e^{\s[o,\gamma.o]},\]
    where the $O(1)$ appearing in the exponent comes from the constant of the Proposition.  
    Since $p,p' \in B_R(o)$, by the basic continuity property of slack (Lemma \ref{lem: slack continuous finite}), we have $|\s[o,\gamma.o] - \s[p,\gamma.p']|<4R$, which gives 
    \[e^{\s[o,\gamma.o]}\asymp e^{\s[p,\gamma.p']},\]
    thus the first claim is established with $C$ depending on $o$ and $R$.

    The second claim is established similarly, using $\rho\chi$-equivariance.  Indeed, $q = \gamma.q'$ for some $\gamma \in \Gamma$ and $q' \in B_R(o)$.  Consider the Cartan decomposition $\rho(\gamma) = k a \ell \in K\overline{A^+}K$.  Since $q'$ belongs to a precompact set far away from the hyperplane $\ell\inverse\langle e_2, ..., e_{d+1}\rangle$ for all but finitely many elements $\gamma \in \Gamma$ (see Lemma \ref{lem: uniform convergence to boundary map}), we deduce that $\|\rho(\gamma).\Psi(q')\|\asymp e^{\sigma_1(\rho(\gamma))}$.  Thus, as in the previous argument
    \[\|\Psi(q)\|  = \|\rho\chi (\gamma).\Psi(q')\| \asymp e^{\s[p,\gamma.q']}=e^{\s[p,q]},\]
    completing the proof.
\end{proof}

\subsection{Equivariant surfaces in coaffine space}\label{subsec: integrating forms and equivariant surfaces}

We first construct an embedding of $\widetilde M=\Omega  $ into $\mathscr C \Omega$ by way of de Rham cohomology as follows.

As is standard, we consider the flat vector bundle

\[E= E_{\rho\chi} = (\R^{d+1}\times \Omega)/\Gamma\]
where the twisted product is the quotient of $\R^{d+1} \times \widetilde M$ by the diagonal action of $\rho\chi$ and deck transformations. The dual bundle $E^*$ is flat in a compatible way.

Recall that 
\[\Omega^1(E^*) = \Gamma(T^*M\otimes E^*)\]
is the vector space of smooth one-forms valued in $E^*$. 

There is a natural isomorphism induced by the covering projection
\[\Omega^1(\widetilde M, (\R^{d+1})^*)^{\rho \chi} \to \Omega^1(E^*) \]
from $\rho \chi$-equivariant $(\R^{d+1})^*$-valued one-forms on $\widetilde M$ to $E^*$-valued one-forms on $M$. 
\medskip

Consider any closed one-form $\underline \alpha \in \Omega^1 (E^*)$ and let $\alpha \in \Omega^1(\widetilde M, (\R^{d+1})^*)^{\rho \chi}$ be the pullback of $\uu\alpha$ by the covering projection $\widetilde M \to M$. For a pair of points $p,q\in  \Omega\cong \widetilde M$, define
\[A(p,q) = \int_p^q \alpha \in (\R^{d+1})^*.\]
Since $d\alpha = 0$ and $\widetilde M$ is simply connected, this integral depends only on the endpoints, not the path.
\vspace{.5cm}

Choose a basepoint $o\in\Omega$, and let $\alpha \in \Omega^1(E^*)$ be closed.
The map
\[\gamma \mapsto A(o,\gamma.o) \in (\R^{d+1})^*\]
satisfies a cocycle relation, hence defines an element of $Z^1(\Gamma,(\R^{d+1})^*)$.  
The correspondence $\alpha\mapsto A(\cdot , \cdot)$ descends to an isomorphism at the level of cohomology, i.e., 
$H^1_{dR}(M,E^*) \cong H^1(\Gamma, (\R^{d+1})_{\rho\chi}^*),$ see \cite{Brown, Bredon, BT}.

The cocycle condition is as follows:
\begin{align*}
    A(o,\gamma_1 \gamma_2 .o) &=\int_o^{\gamma_1\gamma_2.o}\alpha
    =\int_o^{\gamma_1.o}\alpha
+\int_{\gamma_1.o}^{\gamma_1.(\gamma_2.o)}\alpha
=\int_o^{\gamma_1.o}\alpha+\gamma_1.\int_o^{\gamma_2.o}\alpha\\
&= A(o,\gamma_1.o)+A(o,\gamma_2.o)\rho\chi(\gamma_1)\inverse.
\end{align*}
Here both the action of the fundamental group by deck transformations and precomposition by $(\rho\chi)\inverse$ are denoted by a dot.

Define  $Z(\gamma) = -A(o,\gamma \inverse.o) = A(o, \gamma.o)\rho\chi(\gamma)$ 
and observe by computation that 
\[Z(\gamma_1 \gamma_2)=Z(\gamma_1)\rho\chi(\gamma_2) +Z(\gamma_2)\]
i.e.\ that $\eta=\eta(\rho,\chi,Z)$ is a representation.

We arrive at the desired construction 

\begin{lemma}\label{lemma: eta equivariant embedding exists}
    For any $\eta=\eta(\rho,\chi,Z)$, there exists a continuous $\eta$-equivariant embedding
    \[\Omega\hookrightarrow \R^{d+2}\] 
    given by
    \[p\mapsto
    \begin{pmatrix}
        \Psi(p)\\
        A(o,p)(\Psi(p))
    \end{pmatrix}.\]
\end{lemma}

\begin{proof}
    Firstly the map in the lemma statement is an embedding because $\Psi$ is a section, and continuity follows from the continuity of $A$ and $\Psi$. That its image (after projectivization) is a subset of $\mathscr C\Omega$ follows from the definitions.

    One checks equivariance as follows:

    \[\begin{pmatrix}
        \rho\chi(\gamma) & 0 \\
        -A(o,\gamma\inverse . o) & 1
    \end{pmatrix}
    \begin{pmatrix}
        \Psi(q)\\
        A(o,q)(\Psi(q))        
    \end{pmatrix} = 
    \begin{pmatrix}
        \Psi(\gamma.q)\\
        -A(o,\gamma\inverse.o)\Psi(q)+A(o,q)\Psi(q)
    \end{pmatrix}\]

    Rewriting the last entry of the resultant gives
    \begin{align*}
        \left(-A(o,\gamma\inverse.o)+A(o,q)\right)\Psi(q) &=  \left(-\int_{o}^{\gamma\inverse o}\alpha+\int_o^q \alpha\right)\Psi(q)\\
        &=\left(\gamma\inverse. \int_o^{\gamma.q}\alpha\right)\Psi(q)\\
        &=A(o,\gamma.q)\Psi(\gamma.q)
    \end{align*}
    as required.
\end{proof}

\subsection{Height}\label{subsec: height}

We wish to obtain a notion of height which is essentially a trivialization of the vertical fibers of $\mathscr C \Omega.$ It is induced by a choice of hyperplane transverse to every vertical line,  the chosen norm on $\R^{d+1}$, and a choice of representative $e_{d+2} \in [e_{d+2}]$.

Recall that we have chosen a connected component $\overline{\Omega_0}\subset \R^{d+1}$ of unit norm representatives of the points in $\Omega$ (and its closure).

Let $e^{d+2}$ be the standard dual basis vector to $e_{d+2}\in \R^{d+2}$. This splits
\[\R^{d+2} = \ker[e^{d+2}] \oplus [e_{d+2}]\]
and the quotient $\hat Q : \R^{d+2}\to \R^{d+1}\cong \R^{d+2}/[e_{d+2}]$ is an isomorphism when restricted to $\ker[e^{d+2}]$. By pullback via $\hat Q$, we treat $\|\cdot\|$ as a norm on $\ker[e^{d+2}]$ as well.

Then for each $[v]\in \mathscr C \Omega \cup \partial^\infty \mathscr C \Omega$, $H([v])$ is the unique real number satisfying
\[v_0 + H([v]) e_{d+2} \in [v]\]
where $v_0\in\ker e^{d+2}$ and $v_0 \in \overline{\Omega_0}$. 
We call $H([v])$ the \textbf{height} of the point $[v]$ (with respect to the splitting).
\medskip

A closely associated notion is the \textbf{height} of a point $p\in \overline{\Omega}$ with respect to the equivariant map defined in the \S\ref{subsec: integrating forms and equivariant surfaces} as the real number 
\[h(p) = H\left(\begin{bmatrix}
        \Psi(p)\\
        A(o,p)(\Psi(p))
    \end{bmatrix}\right) = \frac{A(o,p)\Psi(p)}{\|\Psi(p)\|}.\]
We will deal with $h$ so much that it deserves its own name.

The convergence of the function $H$ to $\pm\infty$ reflects convergence of points in $\mathscr C \Omega\cup \partial \mathscr C \Omega$ to $[e_{d+2}]$ from above or below. Likewise, convergence to $\pm\infty$ of the function $h$ evaluated along a ray in $\Omega$ detects when the equivariant section converges to $[e_{d+2}]$ along the ray.

To be precise, observe that if $U$ is a contractible neighborhood of $[e_{d+2}]\in \RP^{d+1}$, the intersection $U\cap \mathscr C \Omega$ has two components. The orientation  on $\Omega$ and the choice of the representative $e_{d+2}\in [e_{d+2}]$ orients the lines in $\mathscr C \Omega$ through $[e_{d+2}]$. Call the connected component of $U$ into which this orientation points $U_-$ and the other $U_+$. We state the following lemma without proof as it is an exercise in projective geometry.

\begin{lemma}
    For any $U$ a contractible neighborhood of $[e_{d+2}]$ in $\RP^{d+1}$,
    the height of a sequence of points $[p_n]$ tends to $\pm\infty$ if and only if the sequence $[p_n]$ is eventually contained in $U_{\pm}$. 
\end{lemma}

\section{Positive cocycles and bounding height when $\underline\lambda$ is a multicurve}\label{sec: heights bounded multicurve}

In this section, we work under hypotheses \ref{hyp A}.
That is: 
\begin{itemize}
    \item $\rho: \Gamma\to \SL(d+1, \R)$ divides a properly convex domain $\Omega\subset \RP^d$ and $d_\rho$ is a choice of domain shape metric on $\Omega$ (see \S\ref{sec: domain shape}).
    \item $\chi = e^{-\varphi}: \Gamma \to \R_{>0}$ has  $\|\varphi\|_\rho=1$  
    \item We choose a $\varphi$-equivariant tight map $\tau: \Omega \to \R$ provided by Theorem \ref{thm: co-stable norm and optimal best maps} with corresponding tension locus $\uu\lambda = \Tension(\tau)/\rho$, which is an oriented chain recurrent geodesic lamination.
    \item We assume that $\uu\lambda$ is a multicurve in $M=\Omega/\rho$ and \[\underline{\lambda}=\cup_{i=1}^m \underline{\delta}_{i},\]
    where $\uu \delta_i$ is an oriented simple closed curve.
\end{itemize}

The goal of the section is to define a certain cone of cohomology classes in $H^1(\Gamma, (\R^{d+1})^*_{\chi \rho})$ called the \emph{positive} cocycles and to show Theorem \ref{thm: heights bounded below scc}, that the smallest nontrivial closed $\eta(\rho, \varphi, Z)$-invariant convex set is bounded below if $[Z]$ is positive.

\subsection{The positive cone in cohomology}\label{subsec: positivity scc}

Choose a number $R$ between $\diam M$ and $2\diam M$, and a 
representative $b_i\in \pi_1 M$ such that $b_i$ is in the conjugacy class defining $\underline \delta_i$ and has axis $\delta_i$ which intersects $B_R(o)\subset \Omega$. 

\begin{remark}\label{rmk: projection from endpoints}
    Given distinct points $y,z \in \partial \Omega$ we denote by $y^* \otimes z\in \End \R^{d+1}$ the projection whose kernel is the unique supporting hyperplane to $\Omega$ at $y$ and whose image is the line $z\subset \RP^d$. Note that there is a unique representative of $y^*$ which evaluates to $1$ on the representative $z_0\in z$. We will choose this representative without comment in what follows.
\end{remark}

Now we observe that positive powers of elements 
\[\beta_i =\rho\chi(b_i) \in \GL(3,\R)\] are projections up to a summable error term.  A similar, but substantially more involved, result is proved in Proposition \ref{prop: projectors and slack} in the setting that $\lambda$ has infinite leaves. We direct the reader to its proof for further details.

\begin{lemma}\label{lemma: beta_i are almost projections}
    There exists a constant $a>0$ such that for any $i=1,\dots, m$, 
    \[\beta_i^n=(\delta_i^-)^* \otimes (\delta_i^+) + O(e^{-an}).\]
\end{lemma}

\begin{proof}
    By assumption, we have that 
    \[\varphi(b_i) = \lambda_1^\rho(b_i)\]
    for all $i$.  
    Thus, the largest eigenvalue of $\beta_i$ is $1$, while the projective-Anosov property of $\rho$ gives that all others are strictly smaller than $1$. The subspaces spanned by
    $\delta_i^+$ and $(\delta_i^-)^*$ define a splitting of the ambient vector space. $\delta_i^+$ is the eigenline corresponding to the eigenvalue $1$ for $\beta_i$ and every vector in $(\delta_i^-)^*$  is uniformly contracted by $\beta_i$ at some exponential rate $a_i>0$.  The powers of such a transformation converge to the projection $(\delta_i^-)^* \otimes (\delta_i^+)$ with an exponentially small error term. Thus the conclusion follows with $a = \min \{a_i\}$.  
\end{proof}

In the following lemma we give several equivalent characterizations of positivity, defined directly after.

\begin{lemma}\label{lemma: positivity lemma}
    Let $\alpha \in \Omega^1(M,E_{\chi \rho}^*)$ be a closed one-form. 
    
    Then the following are equivalent:
    \begin{enumerate}[i)]
    \item There exists $p\in \Omega$ such that
    \[A(p,b_i p)(\delta_i^+)_0>0.\]
    \item For all $p\in \Omega$
    \[A(p,b_i p)(\delta_i^+)_0>0.\]
    \item For all  $p,q\in\Omega$ and $v\in \overline{\Omega}_0 \setminus \{\delta_i^-$\},
    \[\lim_{n\to\infty}\frac{1}{n}A(p,b_i^n q) \beta_i^n(v)>0.\]
    \end{enumerate}
    Furthermore, if $\alpha$ and $\alpha'$ are cohomologous, then the three properties hold for $\alpha$ if and only if they hold for $\alpha'$.  
    Finally, for each $i$, the set of cohomology classes satisfying i)--iii) is an open half space in $H^1(M,E_{\rho\chi}^*)$.
\end{lemma}

Following the Lemma, we define the positive cone in cohomology.

\begin{definition}\label{def: positive cone and cocylces}
    A cohomology class $[\alpha]$ is \textbf{positive} if for any representative $\alpha$, any (and therefore all) of the properties in Lemma \ref{lemma: positivity lemma} hold for all $i=1,\dots, m$. We denote by
    \[\mathscr P_{\rho\chi}\subset H^1(M,E_{\rho\chi}^*)\]
    the set of all positive cohomology classes. It is a convex open cone cut out by $m$ strict linear inequalities.  
\end{definition}

\begin{proof}[Proof of Lemma \ref{lemma: positivity lemma}]
Firstly consider $x,y\in\Omega$, and observe that, because $\alpha$ is closed,
\[A(y,b_i.y) (\delta_i^+)_0-A(x, b_i.x) (\delta_i^+)_0 = A(y,x)(I- (\beta_i)\inverse)(\delta_i^+)_0 = 0 \]
The final equality holds because $\delta_i^+$ is the $1$-eigenline of $\beta_i$. This shows the equivalence of i) and ii).

For the final equivalence, choose a point $x\in\delta_i$ and observe that
\begin{align*}
    \frac{1}{n}A(p,b_i^n q)\beta_i^n(v) &= \frac{1}{n}\left( A(p,x) \beta_i^n+ A(b^n_i.x,b^n_i.q) \beta_i^n  +\sum_{j=0}^{n-1}A(b_i^j.x,b_i^{j+1}.x)\beta_i^n\right)(v)
    \\
    &= \frac{1}{n}
    A(p,x)\beta_i^n(v) + \frac{1}{n}A(x,q)(v)
    + \frac{1}{n}\left(\sum_{j=0}^{n-1}A(x, b_i. x) \beta_i^{n-j}(v)\right)\\
    &= O\left(1/n\right) 
    +\frac{1}{n}\left(\sum_{j=0}^{n-1}  A(x, b_i. x)\circ(\delta_i^+)_0 \otimes (\delta_i^-)^* (v) \right)\\
    &= (\delta_i^-)^*(v)\cdot A(x, b_i. x)(\delta_i^+) + O(1/n).
\end{align*}

The first two equalities are manipulations of the integral and definitions of the objects involved. The third equality follows from Lemma \ref{lemma: beta_i are almost projections} applied twice, and the fourth is a manipulation of linear maps.
The final term tends to zero in the limit as $n\to \infty$. The quantity 
$(\delta_i^-)^*(v)>0$
because $v\in \overline \Omega_0\setminus \delta_i^-$, and $(\delta_i^-)^*$ has kernel a supporting hyperplane of $\Omega$ and evaluates positively on $\Omega$. Therefore the sign of the limit of this quantity is equal to the sign of $A(x, b_i. x)(\delta_i^+)_0$, which shows the last equivalence.

To see that positivity is well-defined on cohomology classes, one supposes that $\alpha' = \alpha +df$, where $df$ is some exact $1$-form. A computation with the integral shows that the quantity $A(x, b_i. x)(\delta_i^+)_0$ is independent of exchanging $\alpha'$ for $\alpha$, again because $\delta_i^+$ is the $1$-eigenspace for $\beta_i$.

The evaluation of $A(p,b_i p)$ on $(\delta_i^+)_0$ is linear, proving the final claim.
\end{proof}

\subsection{Low-slack estimates}

The behavior of the height $h$ along a ray is significantly different when the slack is 'high' (i.e.\ is bounded below by a constant times distance) or 'low' (i.e.\ the ray fellow-travels a line in $\lambda$.) We will clarify this in the following two subsections.

We record an estimate that is useful in the low-slack regime.

\begin{lemma}\label{lemma: low slack linear growth (simple)}
    Let $\alpha$ be a positive cocycle. 
    For all $p,q\in B_R(o)$, and for each $i=1,\dots,m$,
    \[A(p,b_i^n q)\beta_i^n = n  \varrho_i \cdot (\delta_i^-)^* + O(1).\]
    where $\varrho_i = A(o, b_i. o)((\delta_i^+)_0)$.
    
\end{lemma}

\begin{proof}
    For the proof we omit the subscript $i$, and proceed as in the previous lemma.
    \begin{align*}
        A(p,b^n q)\beta^n &=  \left (A(p,o)+ A(b^no, b^nq)\right) \beta^n + \sum_{j=0}^{n-1}A(b^j.o, b^{j+1}.o)\beta^n\\
        &= O(1) + \sum_{j=0}^{n-1}A(o,b.o) \beta^{n-j} \\
        &= O(1) + n A(o,b.o) \circ (\delta^-)^* \otimes (\delta^+)
    \end{align*}
    The first equality is a manipulation of the integral defining $A(\cdot,\cdot)$, the second is true because the form $\alpha$ is bounded on $B_R(o)$, and the operator norm of  $(\beta^n-I)$ is bounded (uniformly) following Lemma \ref{lemma: beta_i are almost projections}. The rest follows from simplifying and again applying Lemma \ref{lemma: beta_i are almost projections}. The quantity in the lemma is a simplification of the composition of the linear functional $A(o, b_i. o)$ with the endomorphism $(\delta_i^-)^* \otimes (\delta_i^+)$.
\end{proof}

The following is a direct corollary and will be applied in the main argument to handle segments of a ray which fellow-travel a line in $ \lambda$.

\begin{lemma}[Low-slack Lemma]\label{lem: main low slack cor}
    For all $p,q\in B_R(o)$, and any (sufficiently small) open sets $\mathcal U\subset \RP^d$ containing $\{(\delta_i)^-\}$, there exists $\varrho>0$ such that for all $v \in \overline \Omega \setminus \mathcal U$ and $n$ large enough, 
    \[\frac{1}{\varrho} n\ge A(p,b^n_i q)\beta_i^n(v_0) \ge  \varrho n.\]
\end{lemma}

\begin{proof}
    It is sufficient to prove the lemma for any $i\in\{1,\dots, m\}$ since this is a finite collection.
    Lemma \ref{lemma: low slack linear growth (simple)} gives that
    \[A(p,b^n_i q)\beta_i^n(v_0) = n \varrho_i (\delta_i^-)^*(v_0) +O(1).\]
    Because $\mathcal U$ contains $\delta_i^-$ there are uniform bounds on the evaluation $(\delta_i^-)^*(v_0)$ (recall that $\Omega$ is strictly convex and $(\delta_i^-)^*$ is a supporting hyperplane at $\delta_i^-$). The result follows
    \end{proof}

\subsection{Notation and high slack}

The following Lemma is a manipulation which occurs frequently, so we package it here for convenience; see Figure \ref{fig: notation lemma }.

\begin{lemma}[Notation Lemma]\label{lemma: notation lemma}
    Let $r:[0,T]\to \Omega$ be a geodesic ray with $r(0)\in B_R(o)$, and consider 
    \[0=t_0<\dots<t_i<\dots<t_N =T.\]
    Then there exist $\gamma_i\in \Gamma$ and $p_i\in B_R(o)$ such that for each $i=0,\dots, N$
    \[(\gamma_1\cdots\gamma_{i}).p_i=q_i \defeq r(t_i).\]

    Furthermore, 

    \[A(q_0,q_N)\rho\chi (\gamma_1\cdots \gamma_N) = \sum_{i=0}^{N-1}A(p_i,\gamma_{i+1} p_{i+1})\rho\chi (\gamma_{i+1}\cdots\gamma_N)\]
    and
    \[A(q_0,q_N)\Psi(q_N) = \sum_{i=0}^{N-1}A(p_i,\gamma_{i+1} p_{i+1})\rho\chi (\gamma_{i+1}\cdots\gamma_N)\Psi(p_N).\]
\end{lemma}

\begin{figure}[h]
    \centering
    \includegraphics[width=.6\linewidth]{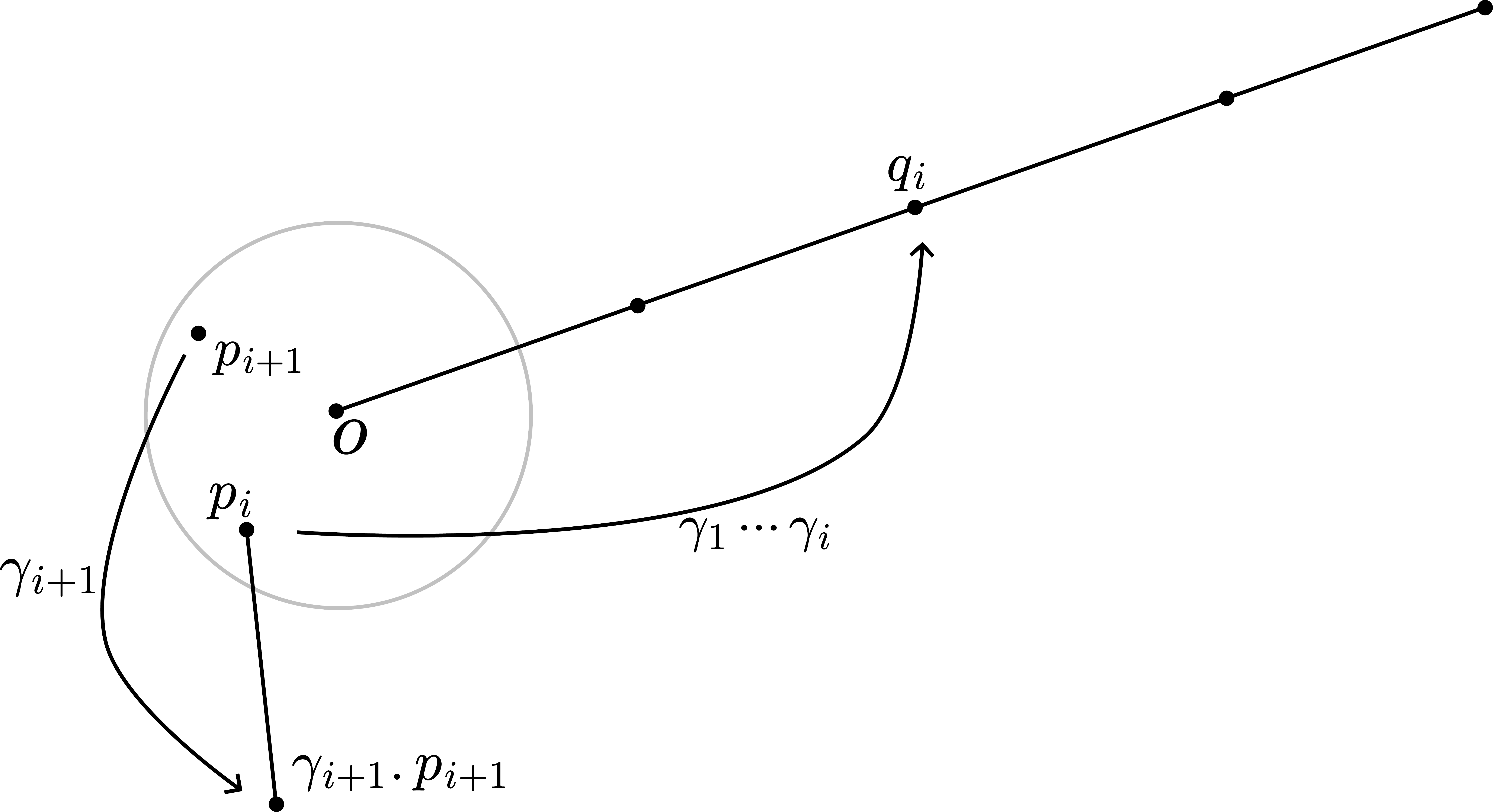}
    \caption{Setup of the Notation Lemma \ref{lemma: notation lemma}}  
    \label{fig: notation lemma }
\end{figure}

We leave the Notation Lemma without proof as it is a manipulation of the definitions and properties of the integral $A$ and the map $\Psi$.
\medskip

To handle segments of the ray $r$ with slack which is at least proportional to their length, we will make use of the following general fact.

\begin{lemma}\label{lemma: geometric series bound for high slack}
    Suppose $a_0,\dots, a_N$ are real numbers and that there exists some $\varepsilon>0$ satisfying $a_i>\epsilon$ for all $i$. Then 
    \[\sum_{i=0}^{N-1}\exp (\sum_{j=0}^i a_j) < \frac{e^\varepsilon}{e^\varepsilon-1}\exp (\sum_{j=0}^N a_j) \]
\end{lemma}

\begin{proof}
    One examines the quantity
    \[\exp (\sum_{j=0}^N -a_j)\sum_{i=0}^{N-1}\exp (\sum_{j=0}^i a_j) = \sum_{i=0}^{N-1} \exp(-\sum_{j=i+1}^N a_j)\]
    and applies standard facts about geometric series.
\end{proof}

The following Lemma will be used to handle high-slack segments.

\begin{lemma}[High-slack Lemma]\label{lemma: high-slack lemma}
Given $\epsilon>0$, $R>0$, and $\alpha$, there is a $C>0$ such that the following holds:

Suppose that $[a,b]\subset \Omega$ is a geodesic segment with a partition into $N$-subsegments defined by their endpoints $q_i$, where $q_0=a$ and $q_N=b$. Choose $\gamma\in \Gamma$ satisfying $\gamma.p_N=b$ for some point $p_N\in B_R(o)$. 
Suppose that $\s [q_i,q_{i+1}]>\epsilon$ and that $d_\rho(q_i,q_{i+1})<R$ for all $i$. 
Then
\[\|A(a,b) \rho\chi (\gamma) \|\le C e^{\s[a,b]}.\]
\end{lemma}

\begin{proof}
    Using notation from Lemma \ref{lemma: notation lemma}, we write
    \[A(q_0,q_N)\rho\chi (\gamma_1\cdots \gamma_N) = \sum_{i=0}^{N-1}A(p_i,\gamma_{i+1} p_{i+1})\rho\chi (\gamma_{i+1}\cdots\gamma_N)\]

    Note that $\|A(p_i,\gamma_{i+1} p_{i+1})\|$ is uniformly bounded above by the assumption on $d(q_i,q_{i+1})$.
    Finally apply Lemma \ref{lemma: basic slack estimates} to see that
    \[\|\rho\chi (\gamma_{i+1}\cdots\gamma_N)\|\le  const \cdot  e^{\s[q_{i+1},q_N]}\]
    and Lemma \ref{lemma: geometric series bound for high slack} using the assumption that $\s [q_i,q_{i+1}]>\epsilon$. This completes the lemma.
\end{proof}

\subsection{The height along a ray}

In the previous subsection we saw that low-slack segments of a ray exhibit positive linear growth, and that high-slack segments exhibit behavior like a geometric series. In this subsection we find a partition of the ray into high- and low-slack segments and apply both estimates to uniformly bound the height function  from below along the ray for cocycles in the positive cone.

By a partition of a ray $r$ we mean a countable subset of the domain and the complementary subsets (considered up to their endpoints).

\begin{lemma}[Partition Lemma]\label{lemma: partition lemma multicurve case}
    Given $N\ge 1$ there is an $\epsilon_N>0$ such that the following holds.
    Let $r:[0,\infty)\to \Omega$ be a geodesic ray. Then there exists a countable partition of $r$ into non-trivial subsegments labeled $H$ and $L$ (`high' and `low') such that 
    \begin{itemize}
        \item There exist $\ell_0>0$ and $\vartheta>0$ such that every $H$-subsegment $J$ satisfies $\ell(J)>\ell_0$, and for every subsegment $[a,b]\subset J$ with $b\ge a+\ell_0/4$, it holds that $\s[r(a), r(b)]\ge\vartheta (b-a)$.
        \item For every $L$-subsegment $J=[a,b]$, there exists $i$ such that $r(J)$ is contained in $\mathcal N_{\epsilon_N} (\delta_i)$ and 
        there are $p,q \in B_R(o)$ and $\gamma \in \Gamma$ such that
        \[r(a) =\gamma.p \text{ and } r(b) = \gamma b_i^n.q\]
        where $n\ge N$ for some  $i \in \{1, ..., m\} $.
    \end{itemize}
\end{lemma}

In the present case, it is not difficult to find a suitable partition of the ray. The analogue when the tension locus is not a multicurve, (Lemma \ref{lem: partition lemma general}) is more involved.

\begin{proof}
    Choose $\epsilon_N$ such that any vector $v \in T^1M$ that is $\epsilon_N$-close to $T^1_+\underline \lambda$ wraps around the corresponding component of $\underline \lambda$ at least $N$ times under the forward geodesic flow before exiting this neighborhood in $T^1M$.
    Additionally choose $\epsilon_N$ to ensure that $\epsilon_N$ neighborhoods of the components of $\uu\lambda$ have disjoint closure.  Thus there is a uniform lower bound $\ell_0>0$ such that any geodesic ray connecting these tubes has length at least $\ell_0$.  
    
    Define $H$-subsegments to be those maximal subintervals $J$ for which $\dot r(J)\subset N_{\epsilon_N}(T^1_+\underline \lambda)^c$, and $L$-subintervals to be the connected components of $[0,\infty)\setminus \cup \{J\colon \textrm{$J$ an $H$-subinterval}\}.$ The required property of $L$-subintervals is clear due to the choice of $\epsilon_N$, and Lemma \ref{lem: slack grows} (with $\ell_0/4$ playing the role of $\delta$) shows the required linear slack growth property for subsegments of $H$-subsegments.
\end{proof}

We are now ready to prove the main result of this section that the height of any ray is uniformly bounded from below for a given cocycle in the positive cone.

\begin{theorem}[Bounded Height]\label{thm: heights bounded below scc}
    Given $\alpha\in \mathscr P_{\rho\chi}$, there is a constant $C(\alpha)>0$ such that for all geodesic rays $r: [0,\infty) \to \Omega$ with $r(0) =o$ and all $T\ge 0$, we have 
    \[h(r(T))\ge -C(\alpha).\]
\end{theorem}

\begin{proof}
    Recall that 
    \[h(r(T)) = \frac{A(o,r(T))\Psi(r(T))}{\|\Psi(r(T))\|}.\]
    Using Lemma \ref{lemma: basic slack estimates}, we have 
    \[h(r(T))\succeq \frac{A(o,r(T))\Psi(r(T))}{e^{\s[o,r(T)]}}.\]
    Our goal for the proof will be to control the size of this quantity.
    
    Choose $N$ large enough so that Lemma \ref{lem: main low slack cor} holds for all $n\ge N$, and apply the Partition Lemma \ref{lemma: partition lemma multicurve case} for this $N$ to the ray $r$. The Partition Lemma gives us a sequence $ 0=t^0<t^1 < ...$, (where it's possible that the $t^i$'s are $\infty$ for all $i\ge i_0 \ge 1$).  Given $T>0$, there is a largest $n_T$ such that $t^{n_T-1}\le T$.  We set 
    \[q_i = \begin{cases} r(t^i), & 0\le i <n_T,\\
    r(T), & i = n_T.
    \end{cases}\]
    Now, we apply the Notation Lemma \ref{lemma: notation lemma} to obtain $p_i \in B_R(o)$ and $\gamma_i \in \Gamma$ such that 
    \[q_i = \gamma_1 \cdots\gamma_i.p_i,\]
    and
    \begin{align*}
        A(o,r(T))\Psi(r(T))= \sum_{i = 0}^{n_T-1} A(p_i,\gamma_{i+1}p_{i+1})\rho\chi(\gamma_{i+1})\rho\chi(\gamma_{i+2}\cdots\gamma_{n_T})\Psi(p_{n_T}).
    \end{align*}

    We may assume that the $H$-subsegments from the Partition Lemma have even index and the $L$-subsegments have odd index and handle even and odd terms separately.
    The $L$-segments correspond to geodesic segments of the form $[q_i, q_{i+1}]$ where $i$ is odd, and each one also comes with an integer $M_i\ge N$;
     we may assume that 
    \begin{equation}\label{eqn:M_i defining property}
        \gamma_{i+1} =b_{j_i}^{M_i},
    \end{equation}
    where $j_i \in \{1, ..., m\}$.
    \medskip

    \noindent\textbf{1) Even terms and high-slack estimates:}    Consider an even term, $i\ge2$ corresponding to an $H$-segment.
    We can cut $[p_i, \gamma_{i+1}p_{i+1}]$ into pieces whose length is in $[\ell_0/4, \ell_0/2]$.  As a consequence of the Partition Lemma \ref{lemma: partition lemma multicurve case}, the slack of each piece is bounded below by $\vartheta \ell_0/4$.  Using Lemma \ref{lemma: high-slack lemma}, we have 
    \[(*) \defeq \left| A(p_i,\gamma_{i+1} p_{i+1})\rho\chi(\gamma_{i+1})\rho\chi(\gamma_{i+2}\cdots\gamma_{n_T})\Psi(p_{n_T}) \right| \prec   e^{\s[q_i, q_{i+1}]}\| \rho\chi(\gamma_{i+2}\cdots\gamma_{n_T})\| \cdot \|\Psi(p_{n_T})\|.\]
    Since $p_{n_T} \in B_R(o)$ and $\Psi$ is continuous, we know that $\|\Psi(p_{n_T})\|$ is of size $O(1)$, while Lemma \ref{lemma: basic slack estimates} gives $\| \rho\chi(\gamma_{i+2}\cdots\gamma_{n_T})\|\asymp e^{\s [q_{i+1}, r(T)]}$.  Using additivity of slack, we therefore obtain the bound 
    $(*)\prec e^{\s[q_i, r(T)]}$.

    Summing all of the even terms and taking the absolute value, we thus obtain
    \begin{align*}
    \left| \frac{1}{e^{\s[o,r(T)]}} \sum_{i \text{ even} } A(p_i,\gamma_{i+1}p_{i+1})\rho\chi(\gamma_{i+1})\rho\chi(\gamma_{i+2}\cdots\gamma_{n_T})\Psi(p_{n_T})\right| 
   &\prec  \frac{1}{e^{\s[o,r(T)]}}\sum_{i \text{ even} } e^{\s[q_{i},r(T)]}\\
   & = \sum_{i \text{ even} } e^{-\s[o,q_i]}=O(1) 
   \end{align*}
    In the last inequality, we observe that the lower bounds on the slack of $H$-segments bound the sum in terms of a geometric series.  Indeed, the Partition Lemma (\ref{lemma: partition lemma multicurve case}) furnishes positive constants $\ell_0, \vartheta$ such that
    \[\s[0,q_i]\ge i\ell_0\vartheta/2,\]
    since the segment $[0,q_i]$ contains $(i/2)$ $H$-segments by construction. This proves the claim.
    \medskip

    \noindent\textbf{2) Odd terms and low-slack estimates:}
    Consider an odd term
    \[A(p_i,\gamma_{i+1} p_{i+1})\rho\chi(\gamma_{i+1})\rho\chi(\gamma_{i+2}\cdots\gamma_{n_T})\Psi(p_{n_T})\]
    corresponding to an $L$ segment.
    
    We claim that $\rho\chi(\gamma_{i+2}\cdots\gamma_{n_T})\Psi(p_{n_T})$ has size $ \asymp e^{\s[q_{i+1}, r(T)]}$.
    The size estimate is given by Lemma \ref{lemma: basic slack estimates} after observing that $\Psi(p_{n_T})$ lies in a compact set and is thus bounded away from the hyperplane corresponding to small singular values of any sufficiently large non-trivial element of $\Gamma$.

    Additionally, we claim that $[\rho\chi(\gamma_{i+2}\cdots\gamma_{n_T})\Psi(p_{n_T})]$ lies far away from a neighborhood of $\cup_{i=1}^m\delta_i^-$.
    This follows from  our choice of $\epsilon_N$ from the Partition Lemma \ref{lemma: partition lemma multicurve case}:
    the tangent vector to $r$ at the end of this $L$-segment is still close to the corresponding line in $T^1_+\lambda$, and so fellow travels with that line for some time.  This bounds the position in $\Omega$ far away from any of the backward endpoints of the lines $\cup_{i=1}^m \delta_i$ (Figure \ref{fig: bounded_ht_fig1}).

    \begin{figure}[h]
    \centering
    \includegraphics[width=.35\linewidth]{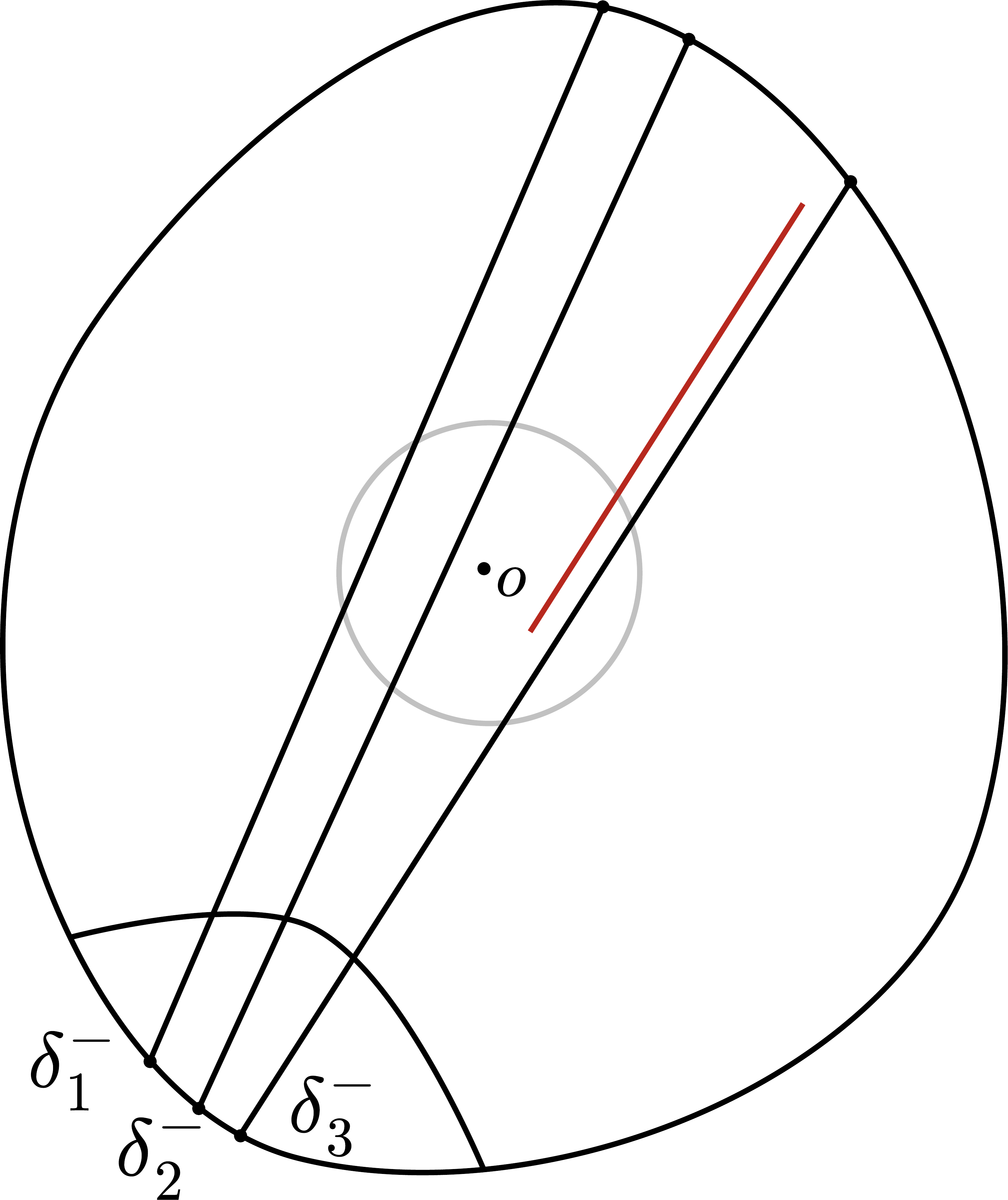}
    \caption{A ray based in $B_R(o)$ fellow-traveling the curves $\delta_i$ has forward endpoint bounded away from $\delta_i^-$.}  
    \label{fig: bounded_ht_fig1}
    \end{figure}
    
    Given the observations made in the previous two paragraphs, the Low Slack Lemma \ref{lem: main low slack cor}, gives
    \[\frac{1}{\rho}M_i\succ A(p_i,\gamma_{i+1} p_{i+1})\rho\chi(\gamma_{i+1})\left(\frac{\rho\chi(\gamma_{i+2}\cdots\gamma_{n_T})\Psi(p_{n_T})}{e^{\s[q_{i+1}, r(T)]}}\right)  \succ \varrho M_i >0. \]

    Summing all of the odd terms, we thus obtain
    \begin{align*}
     \frac{1}{e^{\s[0,r(T)]}} \sum_{i \text{ odd} } A(p_i,\gamma_{i+1}p_{i+1})\rho\chi(\gamma_{i+1})\rho\chi(\gamma_{i+2}\cdots\gamma_{n_T})\Psi(p_{n_T}) 
     &\asymp  
     \frac{1}{e^{\s[0,r(T)]}}\sum_{i \text{ odd} } M_i e^{\s[q_{i+1},r(T)]}\\
    & = 
    \sum_{i \text{ odd} } M_i e^{-\s[0,q_{i+1}]},
    \end{align*}
    where we have used the additivity of slack in the equality.  
    \medskip

    \noindent\textbf{3) Combining the estimates:} Putting it together gives
    \begin{equation}\label{eqn: upper and lower height bound}
        \left(O(1)+ const \sum_{i \text{ odd} }^{n_T-1} M_i e^{-\s[0,q_{i+1}]}\right) 
    \ge
    h(r(T))\succ\frac{A(o,r(T))\Psi(r(T))}{e^{\s[o,r(T)]}}
    \ge
    \left(O(1)+ const' \sum_{i \text{ odd} }^{n_T-1} M_i e^{-\s[0,q_{i+1}]}\right),
    \end{equation} 
    where the constants are positive and don't depend on $T$ or $r$.    
    Since each odd term is positive, we thus achieve a uniform lower bound on height for every ray $r$, proving the theorem.
\end{proof}

\begin{remark}\label{rmk: upper height bound}
    We will use the uniform upper bound from \eqref{eqn: upper and lower height bound} in Proposition \ref{prop: average slack and height} to analyze the average slack of rays with unbounded height.  
\end{remark}

\begin{remark}\label{rmk: lots of rays going to infinity}
    It follows from \eqref{eqn: upper and lower height bound} that a ray that has infinitely many $H$-subsegments with sufficiently long $L$-segments between (in fact, longer and longer) them will accumulate infinite height. 
\end{remark}

\section{The positive cone in dimension two}\label{sec: positive cone lamination}

In this section, we essentially work with hypotheses \ref{hyp B}, i.e., $d=2$ and $\uu\lambda$ is not necessarily a multicurve.  
We define a notion of positivity for elements of $H^1(\Gamma, (\R^3)_{\rho\chi}^*)$, which is significantly more involved than in the case that $\Tension(\varphi)$ is a multicurve.  
After defining positivity, we prove an analogue of the Low Slack Lemma, which is one of the main new ingredients needed in \S\ref{sec: heights bounded below general} to prove that heights are bounded below for cocycles in the positive cone.  
\medskip

The hypotheses of and notation for this section are:
\begin{itemize}
    \item $\Gamma = \pi_1S$ is the fundamental group of a closed, oriented surface $S$ of genus at least $2$.  
    \item $\rho: \Gamma \to \SL(3,\R)$ is a Hitchin representation, dividing a strictly and properly convex domain $\Omega\subset \RP^2$. 
    \item $\varphi \in H^1(S,\R)$ has $\|\varphi\|_\rho=1$ and we write $\chi = e^{-\varphi}$.
\end{itemize}

We also make a choice of basepoint $o \in \Omega$, a domain shape metric $d_\rho$ on $\Omega$ (see \S\ref{sec: domain shape}), and an $R\ge \diam(\Omega/\rho)$, so that 
$B_R(o) = \vec B_R(o)\cap \cev B_R(o)$ contains a fundamental domain for the action of $\rho$ on $\Omega$.
The natural projection is denoted $\pi: \Omega \to \Omega/\rho$.

\begin{remark}\label{rmk: why not higher dimensions}
    Many of the tools in the next two sections work just as well in some higher-dimensional contexts.  The only two features of dimension $2$ that we use are (1) that the minimal sublaminations are isolated from one another and (2) that there are only finitely many leaves that do not belong to minimal sublaminations.  The proofs go through unchanged in higher dimensions under these hypotheses. 

    There are examples of hyperbolic manifolds in dimension $3$ where the tension locus corresponding to an equivariant tight map is known to be a multicurve \cite{Rudd}. We are not aware of any hyperbolic examples where the hypotheses from the previous paragraph do not hold. 
\end{remark}

\subsection{Laminations and dynamics}\label{subsec: laminations and dynamics}

Using Theorem \ref{thm: co-stable norm and optimal best maps}, there is a $\varphi$-equivariant $1$-forward Lipschitz map 
\[\tau: \Omega \to \R\]
with optimal tension locus.
That is, $\Tension(\tau) =\Tension(\varphi) \subset \Omega$ is $\rho$-invariant and projects to an oriented, chain recurrent geodesic lamination 
\[\underline \lambda \defeq \pi(\Tension(\tau))\subset \Omega/\rho.\]

The structure theory for chain recurrent geodesic laminations on closed (negatively curved) surfaces gives that 
\[ \underline \lambda =  \underline \delta_1 \cup ... \cup  \underline \delta_{m} \cup  \underline \mu_1 \cup ...  \underline \mu_{m'} \cup  \underline \ell_1 \cup ...  \underline \ell_{m''},\]
where, for each $i$, $ \underline \delta_i$ is a simple closed geodesic, $ \underline \mu_i$ is a minimal geodesic lamination with an infinite leaf, and $ \underline \ell_i$ is an infinite isolated leaf that is forward and backward asymptotic to minimal sublaminations of $ \underline \lambda$, and all are (continuously) oriented.
\medskip

Let $\underline  \mu \in \{\underline  \mu_1, ...,  \underline \mu_{m'}\}$.
The natural map $\underline  \mu \to T^1\Omega/\rho$ taking a point to the tangent vector pointing in the direction of the orientation of $\underline  \mu$ is a homeomorphism onto its image, on which the geodesic flow $\{\underline g_t\}$ is defined and continuous.  By abuse of notation, the geodesic flow is thus defined on $\underline  \mu$.

\begin{definition}
    An oriented geodesic segment $\underline  k$ transverse to the leaves of $\underline  \mu$ is called a \textbf{train track transversal} if $\underline  k$ meets the leaves of $\underline  \mu$ transversely with a consistent orientation, does not meet any other minimal component of $\underline  \lambda$, and if the length of $ \underline k$ is smaller than the injectivity radius of $\Omega/\rho$.
\end{definition}
Let $\underline k$ be a train track transversal to $\underline \mu$, and let $\underline X = \underline \mu \cap \underline  k$.  Let $\cT: \underline X \to \R_{>0}$ be the first return time for the geodesic flow, i.e.,
\[\cT(\underline x) = \min\{ t> 0 : \underline g_t \underline x \in  \underline X\}.\]
By continuity of the geodesic flow, minimality of $\underline \mu$, and compactness of $\underline X$, $\cT$ is a positive continuous function bounded away from $0$ and $\infty$. 
Thus, the first return map
\begin{align*}
     \underline \sigma :  \underline X &\to  \underline X\\
    \underline x&\mapsto\underline  g_{\cT(\underline x)}\underline x
\end{align*}
is a homeomorphism.  The $n$th return time is $\cT^n(\underline x) = \cT(\underline \sigma^{n-1}\underline x) + ...+ \cT(\underline x)$. For all $\underline x \in  \underline X$,  we have $\uu g_{\cT^n(\underline x)}(\underline x) = \underline\sigma^n\underline x$.
\medskip

Now, we lift the situation to $\Omega$, as follows.
Let $\mu = \pi\inverse(\underline \mu)$, choose a lift $k \subset B_R(o)$ of $\underline k$, and let $X = k\cap \mu$.  Let $\sigma: X \to X$ be defined by lifting the dynamics from $\underline X$, i.e., $\sigma(x)\defeq \pi\inverse (\underline \sigma\pi(x))\cap X$. 
The geodesic flow on $T^1\Omega$ and $\mu$ at time $t$ is denoted by $g_t$.  For $x \in X$, $\cT(x) = \cT(\underline x)$ is just the first return time to the translates of $X$ by $\rho$. 

For $x\in X$ and $n\in \N$ define
    \[G^n(x) \defeq g_{\cT^n(x)}(x),\]
which is an element of $\mu$ (or more precisely $T^1_+\mu$).

The following lemma follows from the properties of train track transversal $k$ and the definitions; its proof is left to the reader.
\begin{lemma}\label{lemma: train track branches from returns}
    For all $x \in X$, there is a unique $b(x) \in \pi_1S$ such that $\beta(x) \defeq \rho\chi(b(x))$ satisfies
    \[G^1(x)= \beta(x)\sigma x.\]
    Moreover, $b(x)\not=id$.
    For $n\ge1$ $b^n(x) \defeq b(x) b(\sigma x) \dots b(\sigma^{n-1}x)$ and $\beta^n(x) \defeq \beta(x)\beta(\sigma x) \dots \beta(\sigma^{n-1}x) = \rho\chi(b^n(x))$.
    Then, for all $n\ge1$, we have
    \[G^n x = \beta^n(x) \sigma^nx.\]
\end{lemma}

Note that the following cocycle relation holds
\[\beta^n(x) = \beta^{i}(x) \beta^{n-i}(\sigma^ix)\]
for all $0\le i\le n$, if we define $\beta^0(x) = id$.

\subsection{Projections and slack}

Now we identify the behavior of $\beta^n(x)$.
For a point $x \in X$, $\langle x_+\rangle \subset \R^3$ corresponds to $\lim_{t\to \infty} g_tx \in \partial^\infty \Omega$, while $\langle x_-^*\rangle \subset (\R^3)^*$ has $[\ker x_-^*]$ tangent to $\partial \Omega $ at $\lim_{t\to -\infty}g_tx$.
Recall from \S\ref{sec: slack} that for two oriented leaves $\ell, \ell' \subset \Tension (\tau) \subset \Omega$, if $[\ell_-, \ell_{+}']$ is the directed line backward asymptotic to $\ell$ and forward asymptotic to $\ell'$, that $\s[\ell_-,\ell_+'] <+\infty$.

\begin{lemma}\label{lem: continuity of Pi}
    The function 
    \begin{align*}
        \Pi: X \times X & \to \End(\R^3)\\
        (x,y)&\mapsto  e^{\s[x_-,y_+]}x_-^*\otimes y_+
    \end{align*}
    is continuous.  Here,  $x_-^*(y_+) =1$, i.e., $x_-^* \otimes y_+$ is the projection onto the line $\langle y_+\rangle$ with kernel equal to the plane $\ker x_-^*$; see Remark \ref{rmk: projection from endpoints}. 
\end{lemma}
\begin{proof}
    By continuity of the geodesic flow, the maps $x \mapsto \langle x_-^*\rangle $ and $y\mapsto \langle y_+\rangle $ are continuous.  So the map $(x,y) \mapsto x_-^*\otimes y_+$ is also continuous.  
    From \S\ref{sec: slack}, $(x,y)\mapsto \s[x_-,y_+]$ is finite and continuous (Lemma \ref{lem: slack finite and continuous}).  Combining these ingredients gives that $\Pi$ is continuous.
\end{proof}

The following generalizes Lemma \ref{lemma: beta_i are almost projections}, but the slack of long closed loops fellow traveling the tension locus enters the computation in a nontrivial way that did not appear before.

 \begin{proposition}\label{prop: projectors and slack}
    There is a constant $a>0$ such that for all $x \in X$ we have
    \[\beta^n(x) = \Pi(\sigma^nx, x) +O(e^{-an}) \in \End(\R^3),\]
     The implicit multiplicative constant is independent of $x$.  
\end{proposition}

\begin{proof}
    We claim that,  near $k$, the axis of $\beta^n(x)$ is \emph{exponentially close}  to the geodesic line $\ell_n = [(\sigma^nx)_-,x_+]$.  More precisely, we claim the following.
\begin{claim}\label{clm: geodesic segments converge exponentially}
    There is a $b>0$ such that the following holds.  For each $x\in X$ and $n \ge1$, there is a point $p_n$ on the axis of $\beta^n(x)$ such that the segment $[p_n, \beta^n(x)p_n]$ passes through $k$ and 
\[d_\rho(p_n, \ell_n)= O(e^{-bn}) \text{ and } d_\rho(\beta^n(x)p_n,\ell_n)=O(e^{-bn}),\]
and 
\[\min\{d_\rho(p_n,k), d_\rho(\beta^n(x)p_n, k) , d_\rho(k,p_n), d_\rho(k, \beta^n(x)p_n)\} \ge bn.\]
\end{claim}

\begin{proof}[Proof of Claim \ref{clm: geodesic segments converge exponentially}]
    Theorem \ref{thm: flow Anosov} gives that the geodesic flow on $T^1\Omega/\rho$ for the domain shape metric is an Anosov flow.
    The claim then follows from applying the Anosov Closing Lemma (see \cite[\S6.4.c]{KH:dynamicalsystems}) --- the closed orbit of the geodesic flow obtained by closing up $\pi(g_{[0, \cT^n(x)]}x)$ using a path over $\uu k$ fellow travels the trajectory of $\uu x$ for a long time going forward and the backward trajectory of $\underline{\sigma^nx}$ in the past for a long time.  
\end{proof}

Now, $\ell_n$ is asymptotic to $[x,x_+]$ in the future and asymptotic to $[(\sigma^nx)_-,\sigma^nx]$ in the past.
Apply Lemmas \ref{lem: slack continuous finite} and \ref{lem: slack finite and continuous}  to see that 
\begin{equation}\label{eqn: approximate slack of leaf}
    \s(\ell_n) = \s[(\sigma^nx)_-,x_+] = \s[p_n, \beta^n(x)p_n] +O(e^{-bn}).
\end{equation}
Since $[p_n,\beta^n(x)p_n]$ is a fundamental domain for the action of $\beta^n(x)$, the  slack of the segment $[p_n, \beta^n(x)p_n]$ is just the logarithm of the largest eigenvalue of $\beta^n(x)$, i.e., 
\begin{equation}\label{eqn: slack is eigenvalue difference}
    \s[p_n, \beta^n(x)p_n] = (\lambda_1^\rho - \varphi)(b^n(x)).
\end{equation}
This quantity is uniformly bounded, by basic properties of slack (\S\ref{sec: slack}), becuase the axis of $\beta^n(x)$ spends most of its time near a leaf of $\lambda$.
\medskip

Let $P_n$ be projection onto $\beta^n(x)_+$ parallel to $\beta^n(x)_-^*$.
Since $X_+$ and $X_-$ are compact subsets of $\partial ^\infty \Omega$ that do not intersect, Lemma \ref{lem: rank one projectors and Anosov} and the Anosov property of $\rho\chi$ gives us a $c>0$ such that 
    \[\frac{1}{e^{(\lambda_1^\rho-\varphi)(b^n(x))}}\beta^n(x) =  P_n +O(e^{-cn}),\]
where we have used that  $|b_n|\asymp n$.
Since $(\lambda_1^\rho-\varphi)(b^n(x))$ is uniformly bounded, we get
\begin{equation}\label{eqn: beta and projector}
    \beta^n(x) =  e^{(\lambda_1^\rho-\varphi)(b^n(x))}P_n +O(e^{-cn}).
\end{equation}

On the other hand, a consequence of Claim \ref{clm: geodesic segments converge exponentially} and Lemma \ref{lem: uniform convergence to boundary map} is that
\begin{equation}\label{eqn: projector comes from a leaf}
    P_n = (\sigma^nx)_-^* \otimes x_+ +O(e^{-dn})
\end{equation}
for some $d>0$.

Combining \eqref{eqn: approximate slack of leaf}, \eqref{eqn: slack is eigenvalue difference}, \eqref{eqn: beta and projector},  and \eqref{eqn: projector comes from a leaf} gives the Proposition.
\end{proof}

\subsection{Integrating along a leaf of $\mu$}
The goal of this subsection is to define \emph{positivity} for a closed $1$-form $\underline \alpha \in \Omega^1(E_{\rho\chi}^*)$. Let $\alpha =\pi^*\alpha$ and recall (\S\ref{subsec: integrating forms and equivariant surfaces}) that for $p,q\in \Omega$ 
\[A(p,q) \defeq \int_p^q\alpha \in (\R^3)^*\]
satisfies the equivariance relation $A(\gamma.p,\gamma.q) = \rho\chi(\gamma).A(p,q) = A(p,q) \rho\chi(\gamma)\inverse$.
We will analyze the asymptotic behavior of 
\begin{equation}\label{eqn: averages}
    \frac1nA(x,G^nx)\beta^n(x) \in (\R^3)^*.
\end{equation}
Define 
\begin{align*}
    L: X &\to (\R^3)^*\\
    x & \mapsto A(x, G(x)).
\end{align*}
Since $\alpha$ is closed and smooth, the function $(p,q)\mapsto A(p,q)$ is continuous on $\Omega\times \Omega$.  So by continuity of $\cT$ and the geodesic flow, $L$ is also continuous.  

The following is a straightforward computation that follows from Lemma \ref{lemma: train track branches from returns} and the cocycle property of $A$.  The proof is left to the reader.
\begin{lemma}\label{lem: notation lemma laminations}
    Let $x \in X$ and $n\ge 1$. 
    Then
    \[A(x,G^nx)\beta^n(x) = \sum_{i =0}^{n-1}L(\sigma^ix) \beta^{n-i}(\sigma^ix).\]
\end{lemma}

Using Lemma \ref{lem: notation lemma laminations}, \eqref{eqn: averages} is starting to look like an average of some function along an orbit of $\sigma$, and we are ready to analyze it once again.

\begin{lemma}\label{lem: averages revisited}
    For all $x \in X$ and $n\ge 1$, we have 
    \[\frac1nA(x,G^nx)\beta^n(x) = \frac1n\left(\sum_{i=0}^{n-1} L(\sigma^ix)\Pi(\sigma^nx,\sigma^ix)
    \right)+O(1/n).\]
    The implicit multiplicative constant is independent of $x$.  
\end{lemma}
\begin{proof}
    Since $L$ is continuous and $X$ is compact, we have $\|L\|_\infty<\infty$.  Using Proposition \ref{prop: projectors and slack} and the definitions, we have for $i = 0, ..., n-1$
    \[ \beta^{n-i}(\sigma^{i}x) = \Pi(\sigma^nx, \sigma^ix) +O(e^{-b_0(n-i)}).\]
    Since $(x,y) \mapsto \s[x_-,y_+]<\infty$ is continuous (Lemma \ref{lem: slack finite and continuous}) and $X \times X$ is compact, we have $\|\Pi\|_\infty<\infty$.  
    Then 
    \[L(\sigma^ix) \beta^{n-i}(\sigma^ix) = L(\sigma^ix)\Pi(\sigma^nx, \sigma^ix) +O\left(\|\Pi\|_\infty \|L\|_\infty e^{-b_0(n-i)}\right).\]
    Summing over $i$ and dividing by $n$ gives the desired bound, as $\|\Pi\|_\infty  \|L\|_\infty\sum_{i=0}^\infty e^{-b_0i}<\infty$.
\end{proof}

There is a non-empty, finite dimensional simplex $\Delta(\sigma)$ of $\sigma$-invariant Borel probability measures on $X$.  For $\nu \in \Delta(\sigma)$ define
\begin{align*}
     L_\nu: X &\to (\R^3)^*\\
    x&\mapsto \int_X L(y) \Pi(x,y)~d\nu(y). 
\end{align*}

    For each $x\in X$, the linear functional $L_\nu$ is a $\nu$-weighted average of projections post-composed with $L_y$.

\begin{lemma}\label{lem: integral jointly continuous}
    The function 
    \[(x,\nu) \in X\times \Delta(\sigma) \mapsto L_\nu(x)\in (\R^3)^*\]
    is continuous.  Consequently, 
    \[ \mathcal L\defeq \{ L_\nu (x) : x \in X \text{ and } \nu \in \Delta(\sigma)\}\] is compact in $(\R^3)^*$.
\end{lemma}
\begin{proof}
    Assume $x_n \to x$ in $X$ as $n\to \infty$.
    Using continuity of all functions involved in the definitions and compactness of the domain, it is an exercise in elementary analysis that the functions $f_n(y)= L(y)\Pi(x_n, y)$ converge uniformly to $f(y)= L(y)\Pi(x,y)$ as $n\to \infty$. 

    For weak-$*$ convergent Borel probability measures $\nu_n \to \nu$, and uniformly convergent functions $f_n \to f$, it is not difficult to show that $\int f_n ~d\nu_n \to \int f~d\nu$ as $n\to \infty$.  It follows that the function $(x,\nu)\mapsto L_\nu(x)$ is continuous.
    The second statement now follows from compactness of the space $X\times \Delta(\sigma)$.
\end{proof}

For a sequence $(f_n)$, let $\omega\text{-}\lim f_n$ denote the set of its accumulation points.
\begin{proposition}\label{prop: limit points of averages}
    For all $x \in X$, we have 
    \[\omega\text{-}\lim  \frac1n A(x, G^nx)\beta^n(x)\subset\mathcal L\]
    \emph{uniformly}.
    In other words, for any neighborhood $U$ of $\mathcal L$ in $(\R^3)^*$ there is an $N$ such that for all $x \in X$ and $n\ge N$, 
    \[\frac1n A(x, G^nx)\beta^n(x)\in U.\] 
    Moreover, $\ker L_\nu(y)\supset \ker y_-^*$ for all $y \in X$ and $\nu \in \Delta(\sigma)$.  
\end{proposition}

\begin{proof}
    By Lemma \ref{lem: averages revisited}, it suffices to understand the limit points of 
    \[ \frac1n\sum_{i=0}^{n-1} L(\sigma^ix)\Pi(\sigma^nx,\sigma^ix).\]

    Suppose for sake of contradiction that for all $N$ there are $n_N\ge N$ and $x_N\in X$ such that 
    \[\frac1{n_N} \sum_{i =0}^{n_N-1} L(\sigma^ix_N)\Pi(\sigma^{n_N}x_N, \sigma^ix_N) \not\in U.\]
    Passing to a subsequence (which we do not rename), we can assume by compactness of $X$ and $\Delta(\sigma)$ that 
    \[\sigma^{n_N}x_N \to y \in X ~\text{ and } ~ \frac 1{n_N} \sum_{i =0}^{n_N-1}\delta_{\sigma^ix_N} \to \nu \in \Delta(\sigma), \text{ as } N \to \infty.\]
    Using uniform continuity of $\Pi$ and boundedness of $\|L\|$, we have 
    \[
        \left\| \frac{1}{n_N}\sum_{i=0}^{n_N-1} L(\sigma^ix_N)\Pi(\sigma^{n_N}x_N, \sigma^ix_N) - \frac{1}{n_N}\sum_{i=0}^{n_N-1} L(\sigma^ix_N)\Pi(y, \sigma^ix_N) \right\| \to 0, \text{ as } N \to \infty.
    \]
    By continuity, the weak-$*$ convergence gives
    \[ \left\| \frac{1}{n_N}\sum_{i=0}^{n_N-1} L(\sigma^ix_N)\Pi(y, \sigma^ix_N) - \int L(x)\Pi(y,x)~d\nu(x) \right\| \to 0. \]
    By the triangle inequality, we conclude
    \[\left\| \frac1{n_N} \sum_{i =0}^{n_N-1} L(\sigma^ix_N)\Pi(\sigma^{n_N}x_N, \sigma^ix_N) - \int L(x) \Pi(y,x)~d\nu(x) \right\| \to 0, \text{ as } N \to \infty.\]
    This is a contradiction, because the integral is just $L_\nu(y) \in \mathcal L\subset U$.
    Thus the uniformity is proved.

    To see that $\ker L_\nu(y)\supset \ker y_-^*$, we simply observe that 
    \[L_\nu(y) = \int_X L(x) \Pi(y,x)~d\nu(x),\]
    and that $\ker \Pi(y,x) = \ker y_-^*$ for all $x \in X$ since $\Pi(y,x)$ is proportional to projection to $\langle x_+\rangle$ parallel to $y_-^*$, verifying the assertion.  
\end{proof}

An immediate corollary of the proof is the following
\begin{corollary}\label{cor: limit points of averages}
    Let $x \in X$ be generic for a $\sigma$-invariant ergodic probability measure $\nu \in \Delta(\sigma)$, and let $y \in X$.  
    For any sequence $n_j$ tending to infinity such that $\sigma^{n_j}x \to y$, we have 
    \[\lim_{j\to \infty}\frac{1}{n_j} A(x,G^{n_j}x)\beta^{n_j}(x) =L_\nu(y).\]
\end{corollary}

\subsection{Defining the positive cone}

We are finally able to give the general notion of positivity for the minimal sublamination $\underline \mu \subset \underline \lambda$.

\begin{definition}\label{def: positive for minimal}
    For a closed $1$-form $\uu\alpha \in \Omega^1(E_{\rho\chi }^*)$ and minimal sublamination $\underline \mu \in \{\underline \mu_1, ..., \underline \mu_{m'}\}$ with an infinite leaf, $\uu \alpha$ is $\underline \mu$-\emph{positive} if 
    \[L_\nu(x)(p_0)>0\]
    for all $x \in X$, for all $\nu \in \Delta(\sigma)$, for some/any $p \in \Omega$.
\end{definition}

That the definition is agnostic to the choice of $p\in \Omega$ follows from the fact that either  $\ker L_\nu(x) \cap \Omega = \emptyset$ or $L_\nu(x)\equiv  0$ by Proposition \ref{prop: limit points of averages}.

This definition descends cleanly to cohomology classes.
\begin{lemma}\label{lem: mu-positivity is well defined}
    A closed $1$-form $\uu\alpha\in \Omega^1(E_{\rho\chi }^*)$ is $\underline \mu$-positive if and only if the class $[\uu \alpha]\in H^1(S,E_{\rho\chi }^*)$ consists entirely of $1$-forms that are $\underline \mu$-positive.  
\end{lemma}

\begin{proof}
    Let $f: \Omega \to (\R^3)^*$ be a $\rho\chi$-equivariant smooth function, i.e., $f\in \Omega^0(E_{\rho\chi}^*)$, so that $\alpha +df \in [\alpha]$.  Define 
    \begin{align*}
        A_f(p,q) = \int_p^q\alpha+df = A(p,q) +f(q) -f(p)
    \end{align*}
    It suffices to show that 
    \[\left\|\frac1n (A(x,G^nx) - A_f(x,G^nx))\beta^n(x)\right\| \to 0\]
    as $n\to \infty$, independently of $x$.
    We compute 
    \begin{align*}
         \left(A(x,G^nx) - A_f(x,G^nx)\right)\beta^n(x)  &=\left(f(x) - f(G^nx)\right)\beta^n(x)\\
         &=\left(f(x) -f(\beta^n(x)\sigma^nx)\right)\beta^n(x)\\
         &= f(x)\beta^n(x) - f(\sigma^nx).
    \end{align*}
    Since $X$ is compact, $\|f|_X\|_\infty<\infty$.  By Lemma \ref{lemma: basic slack estimates}, we know that $\|\beta^n(x)\| \preceq e^{\s[x, G^nx]} =1$, since the directed segment $[x, G^nx]$ lies in $\Tension(\tau)$.
    Thus $\|f(\sigma^nx)-f(x)\beta^n(x)\|<\infty$, and so dividing by $n$ makes this quantity tend to zero.  Thus positivity is an invariant of the cohomology class.
    \end{proof}

We now have the desired general definition of positivity.

\begin{definition}\label{def: positive cone}
    The \textbf{positive cone} $\mathscr P_{\rho\chi}\subset H^1(S, E_{\rho\chi}^*)$ is defined as the intersection of the positive loci for all minimal sublaminations as in Definitions \ref{def: positive for minimal} and \ref{def: positive cone and cocylces}.
\end{definition}

\begin{remark}
    It can be shown that the notion of positivity does not depend on the other choices we made for the definition, i.e., the choice of train track transversal $\underline k$ to $\underline \mu$ and the lift $k$ to $\Omega$. After \S\ref{sec: prox is orb}, it also follows that the positive cocycles are exactly those for which the limit set $\Lambda(\eta)$ is bounded below and for which this property is open amongst cocycles. That is, the positive cone is identifiable by the resulting boundedness condition and by the stability of this condition for such cocycles. 
\end{remark}

\begin{remark}
    At this point, the authors would like to point out that it is not clear whether or not $\mathscr P_{\rho\chi}$ is non-empty, when $\uu\lambda$ contains a minimal component with an infinite leaf.  
\end{remark}

\subsection{Linear growth along leaves in the positive cone}
Here is the analogue of the Low Slack Lemma \ref{lem: main low slack cor} for positive $1$-forms.
\begin{lemma}[Basic Low Slack Lemma]\label{lem: low slack minimal}
    Suppose $\uu\alpha$ is $\underline \mu$-positive, and let all notation be as above.  Let $X_-\subset \partial \Omega$ be the collection of backward endpoints of leaves of $\mu$ in $X$, and let $\mathcal U$ be a small open neighborhood of $X_-$ in $\overline \Omega$. There is a positive integer $N$ and a $\varrho>0$ such that if $n \ge N$ and $x \in X$, then 
    \[\frac{1}{\varrho}n\ge A(x, G^nx)\beta^n(x)v_0 \ge \varrho n\]
    for all $v \in \overline\Omega\setminus \mathcal U$.  Recall that $v_0 \in [v]$ is the unit vector in $\Omega_0 \subset \{\|\cdot \| =1\}$.  
\end{lemma}

\begin{proof}
    Let 
    \[\varrho = \frac 12 \min \left\{ L_\nu(y)(v_0): v \in \overline \Omega \setminus \mathcal U \text{ and } L_\nu(y) \in \mathcal L\right\}\]
    Positivity of $\alpha$ implies that $\varrho>0$.  

    By Proposition \ref{prop: limit points of averages},  $\{\frac1n A(x,G^nx)\beta^n(x) : n\ge 1 \text{ and } x\in X\}$ all approach $\mathcal L$ uniformly, so that  there is an $N$ such that for $n\ge N$, we have 
    \[\max_{x\in X, n\ge N}\min_{L_\nu(y) \in \mathcal L}\left\|\frac1nA(x,G^nx)\beta^n(x) - L_\nu(y)\right\|<\varrho. \]
    In other words, for all $x$ and $n\ge N$, there is a pair $y \in X$ and $\nu \in \Delta(\sigma)$ such that
    \[\left\|\frac1nA(x,G^nx)\beta^n(x) - L_\nu(y)\right\|<\varrho.\]
    By the previous inequality, and because $ L_\nu(y)(v_0)\ge 2\varrho$ where $v_0$ is a unit vector, we have 
    \[\varrho n= (2\varrho-\varrho)n\le nL_\nu(y)v_0-n\varrho\le A(x,G^nx)\beta^n(x)v_0.\]
    The upper bound is obtained similarly.
    This is what we wanted to show. 
\end{proof}

\section{Heights bounded from below in the positive cone}\label{sec: heights bounded below general}

In this section, our hypotheses are identical to those of the previous \S\ref{sec: positive cone lamination}.  The main result is Theorem \ref{thm: heights bounded below general}, which is analogous to Theorem \ref{thm: heights bounded below scc}.

\subsection{Banner systems}
Recall our notation that 
\[ \underline \lambda =  \underline \delta_1 \cup ... \cup  \underline \delta_{m} \cup  \underline \mu_{m+1} \cup ...  \underline \mu_{m+m'} \cup  \underline \ell_{m+m'+1} \cup ...  \underline \ell_{m+m'+m''},\]
where, $ \underline \delta_i$ is a simple closed geodesic, $ \underline \mu_i$ is a minimal geodesic lamination with an infinite leaf, and $ \underline \ell_i$ is an infinite isolated leaf that is forward and backward asymptotic to minimal sublaminations of $ \underline \lambda$, and all are (continuously) oriented.

There is a definite positive distance $d_0$ between each of the minimal components $\{\underline \delta_i, \underline\mu_{j}\}_{i,j}$ of $\underline \lambda$.  
Choose a system $\{\underline k_i\}_{i = 1}^{m+m'}$ of train track transversals such that each $\underline k_i$ is contained in the $d_0/3$-neighborhood of the corresponding minimal component.
Let $\epsilon_0 = 4\max \diam k_i$.
Now, for $i = 1, ..,, m$, let $\underline X_i =\underline k_i \cap \underline \delta_i$, for $i = m+1, ..., m+m'$, define $\underline X_i = \underline k_i\cap \underline \mu_i$, and define $\underline Y_i = \underline k_i \cap \underline \lambda$.  Note that $\underline X_i\subset \underline Y_i$ and there is equality when $\uu\lambda$ has no isolated infinite leaves.

Recall that $\pi: \Omega \to \Omega/\rho$ denotes the natural projection by taking $\Gamma$ orbits and that $B_R(o)\subset \Omega$ projects onto all of $\Omega/\rho$.
For each $i$, choose a component $k_i$ of $\pi\inverse (\uu k_i) \cap B_R(o)$.  Define $X_i = k_i \cap \pi\inverse(\uu X_i)$, $Y_i = k_i\cap \pi\inverse(\uu Y_i)$,
\[X = \bigcup_{i = 1}^{m+m'} X_i, \text{  } Y = \bigcup_{i = 1}^{m+m'} Y_i, \text{ and } k = \bigcup_{i=1}^{m+m'}k_i.\]

As $\lambda$ is oriented, it makes sense to consider each $X_i$ and $Y_i$ (as well as $\underline X$ and $\underline Y$) as a subset of $T^1\Omega$ (respectively $T^1\Omega/\rho$).
Note that for $i =1, ..., m$, $ X_i$ consists of a single point, while for $i = m+1, ..., m+m'$, $X_i$ is a Cantor set.

Abusing notation, we also use $\pi$ for the map $T^1\Omega \to T^1\Omega/\rho$ taking orbits via the $\rho$-action.  
Let $\proj: T^1\Omega/\rho \to \Omega/\rho$ be the  tangent projection --- we do not distinguish between the tangent projections $T^1\Omega \to \Omega$ and $T^1\Omega/\rho \to \Omega/\rho$.  Then $\pi$ commutes with $\proj$, i.e., $\pi \circ \proj (v) =\proj\circ\pi(v)$ for $v \in T^1\Omega$.
We also choose a Finsler metric $d$ on $T^1\Omega$ that agrees with $d_\rho$ on paths that lift directed geodesics.

For a set $Z \subset T^1\Omega$, denote by $Z_+$ the set of points in $\partial \Omega$ obtained as forward ideal endpoints of directed lines determined by points in $Z$, and let $Z_-$ be the set of backward ideal endpoints.

\begin{definition}\label{def: banner}
    A \textbf{banner} $ K_i \subset T^1\Omega$ over $k_i$ is a relatively open neighborhood of $Y_i$ in $\proj\inverse(k_i)$ that is transverse to the geodesic foliation of $T^1\Omega$.
    In addition, we require that $\overline{(K_i)_-}\cap \overline{(K_i)_+} = \emptyset$, i.e., the forward and backward endpoints at infinity have closures which don't intersect.

    A \textbf{complete banner system} $K = \bigcup_{i = 1}^{m+m'} K_i$ is a union of banners that additionally satisfies
    $\overline {(K_i)_-} \cap \overline{(K_j)_+} = \emptyset$ for all $i$ and $j$.
\end{definition}
We denote by $K_{\geq 0}$ the set
\[K_{\geq 0}=\overline{\{\proj(g_t v) \in \Omega \colon v\in K \textrm{ and } t\geq 0\}}\subset \overline \Omega.\]

In practice, we will be interested in banners that are small metric neighborhoods of $Y_i$ in $\proj\inverse(k_i)$.  Since there is definite separation between the minimal components of $\uu\lambda$ in $\Omega/\rho$, and the isolated leaves of $\uu\lambda$ are forward and backward asymptotic to minimal components of $\uu\lambda$, it is easy to see that such a small enough neighborhood of $Y$ in $\proj\inverse(k)$ satisfies the (technical) endpoint property.  See Figure \ref{fig: banner}

\begin{figure}[h]
    \centering
    \includegraphics[width=.5\linewidth]{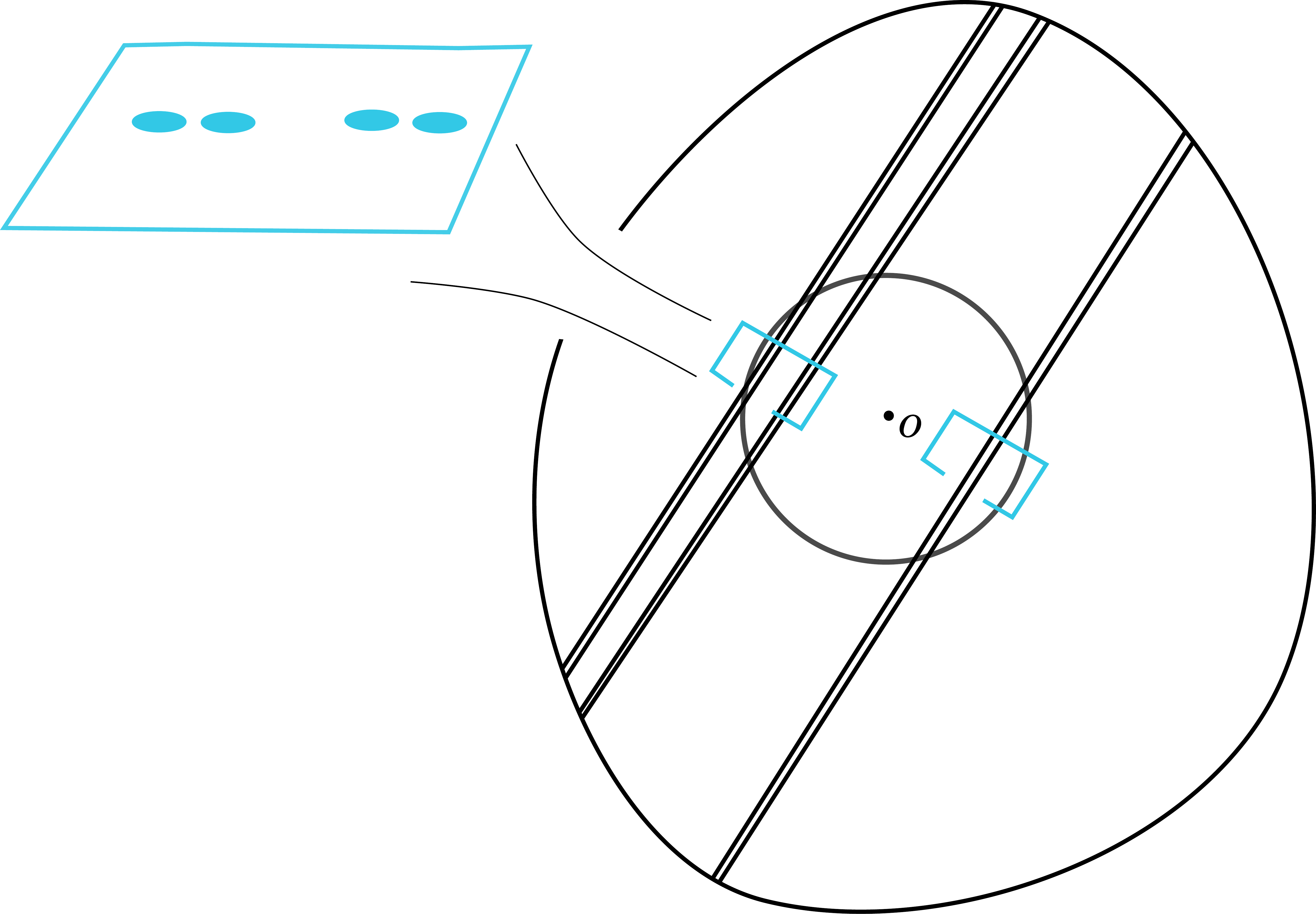}
    \caption{A schematic of a complete banner system.  The blue disks constitute $K$ and lie in the fiber over $k$.}  
    \label{fig: banner}
    \end{figure}
\medskip

From now on, we fix a complete banner system $K$, and let $\underline K = \pi(K) \subset T^1 \Omega/\rho$.  
For $\underline v \in \underline K$, let 
\[\cT^n(\underline v) \in (0, \infty]\]
denote the $n$th return time to $\underline K$.\footnote{Clearly $\cT$ (as well as $G$, $b$, and $\sigma$) depend on $K$, although the notation does not reflect that.  }
For each $\underline v \in \underline K$ and $n\ge 1$ with $\cT^n(\underline v)<\infty$, define $\sigma^n\underline v = \underline g_{\cT^n(\underline v)}\underline v  \in K$ and define $b^n(v)\in \Gamma$ by 
\[G^nv\defeq g_{\cT^n(v)}v = b^n(v).\sigma^n v.\]
Notationally, in the previous line, $v \in K$ satisfies $\pi(v) = \underline v$, we take $\cT^n(v) = \cT^n(\underline v)$, and $\sigma^nv\in K$ defined by the property that $\pi(\sigma^nv) = \sigma^n\underline v$.
Note that for $x \in X$, our definitions of $\sigma$, $b$, and $G$ extend our earlier definitions given in \S\ref{subsec: laminations and dynamics}. As before we take $\beta^n(v) \defeq \rho\chi(b^n(v))$ as well.

Note the following cocycle relation (compare with Lemma \ref{lemma: train track branches from returns}), which holds for all $v\in K$:
\begin{equation}\label{eqn: cocycle general}
    b^{n+m}(v) = b^n(v)b^m(\sigma^nv).
\end{equation}

\begin{lemma}\label{lem: forward cones nest}
    Given a complete banner system $K\subset  T^1\Omega$, there is an $N_0\ge1$ such that for all $y \in Y$ and $n\ge N_0$, we have
    \[\beta^n(y)K_{\ge0}\subset K_{\ge 0 }.\]
\end{lemma}

\begin{proof}
    The proof of Proposition \ref{prop: projectors and slack} implies that there exists a positive number $M$ such that for all $m\geq M$ and all $v\in Y$, $\beta^m(v)_\pm \in K_\pm\subset \partial \Omega$. 
    
    Furthermore, because of the lower bound on return times to $K$, the translation lengths of $\beta^m(v)$ have a uniform positive lower bound. 
    The strict separation of $K_-$ and $K_+$, implies that the compact set $K_{\geq 0}\subset \overline \Omega$ is uniformly bounded away from the repelling hyperplanes of $\beta^m(v)$, when $m\geq M$. This result follows.
\end{proof}

Now we construct banner systems for each $N\ge N_0$ that get smaller as $N$ gets larger, where $N_0$ comes from Lemma \ref{lem: forward cones nest}.

\begin{lemma}\label{lem: construct N banner}
    For every $N\ge1$, there is a complete banner system $K_N\subset K$ such that for every point $v \in K_N$, there is a $y \in Y$ such that $b^n(v) = b^n(y)$ for all $1\le n \le N$.  
\end{lemma}
\begin{proof}
    Construct $K_N\subset K$ as follows. 
    By continuity of $g_t$, given $ y \in Y$, every point in a small enough neighborhood $U(N, y) \subset K$ has the property that 
\[b^n(v) = b^n(y)~ \text{for all $1\le n \le N$,}\]
and $d(G^Nv, G^Ny)\le\epsilon_0/2$.  
In other words, we ask that $v$ is so close to $y$ that they meet the same sequence of transversals in $T^1\Omega$ for at least $N$ combinatorial steps in the future.  
We take $K_N$ to be the union over $y \in Y$ of these neighborhoods $U(N,y)$, which is a subsurface of $K$, hence a complete banner system. 
\end{proof}

\subsection{Partition Lemma and high slack estimates}\label{subsec: partition and high slack surface}

We now consider a geodesic ray $r:[0,\infty)\to \Omega$ with $r(0)=o\in\Omega$. 
We produce a partition of the ray $r$ into high- and low-slack subsegments to which we will apply separate estimates in Theorem \ref{thm: heights bounded below general}.

By a partition of $r$ we mean a countable subset of the domain and the complementary subsets (considered up to their endpoints). In the lemma we require that the basepoint of the ray is not contained in the system of transversals, which is purely for simplicity.

\begin{lemma}[Partition Lemma]\label{lem: partition lemma general}
    There is a positive constant $T_H$ such that the following holds.  
    
    Given $N\ge N_0$, obtain a complete banner system $K_N\subset K$ as in Lemma  \ref{lem: construct N banner}, and let $\uu K_N = \pi(K_N)$.  
    Let $r:[0,\infty)\to \Omega$ be a geodesic ray with $r(0)=o\in\Omega\setminus \mathcal N_\epsilon(\uu \lambda)$.
    Then there exists a partition $0=t^0< t^1< t^2<\ldots$ of the domain of $r$ into non-trivial subsegments colored $H$ and $L$ such that the following properties hold:
    \begin{itemize}
        \item Let $v_i =\dot r (t^i)$ and $\uu v_i = \pi(\dot r(t^i))$.  If $[t^i,t^{i+1}]$ is an $L$-subsegment, then there exists a number
        $M_i\geq 1$  such that 
        \[t^{i+1}= t^i + \cT^{M_iN}(v_i) \]
        and for all $0\le m\leq M_i-1,$
        \[\uu g_{\cT ^{mN}(\uu v_i)}(\uu v_i) \in \uu K_N.\]
        \item If $[t^i,t^{i+1}]$ is an $H$-subsegment, then either $t^i=0$ or $\uu v_i\in \uu K\setminus \uu K_N$,  $ \uu v_{i+1}\in \uu K_N$, and $\pi (\dot r(t^i,t^{i+1}))\cap \uu K_N=\emptyset$. Furthermore, $t^{i+1} - t^i\ge T_H$.
    \end{itemize}
\end{lemma}

The first item says that the initial point $\uu v_i$ of a segment colored  $L$ is close enough to $\uu\lambda$ after its $mN$-th return to $\uu K$ that it fellow travels with a single leaf of $\uu\lambda$ for the next $N$ steps, where $m = 0 ,..., M_i-1$.  The most important consequence of the second item is that  every segment colored with an $H$ has definite length and stays a definite distance away from  $\lambda$.  

\begin{proof}
    We describe an algorithm and prove that the partition it produces satisfies the conditions of the lemma; compare with the proof of Lemma \ref{lemma: partition lemma multicurve case}.

    Let $t^1$ be the first time that the ray in the surface intersects $\uu K_N$, and call the subsegment $[0,t^1]$ an $H$-subsegment. If no such number exists, the partition is trivial and we declare $[0,\infty)$ to be an $H$-subsegment.

    Otherwise, proceed following the algorithm
    \begin{enumerate}
        \item If $\uu v_i \in \uu K_N$, let $M_i$ be the largest integer such that for all $m=1,\dots M_i-1$, the vector $g_{\cT^{mN}(\uu v_i)}(\uu v_i)\in \uu K_N$ and define $t^{i+1}=t^i+\cT^{M_iN}(v_i)$. The subsegment $[t^i,t^{i+1}]$ is an $L$-subsegment.
        \item  If $\uu v_i \in \uu K\setminus \uu K_N$, define $t^{i+1}$ to be the smallest number greater than $t^i$ satisfying $\pi(\dot r(t^{i+1}))\in \uu K_N$. The subsegment $[t^i,t^{i+1}]$ is an $H$-subsegment.
    \end{enumerate}
    It is clear that the conditions in the lemma are exclusive, so the domain of $r$ is two-colored by $H$ and $L$. It also is immediate that the property demanded of the $L$-subsegments is satisfied, as are the first conditions required of the $H$-subsegments. To see the $H$ subsegments have length with a uniform lower bound, observe, for instance, that there is a minimum return time from $\uu K$ to itself. The minimum of this bound and the number $t^1$ suffices. 
\end{proof}

In what follows, we choose a positive $\delta_0$ smaller than $T_H$ from the Partition Lemma \ref{lem: partition lemma general}.

\begin{lemma}\label{lem: slack grows in H segments}
    Given $N\ge N_0$, there is a $\vartheta>0$ such that the following holds.  
    Given a geodesic ray $r: [0,\infty) \to \Omega$ with $r(0)=o$ and an $H$-subsegment $[t^i, t^{i+1}]$ obtained from the Partition Lemma \ref{lem: partition lemma general}, for all $a,b \in [t^i, t^{i+1}]$ with $b\ge a+\delta_0$, we have 
    \[\s[r(a), r(b)]\ge\vartheta (b-a).\]
\end{lemma}

\begin{proof}
    Let $[t^i,t^{i+1}]$ be an $H$-segment, and let $T_i = t^{i+1}-t^i \le +\infty$.  Then the following are satisfied:
    \begin{itemize}
        \item  $\uu v_i = \pi (\dot r(t^i)) \in \uu K\setminus \uu K_N$,
        \item $\uu v_{i+1} = \pi(\dot r(t^{i+1})) \in \uu K_N$, and 
        \item $\uu g_t \uu v_i\not \in \uu K_N$ for all $t \in [0, T_i)$.
    \end{itemize}

    \begin{claim}\label{clm: stays away from the lamination}
        There is a positive $\epsilon>0$ (depending on $N$) such that $[v_i,g_{T_i}v_i]\cap \mathcal N_\epsilon(T^1_+\lambda) = \emptyset$.
    \end{claim}

\begin{proof}[Proof of Claim \ref{clm: stays away from the lamination}]

Using compactness of $T_+^1\uu \lambda$, continuity of $\uu g_t$, and construction of $\uu K_N$, we can find a positive $L>0$ such that for any point $\uu v $  sufficiently close (within distance $\epsilon>0$, say) to $\uu\lambda$, we have that 
    \begin{itemize}
        \item $\uu g_{[-L, 0]}\uu v \cap \uu K_N\not=\emptyset$, and 
        \item if $ t\in [0,L]$ is such that $\uu g_{-t}\uu v \in \uu K$, then also $\uu g_{-t}\uu v \in \uu K_N$.
    \end{itemize}

    For $\uu v_i$, suppose $t\in [0,T_i)$ satisfies that $\uu g_t\uu v_i\in \mathcal N_\epsilon(T_+^1\uu \lambda)$.  If $t<L$, then $\uu v_i \in \uu K_N$, which is a contradiction.  Otherwise, $\uu g_{t'} \uu v_i \in \uu K_N$, where $t' \in [0,T_i]$, which is also a contradiction.  This completes the proof of the claim. 
\end{proof}
Giving $\epsilon$ from Claim \ref{clm: stays away from the lamination} and our choice of positive $\delta_0 /2$ smaller than both $T_H/2$ and $R$ to Lemma \ref{lem: slack grows} provides a $\vartheta>0$ such that $\s[r(a),r(b)]\ge \vartheta(b-a)$, as $b-a\ge \delta_0$.
\end{proof}

\subsection{Low slack estimates}

The goal of this subsection is to construct a small enough complete banner system, i.e., find a large enough $N$, such that all $L$-segments from the Partition Lemma \ref{lem: partition lemma general} contribute positively to the change in height along a ray for cocycles in the positive cone $\mathscr P_{\rho\chi}$.  

We proceed in two steps.  First, we observe that the positivity condition on each minimal component of $\uu\lambda$ contributes uniformly positive growth of the cocycle $A(\cdot, \cdot)$ along segments that are (small perturbations of) sufficiently long segments of leaves of $\uu \lambda$.  This step is entirely straightforward, except for a minor technical complication introduced by the isolated leaves of $\uu\lambda$ and is carried out in Lemma \ref{lem: low slack general case}.

Second, we need to account for the behavior of $L$-segments that do not spend their entire life following one leaf of $\lambda$, but rather shadow a sequence of leaves, each for a long time, and make tiny ``jumps'' between the leaves.  We show in Lemma \ref{lem: zig zag} that such $L$-segments have a linear upper bound on the growth of $A$, but only a constant positive lower bound, in general (if $N$ is chosen large enough).  This is enough for the proof of Theorem \ref{thm: heights bounded below general}, but is not fine enough for our application in \S\ref{sec: measure}.  See Remark \ref{rmk: why multicurve}.

Both of these steps can be seen as applications of the two Basic Low Slack Lemmas \ref{lem: low slack minimal} and \ref{lem: main low slack cor}.

\begin{lemma}\label{lem: low slack general case}
    Given any complete banner system $K$ and positive class $[\alpha] \in \mathscr P_{\rho\chi}$, there exists a positive $\varrho'$ such that the following holds.  For any $N$ sufficiently large, the complete banner system $K_N$ from Lemma \ref{lem: construct N banner} satisfies $v \in K_N$ implies $\cT^N(v)<\infty$ and 
    \[ \frac{1}{\varrho'}N\ge A(v, G^Nv)\beta^N(v)(g_Tw)_0\ge \varrho' N\]
    for all $w \in K$ and $T\ge0$.
    Moreover, $0\le \s[v,G^N v]\le \epsilon_0$. 
\end{lemma}

The proof is rather technical due to the nature of the objects involved.  We would like to emphasize, however, that we have done most of the hard work in Lemmas \ref{lem: main low slack cor} and \ref{lem: low slack minimal} already, and we are essentially just putting these pieces together to get uniform estimates.

\begin{proof}

By construction, for each $v \in K_N$, there is a $y \in Y$ such that $d(v, y)\le\epsilon_0/4 $ and $d(\proj((G^Nv)), \proj (G^Ny))\le\epsilon_0/4$.  
Using Lemma \ref{lem: slack continuous finite}, we thus obtain 
\[0\le \s[v, G^N v]=\s[v, G^N v] - \s[y,G^N y]\le 2d(v,y) + 2d(G^N v,G^N y)\le  \epsilon_0,\]
where we have used that $\s[y,g_{\cT^N(v)}y] = 0$, because the segment $[y,g_{\cT^N(v)}y]$ is contained in $T^1_+\uu\lambda$.
This is the final assertion of the lemma. 

\medskip

Consider geodesic segments $\alpha_1$ joining $\proj (y)$ to $\proj (v)$ and $\alpha_2$ joining $\proj (\sigma^N v)$ to $\proj(\sigma^Ny)$.  The lengths of both $\alpha_1$ and $\alpha_2$ are bounded by $\epsilon_0/2$.

By path independence of the integral and because $b^N(y) = b^N(v)$, we have 
\[A(y, G^N y)\beta^N(y)  = \left( \int_{\alpha_1}\alpha + A(v, G^N v) + \int_{b^N(v)\alpha_2}\alpha \right) \beta^N(v),\]
thus 
\[\left\|  A(y, G^N y)\beta^N(y) - A(v, G^N v) \beta^N(v) \right\| \le \left\|  \int_{\alpha_1}\alpha\right\| \left \| \beta^N(y) \right\|  + \left\| \int_{\alpha_2}\alpha\right\|. \]
By Lemma \ref{lemma: basic slack estimates} and because $\s[y, g_T(y)] = 0$ for all $T\ge 0$, there is a constant $C\ge 1$, such that
$\|\beta^n(y)\|\le C$, for all $n \ge 0$.
Also, since $\alpha_1$ and $\alpha_2$ lie in a compact set in $\Omega$, we have 
$\left\| \int_{\alpha_i}\alpha\right\|=O(1)$ for $i = 1,2$.
Thus 
\[A(y, G^Ny)\beta^N y  = A(v, G^Nv) \beta^N(v)+O(1).\]
Note that the $O(1)$ in the above expression is independent of $N$.

\medskip

Our goal is now to bound $A(y,G^N y)\beta^N(y)(g_Tw)_0$ from above and below, for any $w \in K$ and $T\ge0$.

There are three cases to consider: $\uu y$ is tangent to a closed curve $\uu \delta \subset \uu \lambda$, $\uu y$ is tangent to an infinite leaf of a minimal component $\uu \mu \subset \uu \lambda$, or $\uu y$ is tangent to an isolated leaf of $\uu \lambda$.  

The first two cases are essentially the same: we invoke Lemmas \ref{lem: main low slack cor} and \ref{lem: low slack minimal}. 
To do so, consider the set $\cup_{v \in K}\{g_{-t}v : t \ge 0\}$; its closure in $\overline \Omega$ contains an open neighborhood $\mathcal U$ of $Y_-$. 
By definition of the complete banner system $K$, for any $w \in K$, we know that $g_Tw \not \in \mathcal U$ for all $T\ge 0$.  
Thus, for each $i$, the hypotheses of Lemmas \ref{lem: main low slack cor} and \ref{lem: low slack minimal} are satisfied for $y \in K_i$: this gives us a $\varrho_i>0$ and $N_i$ large such that 
\[ \frac{1}{\varrho_i}N\ge A(y,G^Ny)\beta^N(y)(g_Tw)_0\ge \varrho_iN,\]
for all $N\ge N_i$, $w\in K$, and  $T\ge 0$.   
Taking $N_0'$ to be the maximum of the $N_i$'s and $\varrho'$ to be the minimum of the $\varrho_i$'s, the conclusion of the lemma holds in the first two cases.
\medskip

Now we sketch why the isolated leaves pose no additional difficulties (up to changing the constants $N_0'$ and $\varrho'$ obtained from the previous paragraph).  
Let $\uu \ell$ be (the forward tangents to) such an isolated leaf.  There is a unique component $\uu l$ of $\uu \ell \setminus \uu K$ that does not fellow travel a segment of a leaf of a minimal component of $\uu\lambda$.  Using compactness, we obtain a bound on $\int_l \alpha$, where $l$ is a lift of $\uu l$ (whose projection) meets $B_R(o)$.  
Points in $K_N$ follow a single leaf of $\lambda$ for at least $N$-combinatorial steps, so during these $N$-steps, $v$ meets at most one translate of such a compact segment of an isolated leaf. Its integral thus picks up at most one such additive ``error term.''

Using the fact that $\beta^n(y)$ is uniformly bounded for all $n\ge 0$ (Lemma \ref{lemma: basic slack estimates} and the slack of any segment contained in $\lambda$ is $0$)  and the uniform growth obtained in the previous paragraph, this additive error can be absorbed into the constant $\rho'$ to obtain $\rho$ after possibly taking $N_0'$ larger.

This sketch concludes the proof.
\end{proof}

The Partition Lemma \ref{lem: partition lemma general} gives us potentially arbitrarily long concatenations of subsegments of the type treated by Lemma \ref{lem: low slack general case}. The following lemma provides the necessary bound for subsegments of this type which will ultimately give control of the height function.  Note the asymmetry in the upper and lower bounds.  

\begin{lemma}\label{lem: zig zag}
    There is a positive $\varrho$ such that the following holds.  For $N$ large (as in Lemma \ref{lem: low slack general case}), let $K_N$ be the complete banner system obtained from Lemma \ref{lem: construct N banner}.  
    Suppose there is an $M\ge1$ such that $v \in K_N$ satisfies that 
    $\sigma^{mN}v \in K_N$
    for all $0\le m \le M-1$.  Then
    \[\frac{1}\varrho NM\ge \frac{1}{e^{\s[v,G^{NM}v]}} A(v,G^{NM}v)\beta^{NM}(v)(g_{T}w)_0\ge \varrho N\]
    for all $w\in K$ and $T\ge 0$. 
\end{lemma}
\begin{proof}
    Let $v_m\defeq \sigma^{mN}v$ so that 
    \[\beta^{MN}(v) = \beta^N(v)\beta^N(v_1)\dots\beta^N(v_{M-1}).\]
    Apply Lemma \ref{lemma: notation lemma}, stopping at each integer multiple of $N$ combinatorial steps in time, to write
    \begin{align*}A(v, G^{NM}v )\beta^{NM}(v)=& \sum_{m=0}^{M-1}A(v_m, \beta^{N}(v_m)v_{m+1})\beta^N(v_m) \beta^{N(M-m-1)}(v_{m+1})
     \end{align*}
    Now consider, for each $m = 0, ..., M-1$, $w \in K$ and $t\ge0$, the vector
    \begin{equation}\label{eqn: vector quantity}
        \beta^{(M-m-1)N}(v_{m+1})(g_tw)_0 \in \R^3.
    \end{equation}
    Let $\alpha_m = [v_{m+1},G^{N(M-(m+1))}v_{m+1}]$; this is the tail of the segment $[v, G^{NM}v]$ in $T^1\Omega$.
    \begin{claim}\label{clm: vector position and size specific}
    
    There is $C_0>0$ such that given $t\ge 0$ and $w \in K$, there is a $w'_0 \in \Omega_0\setminus \mathcal U$ such that 
    \[\beta^{(M-m-1)N}(v_{m+1})(g_tw)_0 = C e^{\s(\alpha_m)}w'_0 ,\]
    where $1/C_0\le C\le C_0$.  
    Furthermore, 
    \[1\le e^{\s(\alpha_m)}\le e^{\epsilon_0N(M-m-1)}.\]
    \end{claim}

    For the claimed equality to hold, $w_0'$ is necessarily a constant multiple of the vector on the left-hand side. The content of the claim is that the scale of the left hand side is controlled by the slack of the noted segment.
    
    The claim says that the size of \eqref{eqn: vector quantity} is controlled in terms of the slack of a certain (possibly long) geodesic segment, and its position is nowhere near the set $Y_-\subset \overline \Omega$ where the estimates of Lemmas \ref{lem: low slack minimal} and \ref{lem: low slack general case} do not apply.

    \begin{proof}[Proof of Claim \ref{clm: vector position and size specific}]
        This is a direct application of the next more general Claim \ref{clm: vector position and size} having kept careful track of the relevant indices and applying the slack bounds of Lemma \ref{lem: low slack general case}.
    \end{proof}
    The advantage of the following claim, is that the notation/indexing is not quite so tedious as the previous one.  
    \begin{claim}\label{clm: vector position and size}
        There is $C_0>0$ such that the following holds.  Given $L\ge1$ and $v\in K_N$ with $\sigma^{l N}v\in K_N$ for all $l =1, ...,L$, then for all $w \in K$, and for all $t \ge0$, there is a $w' \in \Omega \setminus \mathcal U$ such that 
        \[\beta^{LN}(v)(g_tw)_0 = Ce^{\s[v, G^{LN}v]}w'_0,\]
        where $1/C_0\le C\le C_0$.
    \end{claim}

    \begin{proof}[Proof of Claim \ref{clm: vector position and size}]
        By definition of $K_N$, $v$ satisfies that there are $y_1, ..., y_L \in Y$ such that 
        \[\beta^{LN}(v) = \beta^N(y_1) \beta^N(y_2) ... \beta^N(y_L).\]
        By Lemma \ref{lem: forward cones nest}, if $N$ is large enough, we have that $\beta^N(y_i)K_{\ge0}\subset K_{\ge0}$ for all $i$, and hence 
        \[\beta^{LN}(v)K_{\ge0} \subset K_{\ge0}.\]
        Since $\mathcal U \cap K_{\ge0} = \emptyset$, we conclude that 
        \[\beta^{LN}(v)(g_tw)\not\in \mathcal U\]
        for all $t\ge0$ and all $w \in K$.  

        Now, we only need to explain why the size of $\beta^{LN}(v)(g_tw)_0$ is controlled.  
        This is essentially a direct consequence of Proposition \ref{prop: projectors and slack} and the fact that the hyperplane corresponding to the smallest singular values of $\beta^{LN}(v)$ never gets close to $K$.  
    \end{proof}

    By Claim \ref{clm: vector position and size specific}, for each $m = 0, ..., M-1$, there are $C_m$  with $1/C_0\le C_m\le C_0$ and $(w_m')_0 \in \Omega\setminus \mathcal U $ such that 
    \[\beta^{(M-m-1)N}(v_{m+1})(g_tw)_0 = C_m e^{\s(\alpha_m)}(w'_m)_0.\]
    Let $\delta_m = [v, v_{m+1}]$ be the initial subsegment of $[v,G^{NM}v]$ so that $\delta_m \alpha_m = [v,G^{NM}v]$
    Thus
    \begin{align*}\frac{1}{e^{\s[v,G^{NM}v]}}A(v, G^{NM}v)\beta^{NM}(v) (g_tw)_0& = \frac{1}{e^{\s[v,G^{NM}v]}}\sum_{m=0}^{M-1}A(v_m, \beta^{N}(v_m)v_{m+1})\beta^N(v_m)C_me^{\s(\alpha_m)}(w_m')_0 \\
    &= \sum_{m=0}^{M-1}A(v_m, \beta^{N}(v_m)v_{m+1})\beta^N(v_m)C_me^{-\s(\delta_m)}(w_m')_0.
    \end{align*}

    By Lemma \ref{lem: low slack general case} and additivity of slack along subpaths, $0\le \s(\delta_m) \le (m+1)\epsilon_0$.
    With $(*)$ being the quantity from the lemma statement that we want to bound, applying Lemma \ref{lem: low slack general case} to each summand thus gives bounds
    \[\frac{C_0NM}{\rho'}\ge\frac{N}{\varrho'}\sum_{m=0}^{M-1}C_me^{-\s(\delta_m)}\ge (*)\ge \varrho'N\sum_{m=0}^{M-1}C_me^{-\s(\delta_m)}\ge \frac{\varrho'N}{C_0} \sum_{m=0}^{M-1} e^{-(m+1)\epsilon_0}. \]
    Noting that the geometric series on the right converges to a quantity independent of $N$, we can now choose $\varrho$ suitably to prove t he lemma.
\end{proof}

\subsection{Lower height bounds}

Recall from Definition \ref{def: positive cone} the positive cone $\mathscr P_{\rho\chi}$.
\begin{theorem}[Bounded Height]\label{thm: heights bounded below general}
    Given $\alpha\in \mathscr P_{\rho\chi}$, there is a constant $C(\alpha)>0$ such that for all geodesic rays $r: [0,\infty) \to \Omega$ with $r(0) =o$ and all $T\ge 0$, we have 
    \[h(r(T))\ge -C(\alpha).\]
\end{theorem}

The proof of this Theorem is nearly identical to that of Theorem \ref{thm: heights bounded below scc}, with the Low Slack Lemma \ref{lem: low slack minimal} and its consequences (Lemma \ref{lem: low slack general case} and Lemma \ref{lem: zig zag}) used in place of Lemma \ref{lem: main low slack cor} and the new Partition Lemma \ref{lem: partition lemma general} in place of Lemma \ref{lemma: partition lemma multicurve case}.

\begin{proof}

    As before, we recall that 
    \[h(r(T)) = \frac{A(o,r(T))\Psi(r(T))}{\|\Psi(r(T))\|},\]
    and
    \[h(r(T))\succeq \frac{A(o,r(T))\Psi(r(T))}{e^{\s[o,r(T)]}}\]
    from Lemma \ref{lemma: basic slack estimates}.
    
    Given $N$ and  $K_N\subset K$  as in Lemma \ref{lem: low slack general case}, The Partition Lemma \ref{lem: partition lemma general} furnishes $ 0=t^0<t^1 < ...$, (where it's possible that the $t^i$'s are $\infty$ for all $i\ge i_0 \ge 1$).  Given $T>0$, there is a largest $n_T$ such that $t^{n_T-1}\le T$.  We set 
    \[q_i = \begin{cases} r(t^i), & 0\le i <n_T,\\
    r(T), & i = n_T.
    \end{cases}\]
    The Notation Lemma \ref{lemma: notation lemma} gives $p_i \in B_R(o)$ and $\gamma_i \in \Gamma$ such that 
    \[q_i = \gamma_1 \cdots\gamma_i.p_i,\]
    and
    \begin{align*}
        A(o,r(T))\Psi(r(T))= \sum_{i = 0}^{n_T-1} A(p_i,\gamma_{i+1}p_{i+1})\rho\chi(\gamma_{i+1})\rho\chi(\gamma_{i+2}\cdots\gamma_{n_T})\Psi(p_{n_T}).
    \end{align*}
    In fact, we may assume that $p_i \in k$ for all $i = 1, ..., n_{T}-1$.
    \medskip

    We now handle the even and odd terms (corresponding to $H$- and $L$-subsegments, respectively) separately.
    The $L$-segments correspond to geodesic segments of the form $[q_i, q_{i+1}]$ where $i$ is odd, and each one also comes with an integer $M_i\ge1$;
     we may assume that 
    \[\gamma_{i+1} = b^{NM_i}(\dot r(t^{i})).\]

    \noindent\textbf{1) Even terms and high-slack estimates:} 
    As a consequence of the Partition Lemma \ref{lem: partition lemma general} and Lemma \ref{lem: slack grows in H segments}, every $H$-segment $[p_i, \gamma_{i+1}p_{i+1}]$ can be further partitioned into pieces whose length is in the interval $[\delta_0, 2\delta_0]$ and whose slack is at lest $\vartheta 2\delta_0$.  These are the main inputs for the argument in Theorem \ref{thm: heights bounded below scc} which gives that
    \begin{align*}
    \left| \frac{1}{e^{\s[0,r(T)]}} \sum_{i \text{ even} } A(p_i,\gamma_{i+1}p_{i+1})\rho\chi(\gamma_{i+1})\rho\chi(\gamma_{i+2}\cdots\gamma_{n_T})\Psi(p_{n_T})\right| 
   &\prec  \frac{1}{e^{\s[0,r(T)]}}\sum_{i \text{ even} } e^{\s[q_{i},r(T)]}\\
   & = \sum_{i \text{ even} } e^{-\s[0,q_i]}=O(1).
   \end{align*}

    \medskip

    \noindent\textbf{2) Odd terms and low-slack estimates:}
    We claim that
    \[\rho\chi(\gamma_{i+2}\cdots\gamma_{n_T})\Psi(p_{n_T}) \asymp e^{\s[q_{i+1}, r(T)]}(g_tw')_0,\]
    where $t\ge0$, $w'\in K$.  This follows from Lemma \ref{lemma: basic slack estimates}, because $\Psi(p_{n_T})$ lives in a compact set, by the construction of our partition (Lemma \ref{lem: partition lemma general}).

    We can then apply Lemma \ref{lem: zig zag} to obtain
    \[ M_i e^{\s[q_{i},r(T)]} \succ A(p_i,\gamma_{i+1} p_{i+1})\rho\chi(\gamma_{i+1})\rho\chi(\gamma_{i+2}\cdots\gamma_{n_T})\Psi(p_{n_T}) \succ e^{\s[q_{i},r(T)]} \]
    (note that this is a weaker conclusion than the one obtained in the proof of Theorem \ref{thm: heights bounded below scc}).
    Summing all of the odd terms, we thus obtain
    \begin{align*}
     \sum_{i \text{ odd} } M_i e^{-\s[0,q_{i+1}]} &=  
     \frac{1}{e^{\s[0,r(T)]}}\sum_{i \text{ odd} } M_i e^{\s[q_{i+1},r(T)]} \\ 
      &\succ \frac{1}{\|\Psi(r(T))\|}\sum_{i \text{ odd} }A(p_i,\gamma_{i+1}p_{i+1})\rho\chi(\gamma_{i+1})\rho\chi(\gamma_{i+2}\cdots\gamma_{n_T})\Psi(p_{n_T}) 
    \succ \sum_{i \text{ odd} } e^{-\s[0,q_{i+1}]}>0.
    \end{align*}
    where we have used the additivity of slack in the equality.  Note that the last sum on the right is convergent, so the sum of the odd terms is non-negative.
    \medskip

    \noindent\textbf{3) Combining the estimates:} Putting it together gives
    \begin{equation}\label{eqn: upper and lower height bound 2}
        \left(O(1)+ const \sum_{i \text{ odd} }^{n_T-1} M_i e^{-\s[0,q_{i+1}]}\right) 
    \ge
    h(r(T))\asymp\frac{A(o,r(T))\Psi(r(T))}{e^{\s[o,r(T)]}} \ge O(1),
    \end{equation} 
    where the constants are positive and don't depend on $T$ or $r$.    
    Since each odd term is positive, we thus achieve a uniform lower bound on height for every ray $r$, proving the theorem.
\end{proof}

\begin{corollary}\label{cor: finite slack implies infinite height}
     Suppose $r: [0,\infty) \to \Omega$ is a geodesic ray with $\s(r)<\infty$.  Then $h(r(t))\to \infty$ as $t\to \infty$.
\end{corollary}
   
\begin{proof}
    Observe that $\s(r)<\infty$ implies that the $\omega$-limit set of the ray $\pi\circ r$ in $M$ is contained in $\uu\lambda$.  Indeed, Lemma \ref{lem: slack grows} implies that an accumulation point of $\pi\circ r$ away from $\uu\lambda$ would produce infinitely many subsegments of $r$ that contribute a definite amount of slack.
    So $r$ must (eventually) spend longer and longer close to individual leaves of $\lambda$.
    Thus the result follows using the techniques of Lemma \ref{lem: low slack general case} and Theorem \ref{thm: heights bounded below general} to establish a lower bound on $h(r(t))$ as $t$ tends to infinity. 
\end{proof}

\section{Proximal limit set is orbital limit set}\label{sec: prox is orb}

Given the lower bounds on heights proven in \S \ref{sec: heights bounded below general} and \S \ref{sec: heights bounded multicurve} when the cocycle $\alpha$ is in the positive cone $\mathscr P_{\rho\chi}$, we turn our attention to understanding the behavior of the limit sets of $\eta$ in $\partial \mathscr C\Omega$. 

The hypotheses of this section are the union of the hypotheses \ref{hyp A} from \S\ref{sec: heights bounded multicurve} and hypotheses \ref{hyp B} from \S\S\ref{sec: positive cone lamination}--\ref{sec: heights bounded below general}.  That is,
\begin{itemize}
    \item $\rho: \Gamma\to \SL(d+1, \R)$ is injective and divides a properly convex domain $\Omega\subset \RP^d$, where $d\ge 2$.
    \item $\chi = e^{-\varphi}: \Gamma \to \R_{>0}$ has  $\|\varphi\|_\rho=1$  
    \item We choose a $\varphi$-equivariant tight map $\tau: \Omega \to \R$ provided by Theorem \ref{thm: co-stable norm and optimal best maps} with corresponding tension locus $\uu\lambda = \Tension(\tau)/\rho$, which is an oriented chain recurrent lamination.
    \item If $d\ge 3$, we assume that $\uu\lambda$ is a multicurve in $M=\Omega/\rho$; see also Remark \ref{rmk: why not higher dimensions}.
\end{itemize}

For $\alpha \in \Omega^1(E_{\rho\chi}^*)$, we let $\eta$ be the corresponding coaffine representation with linear part $\rho\chi$ and with cocycle corresponding to $\alpha$.

\begin{definition}
    A coaffine representation $\eta = \eta(\rho,\chi, Z)$ is \textbf{exceptional} if the cocycle $\alpha$ defining $Z$ is in the boundary of the positive or negative cones, i.e.\
    \[
    [\alpha]\in \pm \partial\mathscr P_{\rho \chi} \subset H^1(\pi,(\R^{d+1})^*_{\rho\chi}).
    \]
    Otherwise, $\eta$ is \textbf{unexceptional}.
\end{definition}

We pick up from \S\ref{section: geometry of coaffine domain} the discussion of structure and relationship between proximal and orbital limit sets for unexceptional coaffine representations in the closure of Anosov (Theorem \ref{thm: limit sets coincide}).  We also prove Proposition \ref{prop: found a ray}, which finds a vertical ray or line in such limit sets and Corollary \ref{cor: epigraph}, which exhibits limit sets as epi- (or hypo-) graph of a lower (or upper) semi-continuous height function.

The behavior of limit sets for exceptional representations seems somewhat complicated, and we only address a few special cases of interest rather than attempting to give a full account of their behavior.
\subsection{Independence of the orbital limit set}

Recall from \S\ref{section: geometry of coaffine domain} that the orbital limit set $\Lambda_{orb}^o$  of a coaffine representation $\eta$ preserving the cone $\mathscr C \Omega$ is the set of accumulation points of $\eta(\Gamma).o$ in $\partial^\infty \mathscr C \Omega$, where $o \in \mathscr C \Omega$.   We showed the following inclusions for $o \in C_{min}$, the minimal convex $\eta$-invariant subset of $\mathscr C \Omega$:
\[ \Lambda_{prox} \subset \Lambda_{orb}^o \subset \partial^\infty C_{min}.\]

Our first result states that for unexceptional coaffine representations, the orbital limit set does not actually depend on the choice of basepoint $o$.  

\begin{proposition}
\label{prop: positive orbital sets coincide}
    Suppose that $\alpha \in \mathscr P$ is a cocycle defining $\eta$, a representation on the boundary of Anosov. Then for $o$ and $p$ in a common vertical line of $\mathscr C \Omega$, the orbital limit sets of $o$ and $p$ coincide:
    \[\Lambda_{orb}^o = \Lambda_{orb}^p.\]
    The set $\hull \Lambda_{orb}^p = C_{min}$ for any $p\in \mathscr C \Omega$. 
\end{proposition}

We outline the proof for the first statement regarding equality of orbital limit sets for points on the same vertical line; the second statement is a direct consequence of this and Lemma \ref{lemma: minimal domain}.  The idea is that points on the same line are contracted by group elements with high slack, i.e., for which $(\sigma_1^\rho - \varphi)$ is large.  On the other hand, Theorems \ref{thm: heights bounded below scc} and Theorem \ref{thm: heights bounded below general} show that for group elements with bounded slack, the line containing $o$ and $p$ is not contracted, but their heights tend to $\pm \infty$.  In either case, this gives the desired result that the set of accumulation points in $\partial \mathscr C \Omega$ coincide.

\begin{proof}
    By definition, the points $p$ and $o$ have the same image under $Q$, call it $[v] \in \Omega$ and we may assume $v \in \Omega_0$. We write 
    \[ o= 
    \begin{bmatrix}
        v\\
        H(o)
    \end{bmatrix}
    \textrm{ and }
    p = 
    \begin{bmatrix}
        v\\
        H(p)
    \end{bmatrix}\]
    where $H(o)$ and $H(p)$ are the heights of $o$ and $p$ from subsection \ref{subsec: height}.

    By direct computation, we obtain for any $\gamma\in\Gamma$ 
    \[\eta(\gamma).o = 
    \begin{bmatrix}
        \rho\chi(\gamma) & 0\\
        Z(\gamma) & 1
    \end{bmatrix}
    \begin{bmatrix}
        v\\
        H(o)
    \end{bmatrix}
    =
    \begin{bmatrix}
        \rho\chi(\gamma)v\\
        Z(\gamma)v+H(o)
    \end{bmatrix}
    \]    
    and similarly for $p$. By definition of the height $H$, then
    \[
    H(\eta(\gamma).o) = \frac{Z(\gamma)v+H(o)}{\|\rho\chi(\gamma)v\|}
    \]
    and similarly for $p$. The difference
    \[H(\eta(\gamma).o)-H(\eta(\gamma).p) = \frac{H(o)-H(p)}{\|\rho\chi(\gamma).v\|} \asymp \frac{H(o)-H(p)}{e^{\s [o,\gamma.o]}}\]
    where the last estimate is from Lemma \ref{lemma: basic slack estimates}. Thus the effect of a group element on the \emph{differences} of heights on a given vertical line is independent of the cocycle defining $\eta$, and is determined, up to fixed multiplicative error, by the slack of the group element.

    It is sufficient to consider a sequence $\gamma_n\in\Gamma$ satisfying that the limit $\gamma_n.o$ is well-defined in $\RP^{d+1}$. Up to taking subsequences, either $\s [o,\gamma_n.o]$ remains bounded for all $n$ or it does not. In the latter case, the height difference 
    \[H(\eta(\gamma_n).o)-H(\eta(\gamma_n).p)\]
    tends to zero with $n$ by the above computation. The sequences $\eta(\gamma_n).o$ and $\eta(\gamma_n).p$ have the same limit in this case.

    In the case that $\s [o,\gamma_n.o]$ remains bounded, we have that the segments $[o,\gamma_n.o]$ subconverge to a ray of finite total slack. Invoking Corollary \ref{cor: finite slack implies infinite height} completes the proof in this case; both sequences of heights tend to infinity, and so both sequences of points tend to the exceptional point $[e_{d+2}]$.

    The final conclusions follow from Lemmas \ref{lemma: minimal domain} and \ref{lemma: C_min is a nontrivial fibration}.
\end{proof}

We make the following simple observation about the \textit{full} orbital limit set for exceptional representations in the multicurve case. 

\begin{remark}\label{rmk: exceptional and full orbital}
    If $\underline{\lambda}$ is a multicurve, and $\alpha$ is in $\partial \mathscr P_{\rho\chi}\cap \partial (-\mathscr P_{\rho\chi})$, then the \textbf{full orbital limit set} is as large as possible, i.e.,
    \[ \bigcup_{o\in \mathscr C \Omega} \Lambda^o_{orb} = \partial \mathscr C \Omega.\]
    The reason for this is that there is a line in $\lambda$ with stabilizer $\beta$ for which there is a splitting of $\R^{d+2}$ into an attracting $2$-plane and a repelling $d$-plane. The image of the $2$-plane is a vertical line in $\partial \mathscr C \Omega$ that is fixed pointwise by $\eta(\beta)$.  It is not difficult to see that this vertical line is in the full orbital limit set, and the claim follows.
\end{remark}

\subsection{Structure of $C_{min}$ revisited}

Recall our basic structural Lemma \ref{lem: rays or lines} for $\partial^\infty C_{min}$ that there are three mutually exclusive possibilities:
\begin{itemize}
    \item $\partial^\infty C_{min}$ is a union of rays, all pointing in the same direction;
    \item $\partial^\infty C_{min} = \partial^\infty\mathscr C \Omega$; or
    \item $\partial^\infty C_{min}$ is a union of closed segments, some of which may be reduced to points.
\end{itemize}

\begin{remark}\label{rmk: fibers empty}
    The preimage in $\partial^\infty C_{min}$ of a point $z\in\partial \Omega$ may be empty. In fact for those representations in the positive cone $\mathscr P_{\rho\chi}$, the forward endpoints of leaves in the lamination $\lambda$, which is a dense set, have empty pre-image (Corollary \ref{cor: finite slack implies infinite height}).
\end{remark}

\begin{remark}\label{rmk: fibers points and segments}
    There are examples of exceptional coaffine representations such that some of the $Q$-fibers in $\partial^\infty C_{min}$  are segments with non-empty interior and some are singletons.  Indeed, consider $\rho$ and $\chi$ where $\uu\lambda$ is a multicurve and  $\alpha \in \partial \mathscr P_{\rho\chi}\cap \partial (-\mathscr P_{\rho\chi})$ as in Remark \ref{rmk: exceptional and full orbital}.   Assume also that $\alpha$ is not a coboundary.  Then $\Lambda_{prox}$ is not contained in a copy of $\RP^d$ (see, e.g., Lemma \ref{lem: compute height}), so the fiber over a forward endpoint of a leaf of $\lambda$ contains an interval with open interior. On the other hand, for dynamical reasons, if $\eta(\gamma)$ is proximal, then the fiber over $\gamma_+ \in \partial^\infty \Omega$ is a point.
\end{remark}

The following lemma gives a condition under which the three invariant closed subsets of the boundary coincide.

\begin{lemma}\label{lem: sufficient to find a ray}
    Suppose, for some $x \in \partial \Omega$, that $Q\inverse(x)\cap \Lambda_{prox}$ is a positive vertical ray.
    Then, for every $y\in \partial \Omega$, the fiber
    \[
        Q^{-1}(y)\cap\partial^\infty C_{min}
    \]
    is empty or a positive ray contained in $\Lambda_{prox}$.
    
    Analogous statements hold if the proximal limit set contains a negative vertical ray or a line.  
\end{lemma}

\begin{proof}
    We prove the claim in the positive case, the other two follow similarly.
    
    Let
    \[
    R_x=[p_x,[e_{d+2}])
    \subset\Lambda_{prox}
    \]
    be the given maximal positive vertical ray over
    $x\in\partial\Omega$.  Because $Q\inverse(x) \cap \Lambda_{prox} = Q\inverse(x) \cap \partial^\infty C_{min}$, Lemma \ref{lem: rays or lines} implies that every non-empty  $Q$-fiber of $\partial C_{min}$ is a ray.

    Consider any proximal element $\gamma$. By taking a translate of $R_x$ if necessary, we may assume that 
    $R_x$
    is not contained in the repelling plane of $\eta(\gamma)$. Then iterates of $R_x$ under positive powers of $\gamma$ accumulate to a positive ray over $\eta(\gamma)^+$ contained in $\Lambda_{prox}$.

    Consider then any proximal point 
    $z\in \partial \mathscr C \Omega.$
    By definition, there is a sequence of proximal elements 
    $\gamma_n$
    whose attracting fixed points accumulate to $z$. The rays over these fixed points accumulate to a ray over $z$, which is necessarily in the proximal limit set. 

    Applying Lemma \ref{lem: proximal rays, proximal endpoints}, to any other non-empty fiber shows that it is contained in $\Lambda_{prox}$.
\end{proof}

\subsection{Finding a vertical ray in the proximal set}
The main result of this section is the following, with hypotheses as in the introductory text to this section.  
\begin{theorem}\label{thm: limit sets coincide}
    For all unexceptional coaffine deformations $\eta$ of $\rho\chi\oplus 1$, we have 
    \[\Lambda_{prox} = \Lambda_{orb} = \partial^\infty C_{min},\] 
    and the orbital limit set does not depend on the choice of base-point $o \in \mathscr C\Omega$. 
\end{theorem}

The new ingredient for the proof of the Theorem is the following. 

\begin{proposition}\label{prop: found a ray}
    There is a vertical ray or line in $\Lambda_{prox}$  over every proximal limit point for every unexceptional coaffine representation.  
\end{proposition}

\begin{proof}[Proof of Theorem \ref{thm: limit sets coincide} assuming Proposition \ref{prop: found a ray}]
    Assume that  the Proposition 
    supplies us with a positive ray in $\Lambda_{prox}$.  
    
    We claim that  $ \Lambda_{prox}=\partial^\infty C_{min}$.  Indeed, $\Lambda_{prox}\subset\partial^\infty C_{min}$ holds by definition (Lemma \ref{lemma: minimal domain}), while the other inclusion is a direct consequence of Lemma \ref{lem: sufficient to find a ray}.
    
    For $o\in C_{min}$, we always have
    \[
    \Lambda_{prox}
    \subset\Lambda_{orb}^o
    \subset\partial^\infty C_{min},
    \]
    so all three sets coincide. Finally, applying Proposition  \ref{prop: positive orbital sets coincide} shows the orbital limit sets all coincide, so that the above inclusion holds for any point in $\mathscr C \Omega$, which is what we wanted to show.

    The cases that Proposition \ref{prop: found a ray} supplies us instead with a negative ray or line are treated similarly.    
\end{proof}

The proof of the Proposition \ref{prop: found a ray} occupies the remainder of this section.
The strategy is to compute the heights of attracting fixed points of a sequence of proximal elements of the form $\gamma^n\beta^m$, where $\gamma$ is itself $\eta$-proximal (i.e., its oriented geodesic representative is not contained in $\uu\lambda$) and $\beta^m$ travels along a leaf of $\uu\lambda$ for $m$ combinatorial steps; we make this setup more precise below.  As $n$ gets large, the attracting fixed points become very close to the fiber $q\inverse (\gamma^+)$, where $\gamma^+ \in \partial^\infty \Omega$ is the attractive fixed point for the the $\rho$-action of $\gamma$, while as $m$ gets large, the height of the attractive fixed point goes to $+\infty$, if the cocycle defining $\eta$ is positive on the minimal component of $\uu\lambda$ containing the leaf corresponding to the sequence $(\beta^m)$.  Note that $\beta ^m$ is not necessarily powers of some fixed $\beta $ while $\gamma^n$ is.  

We will see that as $n$ gets large, the height increase contribution coming from $\beta^m$ is exponentially contracted.  Also, however, analyzing the (linear) growth rate of $Z(\beta^m)$ under hypotheses \ref{hyp A} allows us to conclude that $\{\eta(\gamma^n\beta^m)^+: n , m\ge 0\}$ accumulate to an upward facing (positive) ray in $q\inverse(\gamma^+)$, which is then contained in $\Lambda_{prox}$.  Compare this to the fact that $\overline{\{e^{-n}m: n,m>0\}} = \mathbb R_{\ge0}$.  We use a more sophisticated argument under hypotheses \ref{hyp B}, since the bounded gap property does not necessarily hold. 

A similar phenomenon occurs when the cocycle defining $\eta$ is negative on the minimal component of $\uu\lambda$ corresponding to the sequence $(\beta^m)$.

First, we compute the height of the attracting fixed point of a proximal element.
This is not strictly necessary for the proof of Proposition \ref{prop: found a ray}, since we carry out a more sophisticated computation there that implies this formula.  For pedagogical reasons, we include it anyway.
\begin{lemma}\label{lem: compute height}
    Let $Z \in Z^1(\Gamma,(\R^{d+1})_{\rho\chi}^*)$ be the cocycle defining $\eta$, i.e., $Z(\gamma) = A(o,\gamma.o)\rho\chi(\gamma)$.
    Suppose $\gamma \in \Gamma$ is $\eta$-proximal.
    Then the height of the attracting fixed point $\eta(\gamma)^+ \in \partial^\infty \mathscr C\Omega$ admits the formula
    \[h(\eta(\gamma)^+) = \frac{Z(\gamma)(\gamma_0^+)}{e^{\s(\gamma)}-1},\]
    where $\s(\gamma) = \lambda_1^\rho(\gamma)-\varphi(\gamma)$ is the $\tau$-slack of the geodesic representative of $\gamma$ and $(\gamma)_0^+\in \R^{d+1}$ is the unit vector in $\partial\Omega_0$ corresponding to $\gamma^+\in \partial^\infty\Omega$.
\end{lemma}

\begin{proof}
    This is a direct computation. For convenience let
    $H= h(\eta(\gamma)^+)$ 
    and
    $v = (\gamma)_0^+$
    for the remainder of the proof.
    The vector
    $(v,H)^t$
    is the attracting fixed point of $\eta(\gamma).$ Since the top eigenvalues of $\chi\rho(\gamma)$ and $\eta(\gamma)$ are both $e^{\s(\gamma)}$, we find that

    \[
    \begin{pmatrix}
        \chi\rho(\gamma) & 0 \\
        Z(\gamma) & 1 
    \end{pmatrix}
    \begin{pmatrix}
        v \\
        H
    \end{pmatrix}
    =
    \begin{pmatrix}
        e^{\s (\gamma)}v \\
        Z(\gamma)v + H
    \end{pmatrix}
    \]
    Following the definition of the height, we have that $H$ satisfies
    \[
    H = \frac{Z(\gamma)v + H}{e^{\s (\gamma)}}
    \]
    and the result follows by solving for $H = h(\eta(\gamma)^+)$.
\end{proof}

\begin{proof}[Proof of Proposition \ref{prop: found a ray}]
    
    As alluded to above, we choose an element $\gamma \in \Gamma$ with $\s(\gamma)>0$.  Since $\eta$ is unexceptional, every minimal component of $\uu\lambda$, has a well defined sign (either `$+$' or `$-$') associated to it; see Lemma \ref{lemma: positivity lemma} or Definition \ref{def: positive for minimal}.
    Let $\uu\mu\subset \uu\lambda$ be such a minimal component, and assume without loss of generality that the sign associated to $\uu\mu$ is `$+$'.
    If $\uu\mu$ is a closed curve, choose $\beta\in \Gamma$ to be the stabilizer of a leaf of $\mu \subset \Omega$ that runs through $B_R(o)$; see \S\ref{subsec: positivity scc}.  If $\uu\mu$ is not a curve, we choose a leaf $\ell_x \in \mu$, which is generic for an invariant ergodic  probability measure $\nu$  that meets a lift $k \subset B_R(o)$ of a train track transversal from \S\ref{subsec: laminations and dynamics}, as in the setup for Lemma \ref{lemma: train track branches from returns}.  We let $\beta^n = \beta^n(x)$ for $n\ge 1$.\footnote{In previous sections, $b$ and $b^n(x)$ referred to elements of $\Gamma$, while $\beta = \rho\chi(b)$, and so on, but this should not cause confusion.}

    Now we approximate the heights of attractors for $\eta(\gamma^n\beta^m)$, where $n, m>0$ and show that they accumulate onto a positive ray over $Q\inverse(\gamma^+)$ (with an explicit lower bound).  The strategy is to compute the asymptotic behavior of heights of attracting eigenlines for $\eta(\gamma^n\beta^m)$.  Instead of working directly with the formula in Lemma \ref{lem: compute height}, it turns out to be more convenient to work directly with projectivizations of the matrices $\eta(\gamma^n\beta^m)$ for well-chosen $n$ and $m$ going to infinity.  
    \medskip
    
    Consider 
    \begin{equation}\label{eqn: projective product ray}
        \frac{\eta(\gamma^n\beta^m)}{e^{n\s(\gamma)}} =
    \begin{pmatrix} 
        e^{-n\s(\gamma)}\rho\chi(\gamma^n )\rho\chi(\beta^m) & \mathbf 0 \\
        e^{-n\s(\gamma)}\left( Z(\gamma^n)\rho\chi(\beta^m) + Z(\beta^m)\right) & {e^{-n\s(\gamma)}}
    \end{pmatrix}
    \end{equation}
    Let $P_\gamma\in \End(\R^{d+1})$ be the rank-one projection corresponding to $\rho\chi(\gamma)$, i.e.,
    \[P_\gamma = \lim_{n \to \infty} e^{-n\s(\gamma)}\rho\chi(\gamma^n) = ( \gamma^-)^* \otimes \gamma^+_0,\]
    using the convention that $(\gamma^-)^*(\gamma^+_0)=1$.
    Then using Lemma \ref{lem: rank one projectors and Anosov} and basic facts about geometric series, we obtain
    \begin{align*}
        e^{-n\s(\gamma)}Z(\gamma^n) &= e^{-n\s(\gamma)} Z(\gamma) \sum_{i=0}^{n-1} \rho\chi(\gamma)^i\\
        & = Z(\gamma) \sum_{i=0}^{n-1}e^{(i-n)\s(\gamma)}P_\gamma +o(1) = \frac{Z(\gamma)P_\gamma}{e^{\s(\gamma)}-1}+o(1).
    \end{align*}

    \begin{proof}[Continuation of the proof under hypotheses \ref{hyp A}]
        Now, let us analyze \eqref{eqn: projective product ray} again.
    Let $P_\beta$ be the rank-one projection corresponding to $\rho\chi(\beta)$, i.e.,
    \[P_\beta = \lim_{n\to \infty} \rho\chi(\beta^n).\]
    With our notation here, Lemma \ref{lemma: low slack linear growth (simple)} gives 
    \[ e^{-n\s(\gamma)}Z(\beta^m) = e^{-n\s(\gamma)}mZ(\beta)P_\beta +o(1) \]
    Thus
    \begin{equation*}
        \frac{\eta(\gamma^n\beta^m)}{e^{n\s(\gamma)}} =
    \begin{pmatrix} 
        P_\gamma P_\beta +o(1) & \mathbf 0 \\
        \frac{Z(\gamma)P_\gamma P_\beta}{e^{\s(\gamma)}-1} + e^{-n\s(\gamma)}mZ(\beta)P_\beta +o(1) & o(1)
    \end{pmatrix}
    \end{equation*}

    Let $t \in \R_{>0}$ be given, and choose $m_n$ so that 
    \[\frac{m_n}{e^{n\s(\gamma)}} \to t, \text{ as } n \to \infty.\]
     Then we have 
    \begin{equation*}
        \lim_{n\to \infty}\frac{\eta(\gamma^n\beta^{m_n})}{e^{n\s(\gamma)}} =
    \begin{pmatrix} 
        P_\gamma P_\beta & \mathbf 0 \\
        \frac{Z(\gamma)P_\gamma P_\beta}{e^{\s(\gamma)}-1} + tZ(\beta)P_\beta  & 0
    \end{pmatrix}.
    \end{equation*}
    Denote by $c>0$ the only non-zero eigenvalue of $P_\gamma P_\beta$, i.e., $c\gamma_0^+ = P_\gamma P_\beta (\gamma_0^+)$.\footnote{This is a kind of cross ratio that calculates the quantity $\lim_{n, m\to \infty} e^{\s(\gamma^n\beta^m) - n\s(\gamma)}$.}
    Let $d = Z(\beta)P_\beta (\gamma_0^+)>0$; this is positive by positivity of $\eta$ on $\beta$.  
    
    An elementary matrix computation then gives that  the matrix just above is rank-one with eigenvalue $c$ and eigenvector
    \[\begin{pmatrix}
        \gamma_0^+ \\ h(\eta(\gamma)^+) + \frac {td}c
    \end{pmatrix}.\]
    The bottom coordinate is the height.
    
    This proves that $(\eta(\gamma^n\beta^{m_n})^+)$ accumulates to the point in $Q\inverse (\gamma^+)$ with height $h(\eta(\gamma)^+) + \frac {td}c$.  Since $t>0$ was arbitrary and $\frac dc>0$, this concludes the proof.
    \end{proof}

    The proof under hypotheses \ref{hyp B} has the same shape, but a slightly more technical argument is needed, because we do not have such good control over $Z(\beta^m)$.  

    \begin{proof}[Continuation of the proof under hypotheses \ref{hyp B}]
    We recall some notation and setup from \S\ref{subsec: laminations and dynamics}.
    We have $x \in X = k\cap \mu$ and the homeomorphism $\sigma: X \to X$ given by the first returns under the geodesic flow and the dynamics of $\rho(\Gamma)$ on $\Omega$.  
    We have chosen $x \in X$ to be a generic point for the $\sigma$-invariant ergodic probability measure $\nu \in \Delta(\sigma)$.  
    Now, choose $y \in X$.  
    By minimality, $\sigma^m(x)$ accumulates to $y$.  We make the following stronger assertion:
    \begin{claim}\label{clm: bounded return times}
        Let $t>0$ be given.  There are sequences $n_j$ and  $m_j$ tending to infinity such that $\sigma^{m_j}x \to y$ as $j \to \infty$ and such that 
        \[\frac{m_j}{e^{n_j\s(\gamma)}} \to t, \text{ as } j \to \infty.\]
    \end{claim}
    \begin{proof}[Proof of Claim \ref{clm: bounded return times}]
        By minimality of $\sigma$ and compactness of $X$, return times to any fixed neighborhood of $y \in X$ are uniformly bounded: for any neighborhood $U$ of $y$, there is an $N$ such that for all $z \in X$, there is a $k \in [0,N]$ such that  $\sigma^kz \in U$.

        Choose a family of neighborhoods $U_j$ shrinking to $y$, and let $N_j$ be the corresponding bounds on return times to $U_j$.  Now choose $n_j$ sufficiently large that $N_je^{-n_j\s(\gamma)}\to 0$, as $j\to \infty$.  
        Choose 
        \[m_j \in [\lfloor te^{n_j\s(\gamma)}\rfloor, \lfloor te^{n_j\s(\gamma)}\rfloor +N_j]\]
        such that $ \sigma^{m_j}x \in U_j$.
        This gives $m_je^{-n_j\s(\gamma)} \to t$ and $\sigma^{m_j}x \to y$.
    \end{proof}

    The Claim and Corollary \ref{cor: limit points of averages} together give that $\frac{Z(\beta^{m_j})}{m_j} =L_\nu(y) +o(1)$.  Thus
    \[\frac{Z(\beta^{m_j})}{e^{n_j\s(\gamma)}} = tL_\nu(y) +o(1).\]
    
    Proposition \ref{prop: projectors and slack} gives 
    \[\rho\chi(\beta^{m_j})  = \Pi(y,x) +o(1).\]

    Combining these ingredients in a similar fashion as before gives 
    \begin{equation*}
        \lim_{j\to \infty} \frac{\eta(\gamma^{n_j}\beta^{m_j})}{e^{n_j\s(\gamma)}} =
    \begin{pmatrix} 
        P_\gamma \Pi(y,x) & \mathbf 0 \\
        \frac{Z(\gamma)P_\gamma \Pi(y,x)}{e^{\s(\gamma)}-1} + tL_\nu(y) & 0
    \end{pmatrix}.
    \end{equation*}
    Thus, with $c$ defined so that  $c\gamma_0^+ = P_\gamma\Pi(y,x)(\gamma_0^+)>0$ and $d = L_\nu(y)(\gamma_0^+)>0$, we arrive at the same conclusion that  $(\eta(\gamma^{n_j}\beta^{m_j})^+)$ accumulates to the point in $Q\inverse (\gamma^+)$ with height $h(\eta(\gamma)^+) + \frac {td}c$.  Since $t>0$ was arbitrary and $\frac dc>0$.
    \end{proof}
    Applying Lemma \ref{lem: sufficient to find a ray} now proves Proposition \ref{prop: found a ray} under the two sets of hypotheses, after repeating to find a negative ray, if necessary.  
\end{proof}

\begin{remark}\label{rmk: characterizing exceptional with a ray or line}
    For exceptional coaffine representations $\eta\in \partial \mathcal A$ where $\uu\lambda$ is a multicurve (hypotheses \ref{hyp A}) , it is not difficult to use the techniques of this subsection to prove the following:
    Let $Z$ be the cocycle defining $\eta$ then \begin{itemize}
        \item  $\Lambda_{prox}$ is bounded below if and only if no curve of $\uu\lambda$ is negative for $Z$.
        \item $\Lambda_{prox}$ is bounded above if and only if no curve of $\uu\lambda$ is positive for $Z$.
    \end{itemize}
   In particular, if every curve $\delta$ of $\uu\lambda$ satisfies $Z(\delta)\delta^+=0$, then $\eta$ is naively convex cocompact (but not strongly convex cocompact).  
\end{remark}

Let $X$ be a compact metric space, and recall that a lower semi-continuous function $f: X\to \R\cup \{\pm \infty\}$  is such that \[\liminf_{y\to x} f(y)\ge f(x).\]
An equivalent characterization of lower semi-continuity is that the epi-graph
\[\{(x,t) \in X\times \R: t\ge f(x) \}\]
is a closed set.

For $\eta$ defined by a cocycle in the positive cone $\mathscr P_{\rho\chi}$, define 
\[\mathcal H : \partial^\infty \Omega\to \R\cup \{\pm\infty\}\]
by 
\[\mathcal H(x) = \inf_{v \in Q\inverse(x)\cap \Lambda_{prox}} h(v),\]
where we recall that the infimum over the empty set is $+\infty$.

\begin{corollary}\label{cor: epigraph}
    The epi-graph of $\mathcal H$ is $\Lambda_{prox}$.  
    Consequently, $\mathcal H$ is lower semi-continuous.  
    Furthermore,
    \[\mathcal H\inverse(+\infty) \supset \lambda^+,\]
    where $\lambda^+$ is the set of forward endpoints of leaves of $\lambda$. 
\end{corollary}

\begin{proof}
    That the epi-graph of $\mathcal H$ is $\Lambda_{prox}$ is essentially by definition.  
    Since the proximal limit set is closed, $\mathcal H$ is lower semi-continuous.  
    The containment follows from  Corollary \ref{cor: finite slack implies infinite height}.  
\end{proof}

\begin{remark}\label{rmk: infty is more than lambda}
    The preimage of $+\infty$ under $\mathcal H$ is more than $\lambda^+$.  See Proposition \ref{prop: average slack and height}, which gives a necessary condition for a ray to have height tending to infinity.  
    It could be interesting to find the precise condition on rays so that their height is infinite. 
\end{remark}

\section{Measure of the limit set}\label{sec: measure}

In this section, we work with hypotheses \ref{hyp A}.  That is, 
\begin{itemize}
    \item $\rho: \Gamma\to \SL(d+1, \R)$ divides a properly convex domain $\Omega\subset \RP^d$
    \item $\chi = e^{-\varphi}: \Gamma \to \R_{>0}$ has  $\|\varphi\|_\rho=1$  
    \item the corresponding stretch lamination $\uu\lambda$ is a multicurve in $\Omega/\rho$
    \item $\eta: \Gamma \to \GL(d+2,\R)$ is unexceptional, i.e., the cocycle $\alpha \in Z^1(\Gamma, (\R^{d+1})^*_{\rho\chi})$ is not in $\partial \pm \mathscr P_{\rho\chi}$.
\end{itemize}
Recall from Theorem \ref{thm: limit sets coincide} that with $\eta$ described as above, there is an equality \[\Lambda(\eta) \defeq \Lambda_{prox} = \Lambda_{orb}=\partial^\infty C_{min}.\]  
The goal is to prove Theorem \ref{thm: positve lebesgue}, below, which proves that $\Lambda(\eta)$ is large in a dimension and measure theoretic sense.  

\begin{remark}\label{rmk: why multicurve}
    We believe that the results of this section should be true in more generality using the tools that we have developed, namely when $\rho: \pi_1S \to \SL(3,\R)$ is Hitchin and the maximally stretch lamination $\uu\lambda$ is arbitrary.  However, the proof (of Proposition \ref{prop: average slack and height}) would require substantial additional technical work, so we choose not to pursue this here.  Indeed, Theorem \ref{thm: positve lebesgue} holds as soon as the analogous statement of Proposition \ref{prop: average slack and height} does.  
\end{remark}

Since $\partial^\infty\Omega$ is a $C^1$-sphere in $\RP^{d}$, $\partial^\infty \mathscr C\Omega$ is a $C^1$-submanifold of $\RP^{d+1}\setminus [e_{d+2}]$, and the action of $\eta$ is smooth preserving this cone.  There is a well defined $d$-dimensional Lebesgue measure class $[m_d]$ on $\partial^\infty \mathscr C\Omega$.
Clearly, when $\Lambda(\eta) = \partial^\infty \mathscr C\Omega$, then $m_d(\Lambda(\eta))>0$.  Our main result is that this continues to hold for all unexceptional coaffine deformations.

\begin{theorem}\label{thm: positve lebesgue}
    With hypotheses as in the beginning of this section, 
    we have $m_d(\Lambda(\eta))>0$.
\end{theorem}

Since $\Lambda(\eta)$ is contained in a $d$-dimensional differentiable submanifold of $\RP^{d+1}$, the following is a direct consequence of Theorem \ref{thm: positve lebesgue}.

\begin{corollary}\label{cor: hausdorff dimension}
    For an unexceptional coaffine representation $\eta$, the Hausdorff dimension of $\Lambda(\eta)$ is $d$.  
\end{corollary}

The proof of Theorem \ref{thm: positve lebesgue} has two steps.  First, by zooming in on the special point $[e_{d+2}]$ at infinity, we reduce our analysis to  the cone over an invariant subset of $\partial^\infty \Omega$.  Then, using some basic features of the slack functional on paths, we show that this invariant subset is co-null and deduce the theorem.
\medskip

Choose a complementary hyperplane $[\ker e^{d+2}]$ to $[e_{d+2}]$. 
Then the matrix $I_t = \textrm{diag}(1, ..., 1, e^{t})$, fixes both $[e_{d+2}]$ and $[\ker (e^{d+2})]$. Define the family of coaffine representations 
\[\eta_t = I_{-t}\eta I_t\]
with corresponding limit sets $\Lambda(\eta_t) = I_{-t}\Lambda(\eta)$. The choice of complementary subspace $[\ker e^{d+2}]$ is not important.

Since $I_{-t}$ acts smoothly on $\partial^\infty\mathscr C\Omega$ preserving the measure class of $m_d$, 
\begin{equation}\label{eqn: preserves md}
    m_d(\Lambda(\eta))>0 \Leftrightarrow m_d(\Lambda(\eta_t))>0, ~ \forall t \in \R
\end{equation}
Without loss of generality, we assume $\eta$ corresponds to a $1$-form in $\mathscr P_{\rho\chi}$.
Let $\Lambda_+(\eta)$ be the intersection of $\Lambda(\eta)$ with the non-negative closed half-space $H$ with $\partial H =[\ker e^{d+2}]$ in an affine chart containing $\mathscr C\Omega$, and similarly define $\Lambda_+(\eta_t) = I_{-t}\Lambda_+(\eta)$.

\begin{definition}\label{def: invariant set F}
Let $F$ consist of those points  $x \in\partial \Omega \subset \partial H$ for which there exists a point in $\Lambda(\eta)$ over $x$ (necessarily at finite height). 
\end{definition}

It is easy to verify that $F$ is $\left(\rho\chi\oplus 1\right)$-invariant.

\begin{lemma}\label{lem: properties of the union}
    For all $0\le s_1\le s_2$, we have 
    \[\Lambda_+(\eta_{s_1}) \subset \Lambda_+(\eta_{s_2}).\]
    The union
    \[\mathscr CF = \bigcup_{t\ge0} \Lambda_+(\eta_{t})\]
    is invariant under the representation $\rho\chi\oplus 1$ and is the one-sided cone over $F$.     
\end{lemma}

 \begin{proof}
 $I_t$ acts preserving $\partial^\infty \mathscr C\Omega$ and restricts to a homothety on each vertical line.  If $t>0$, then points are attracted to $[e_{d+2}]$, while if $t<0$, then points are attracted to the intersection of their vertical line with $\partial H$.  
 The result follows.  
 \end{proof}

Call this one-sided cone $\mathscr C F.$
Let $m_{d-1}$ be the $(d-1)$-dimensional Lebesgue measure class on the $C^1$-sphere $\partial^\infty \Omega$.  
\begin{lemma}\label{lem: measure of union}
    The following are equivalent:
    \begin{enumerate}
        \item $m_d(\Lambda_+(\eta))>0$
        \item $m_d(\mathscr CF)>0$
        \item $m_{d-1}(F)>0$
    \end{enumerate}
\end{lemma}
\begin{proof}
    Since $\mathscr CF$ is the cone over $F$, Fubini's Theorem implies the equivalence of items (2) and (3).  
    
    Now we prove the equivalence of items (1) and (2).  By Lemma \ref{lem: properties of the union}, we have 
    \[
    \mathscr CF = \bigcup_{n\in \mathbb N}\Lambda_+(\eta_n).
    \]
    Using \eqref{eqn: preserves md} and the fact, above, that $\mathscr CF$ is a countable union, we conclude that $m_d(\mathscr CF)>0$ if and only if $m_d(\Lambda_+(\eta))>0$.
\end{proof}

Let $E = \partial^\infty \Omega\setminus F$, which is the set of points $r \in \partial^\infty\Omega$ for which $H(r)=+\infty$.

\begin{proposition}\label{prop: average slack and height}
    With hypotheses as in the beginning of this section, namely that $\uu\lambda$ is a multicurve, for every point $r \in E$ corresponding to a geodesic ray $r: [0, \infty) \to \Omega$ for the domain shape metric with $r(0)=o$, we have
    \[\liminf_{t\to \infty}\frac{\s[o,r(t)]}{t}= 0.\]
\end{proposition}

Before we give the proof, we give an intuitive explanation.  The proof of Theorem \ref{thm: heights bounded below scc} gives the following behavior of height along a ray: the height goes up at a linear rate when fellow traveling a leaf of $\lambda$ in the positive direction, but the rate is a function of the slack accumulated along the ray.  In fact, this function is exponentially decreasing in the amount of slack accumulated.  
Along ``high slack'' segments of the ray, the height can only change a uniformly bounded amount, also divided by the exponential of slack accumulated.  So the only way for the height to go to infinity is to spend an exponential amount of time near leaves of $\lambda$ compared to the slack that has been accumulated, over and over. 

\begin{proof}
    We apply the Partition Lemma \ref{lemma: partition lemma multicurve case} to a ray $r$ with height going to infinity.  We write the image of $r$ as a concatenation of segments
    \[H_1 L_1 H_2 L_2 \dots,\]
    where each $H_i$ is an `$H$' segment and each $L_i$ is an `$L$' segment (from the Partition Lemma \ref{lemma: partition lemma multicurve case}).  Here we recall the features that we need:  there are positive constants $\epsilon_0$, $\epsilon_1$, and $p$ such that 
    \begin{itemize}
        \item $0 <\epsilon_0 \le p\cdot \ell(H_i) \le \s(H_i)$ holds for all $i$.
        \item $L_i$ is a translate of a segment of the form $[p, b_{j_i}^{M_i}p']$, where $p$ and $p'$ are $\epsilon_1$-close to the axis of $b_{j_i}$, and $\ell(L_i)\asymp M_i$.         
    \end{itemize}
    By definition of slack, $\s(H_i) \le \ell(H_i) + \ell(\overline{H_i}) \le \kappa \ell(H_i)$, where $\kappa$ depends only on $d_\rho$.  By the continuity property of slack (Lemma \ref{lem: slack continuous finite}) and the description of $L_i$ from the second item, $0\le \s(L_i)\le 2\epsilon_1$.
    
    Let $r_i = H_1L_1\dots H_i L_i$ and $t_i = \ell(r_i)$.  
    Now compute 
    \begin{align*}
        \frac{\s(r_n)}{t_n} &= \frac{\sum_{i =1}^n \s(L_i) + \sum_{i=1}^n\s(H_i)}{\sum_{i=1}^n\ell(L_i) + \sum_{i=1}^n\ell(H_i)}\\
        &\le \frac{n\epsilon_1 +\kappa\sum_{i =1}^n\ell(H_i)}{b\sum_{i=1}^nM_i + \sum_{i=1}^n\ell(H_i)}.
    \end{align*}
    For convenience, denote
    \[A_n = \sum_{i =1}^n M_i \text{ and } B_n = \sum_{i=1}^n\ell(H_i).\]
    Then we have 
    \[\frac{\s(r_n)}{t_n} \le \frac{n\epsilon_1/B_n +\kappa}{bA_n/B_n+1}.\]
    Since the $H$-segments all have positive length bounded below, $n\epsilon_1/B_n$ is uniformly bounded.
    Thus
    \[\liminf_{n\to \infty} \frac{\s(r_n)}{t_n} =0 \Leftrightarrow \limsup_{n\to \infty} \frac{A_n}{B_n} = \infty.\]

    In Theorem \ref{thm: heights bounded below scc},  we proved \eqref{eqn: upper and lower height bound} that  
    \[h(r(t_n))
    \le
    O(1)+ const \sum_{i=1 }^{n} M_i e^{-\s(r_{i-1}) - \s(H_{i-1})}
    \]   
    Since the left-hand quantity is unbounded so too is the right-hand quantity. 
    Notice, using additivity of slack and our bounds from above that 
    \[\sum_{i=1 }^{n} M_i e^{-\s(r_{i-1}) - \s(H_{i-1})}\prec \sum_{i =1}^n \frac{M_i}{e^{pB_i}}.\]
    Thus the last quantity
    is not bounded. 
    \begin{claim}\label{clm: AI claim}
        \[\limsup_{n\to \infty}\frac{A_n}{B_n} = \infty\]
    \end{claim}

    \begin{proof}[Proof of the claim]
        Suppose not.  Then there is a $C$ such that $M_n\le A_n\le CB_n$ for all sufficiently large $n$.  Hence also \[M_ne^{-pB_n} \le CB_n e^{-pB_n} \le Ce^{-pB_n/2},\]
        for all $n$ large enough.  
        Since $B_n$ is a sum of $n$ positive numbers admitting a global lower bound,  we know that $\sum B_ne^{-pB_n}$ is bounded by a geometric series, hence converges.  This contradicts the claim above that $\sum M_ie^{-pB_i}$ diverges, proving the claim.
    \end{proof}

    The Proposition is proved.
\end{proof}

Now we can conclude the proof of the main result of this section.
\begin{remark}
    If we knew that the analogue of Proposition \ref{prop: average slack and height} is true for $d =2$ and $\uu\lambda$ is arbitrary, then the following proof would apply without any changes.
\end{remark}

\begin{proof}[Proof of Theorem \ref{thm: positve lebesgue}]
    
    Since $\Lambda_+(\eta)\subset \Lambda(\eta)$, by Lemma \ref{lem: measure of union}, it suffices to prove that $m_{d-1}(E)=0$ (hence that $m_{d-1}(F)>0)$.  By Proposition \ref{prop: average slack and height}, if $r \in E$, then $\liminf_{t\to \infty} \frac{\s[o,r(t)]}{t} =0$.  We will show that this is a measure zero condition.  

    The map \[\pi^+:T^1\Omega \to \partial^\infty\Omega\]
    given by following $v \in T^1\Omega$ forward to $v^+ = \pi^+(v)\in \partial^\infty\Omega$ is a surjective $C^1$-submersion.  
    Thus, any Lebesgue class measure on $T^1\Omega$ pushes forward via $\pi^+$ to a (not necessarily locally finite) measure on $\partial^\infty \Omega$ in the class of Lebesgue.  What is relevant for us is that the pushforward measure has the same null sets as any measure in the Lebesgue class on $\partial^\infty \Omega$. 
    
    Let $W^sE = (\pi^+)\inverse (E)$.  Then $m_{d-1}(E)=0$ if and only if the Lebesgue measure of $W^sE$ on $T^1\Omega$ is also zero. 
    
    Observe that $W^sE$ is $g_t$-invariant, $\rho$-invariant, and is a union of the weakly stable leaves of the geodesic flow whose futures share a common endpoint in $E$.  
    Using the definitions and the continuity property of the slack functional given by Lemma \ref{lem: slack continuous finite}, we therefore have that 
    \[\liminf_{t\to \infty}\frac{\s[v, g_tv]}{t}= 0,\]
    for all $v \in W^sE$.  

    The set $W^sE$ descends to a $\uu g_t$-invariant set $\uu {W^sE}$ on the compact manifold $T^1\Omega/\rho$.  The geodesic flow $\uu g_t$ is $C^{1+\epsilon}$ and Anosov (Theorem \ref{thm: flow Anosov}).  Classical results of Ruelle and Bowen \cite{BowenRuelle} (see also \cite[\S8]{Crampon:Hilbertflow} or \cite[\S2.2]{PotrieSambarino}) give the following:  there is a unique $\uu g_t$ invariant ergodic probability measure $\nu$ with full support on $T^1\Omega/\rho$ such that for Lebesgue almost every point $\uu v \in T^1\Omega/\rho$ and every  continuous function $f : T^1\Omega/\rho \to \R$, we have 
    \[\frac{1}{T}\int_0^Tf(\uu g_t\uu v)~dt \to \int f~d\nu, ~ \text{ as } T\to \infty.\]
    Note that in spite of this, $\nu$ is \emph{not} in the Lebesgue measure class, unless $\Omega$ is an ellipsoid 
    \cite{Benoist_CDI}. 
    \medskip

    From \S\ref{sec: slack}, we see that 
    \[\s[\uu v,\uu g_t \uu v] = \int_0^t 1-L_\delta(\uu g_s \uu v)~ds +O(1).\]
    Also, $1-L_\delta$ is a continuous function on $T^1\Omega/\rho$ that is non-negative and strictly positive away from $T^1_+\uu \lambda$.
    Since $\nu$ has full support and $T^1_+\uu\lambda$ is a proper closed set, we conclude that \[\int 1-L_\delta ~d\nu >0.\]
    So for Lebesgue almost every $ \uu v \in T^1\Omega/\rho$, 
    \[\frac{\s[\uu v, \uu g_T\uu v]}{T} \to \int 1-L_\delta ~d\nu \not=0.\]
    This proves that $\uu {W^sE}$, hence $W^sE$, has Lebesgue measure $0$, and concludes the proof of the theorem.
\end{proof}

\section{Questions}\label{sec: questions}

We collect a few questions here. For each question, at least one of the authors would like to know the answer. Wherever possible, we use only notation defined in the introduction of the paper.
\medskip

It's possible that the tools of \cite{Panetal:I,Panetal:II} may be useful for tackling the following question.  
\begin{question}   
    Does the Hausdorff dimension of the limit set of coaffine representations in and on the boundary of $\mathcal A$ satisfy any (semi-)continuity property?  In particular, if $\eta_n\in \mathcal A$ are coaffine representations converging to some unexceptional $\eta\in \partial \mathcal A$, then is it true that the Hausdorff dimensions of the limit sets converge to that of $\eta$?  
\end{question}

Observe that if $\eta_n$ are representations satisfying hypotheses \ref{hyp A} of the introduction which are unexceptional and converge to a representation with trivial cocycle, then the Hausdorff dimension jumps down by one in the limit.  

In the case $d=2$ and the tension lamination $\uu \lambda$ is a multicurve with only a few components relative to the genus of $S$, we can guarantee by direct computation in $H^1(\Gamma,(\R^{d+1})^*_{\rho\chi})$ that there are positive cocycles. In general, we would like to know more about the cone of positive cocycles.

\begin{question}
    Under what circumstances is the cone $\sP_{\rho\chi}$ non-empty? Under hypotheses \ref{hyp B}, is $\sP_{\rho\chi}$ a polyhedral cone (i.e.\ finite-sided)? Under what circumstances is the cone sharp (contains no non-trivial subspace in its closure)?
\end{question}

We view the coaffine space $\cas \subset \RP^{d+1}$ as the primary geometry in this paper. There is a dual geometry, which is usual affine geometry, $\mathbb A^{d+1}$. The minimal domain $C_{min}$ which we studied is dual to a maximal domain $C_{max}\subset \mathbb A^{d+1}$. It would be interesting to understand the structure of this domain and its boundary.
\begin{question}
     What significance do the tension locus $\lambda$ and the bending lamination $\beta$ have in the affine space? How are representations in $\partial \mathcal A$ characterized geometrically in this setting?
     See the related work of Ablondi \cite{Ablondi}.
\end{question}

A representation $\rho: \Gamma \to\PGL(n,\R)$ of a finitely generated group $\Gamma$ is said to be $1$-divergent if for any sequence $\gamma_n \in \Gamma$ of distinct elements, we have 
\[(\sigma_1 - \sigma_2)(\rho(\gamma_n)) \to \infty , ~ n \to \infty.\]

\begin{question}
    Are the unexceptional coaffine representations on the boundary of the Anosov locus considered in this paper $1$-divergent? 
\end{question}
The exceptional representations under hypotheses \ref{hyp A} are not $1$-divergent.
\medskip

As discussed in the introduction, in the case that $\eta\in \partial \mathcal A$ is positive, the previous work in \cite{BF} gives a bending lamination $\beta$ on the boundary of $C_{min}$ in $\sC \Omega$. Drawing on the analogy to the ending lamination theorem \cite{ELTI, ELTII}, in which the conformal structure at the non-degenerate end together with the ending lamination determine the representation up to conjugacy, we ask the following.

\begin{question}
    When $d=2$ and $\Gamma$ is a surface group, 
    given $\eta\in \sP_{\rho\chi}$, do the tension lamination $\uu\lambda$ and bending lamination $\beta$ determine $\eta$ up to conjugacy?
\end{question}

Corollary \ref{cor: epigraph}, which gives $\lambda^+ \subset \mathcal{H} \inverse( +\infty)$ (when $\eta\in \sP_{\rho\chi}$), leaves a loose end: 
\begin{question}\label{q: Hinverse}
    For $\eta \in \sP_{\rho\chi}$, what is a precise characterization of the points in  $\mathcal H\inverse (+\infty)$?
\end{question}

As discussed in Main Theorem \ref{mainthm: structure of limit sets} and its corollary, we construct examples of representations $\eta$ which have an interesting constellation of properties. In particular, we give discrete and faithful representations $\eta$ such that
\begin{enumerate}
    \item $\eta$ preserves a minimal properly convex domain $C_{min}\subset \RP^{d+1}$,
    \item the action of $\eta$ on $\RP^{d+1}$ admits a global fixed point $p$, and
    \item the point $p \in \partial C_{min}$.
\end{enumerate}  

\begin{question}
    Are there other discrete and faithful representations in the coaffine group that satisfy the properties above and are not in $\overline{\mathcal{A}}$?
\end{question}

We do not have a geometric characterization of the minimal convex domain of discontinuity to distinguish these two types of representation, though their dynamical behavior is different.

\appendix
\section{Proof of Theorem \ref{thm: co-stable norm and optimal best maps}}\label{appendix}

We restate the result here for convenience of the reader.
Recall that $\Omega \subset \RP^d$ is a strictly and properly convex domain, $\rho: \Gamma \to \SL(d+1,\R)$ is an injective representation that divides $\Omega$ without torsion, and $M = \Omega/\rho$ is the corresponding closed projective manifold.

\begin{thmrepeated}
    Let $\varphi\in \partial \mathcal B_\rho \subset H^1(M,\R)$, i.e., $\|\varphi\|_\rho =1$.
    There is a $\varphi$-equivariant $1$-forward Lipschitz map $\tau:\Omega \to \R$ with the following properties.
    \begin{enumerate}

        \item $\Tension (\tau)\subset \Omega$ is a non-empty $\rho$-invariant geodesic lamination, that descends to an oriented chain recurrent geodesic lamination $\Tension(\tau)/\rho\subset M$.
        \item $\tau$ is \textbf{optimal} in the sense that 
        \[\Tension (\tau) = \Tension (\varphi) \defeq \bigcap_{f} \Tension(f),\]
        where the intersection is taken over all $\varphi$-equivariant $1$-forward Lipschitz maps. 
        \item If $\{\gamma_n\}$ is any sequence in $\Gamma$ satisfying 
        \[\lim_{n \to \infty} \frac{\varphi(\gamma_n)}{\lambda_1^\rho(\gamma_n)} =1,\]
        then every accumulation point of $\{\gamma_n\}$ in the measure topology has support that is contained in $\Tension(\varphi)/\rho \subset M$.  
    \end{enumerate}
    In particular $\|\varphi\|_\rho = \min_{f\sim \varphi}\vec \Lip(f) = \vec\Lip(\tau)$.
\end{thmrepeated}

This appendix is devoted to giving a (mostly) self-contained proof of this result and its relationship to the (co)stable norm on (co)homology.

\subsection{An extension theorem for directional Lipschitz maps}
The following classical extension theorem is due to McShane \cite{McShane} in the context of real-valued Lipschitz mappings on metric spaces.  The proof adapts easily to this setting. 
\begin{theorem}\label{thm: extension}
    Let $Y\subset X$ be any non-empty subspace and suppose $f: Y \to \R$ is $L$-forward Lipschitz.  There is an $L$-forward Lipschitz extension $\phi: X \to \R$, i.e., $\phi|_Y = f$ and $\phi$ is $L$-forward Lipschitz. 
\end{theorem}

\begin{proof}
    Define 
    \[\phi(x) = \sup_{y\in Y} f(y) - Ld(x,y),\]
    and observe that if $x\in Y$, then for all $y \in Y$
    \[f(y)-f(x)\le Ld(x,y) \Leftrightarrow f(y) - Ld(x,y)\le f(x)\]
    by the $L$-forward Lipschitz property of $f$.  Thus $\phi(x) = f(x)$, when $x \in Y$.
    Now let $x_1,x_2 \in X$ be arbitrary.
    \begin{align*}
        \phi(x_1) - \phi(x_2) &= \sup_{y_1\in Y} \left(f(y_1) - Ld(x_1,y_1)\right) -\sup_{y_2\in Y} (f(y_2) - Ld(x_2,y_2))\\
        &\le \sup_{y\in Y}( f(y) - Ld(x_1,y) - f(y) +Ld(x_2,y)) \le Ld(x_2,x_1),
    \end{align*}
    where the last inequality follows from the triangle inequality
    \[d(x_2,y) \le d(x_2, x_1) +d(x_1,y),\]
    completing the proof.
\end{proof}

\subsection{The stable norm and co-norm}
We now define a function on homology $H_1(M,\Z)$ that induces, via homogeneity and continuity, an (asymmetric) norm on $H_1(M,\R)$.  Later, we relate this asymmetric norm to the function $\varphi\mapsto \|\varphi\|_\rho$ on cohomology by duality.

Recall that abelianization provides a natural surjection $\gamma\in \Gamma \mapsto [\gamma]\in H_1(M,\Z)$.  For $c \in H_1(M,\Z)$, define
\[f(c) = \inf_{\gamma\in \Gamma \text{ and } [\gamma]=c} \lambda_1^\rho(\gamma).\]
By Theorem \ref{thm: rhodist}, $\lambda_1^\rho(\gamma)$ is the length of the geodesic representative of the free homotopy class of $\gamma$ in $M$ for the metric $d_\rho$.
While $f$ need not be symmetric, it is readily seen to have the following properties
\begin{enumerate}
    \item $f(c) \ge 0$ with equality only if $c =0$.
    \item $f(nc) \le nf(c)$
    \item $f(c_1+c_2) \le f(c_1) +f(c_2) +2D_\rho$, where $D_\rho = \max_{p,q \in M} d_\rho(p,q)$ is the diameter of $M$ with the $d_\rho$ metric.
\end{enumerate}

\begin{proposition}
    There is a (not necessarily symmetric) norm $ \|\cdot \|_\rho$ on $H_1(M,\R)$ continuously extending the rule
    \[\|c\|_\rho = \lim_{n\to \infty}\frac{f(nc)}n\]
    where $c$ is an integer point of $H_1(M,\R)$.
\end{proposition}
\noindent The function $\|\cdot \|_\rho$ on $H_1(M,\R)$ from the proposition is called the \textbf{stable norm} (even when it is not symmetric).
The details of the proof of the proposition can be found, e.g., in \cite[Proposition 2.1]{Bangert:minimalgeodesics},  \cite[\S1.1]{Massart:stablenorm}, or \cite[\S3]{AL:stablenorm}, but we give an outline, since these references explicitly state the result for Riemannian metrics.  Note however that Bangert has studied the stable norm for (not necessarily symmetric) \emph{smooth} Finsler metrics \cite{Bangert:equiv, Bangert:calibrations}. 
\begin{proof}
First, that the limit defining $\|c\|_\rho$ for an integral class $c$ exists (and is equal to $\inf_n\frac{f(nc)}n$) follows from items 2 and 3.
For a positive integer $m$ and integer class $c$, it is clear also that $\|mc\|_\rho = m\|c\|_\rho$.
From item 3, we deduce that 
\[\|c_1+c_2\|_\rho \le \|c_1\|_\rho + \|c_2\|_\rho\]
holds for all $c_1, c_2\in H_1(M,\Z)$.
These properties hold just as well after extending the domain of $\| \cdot \|_\rho$ to the rational points $H_1(M,\Q)$ by homogeneity via division/scaling by \emph{positive} integers. 

To see that $\|\cdot \|_\rho$ extends continuously to $H_1(M,\R)$, choose a basis $\{e_1, ..., e_{b_1}\}$ for $H_1(M,\Z)$.  Given $v \in H_1(M,\Q)$, write $v = \sum a_i e_i$ for rationals $a_i$ and observe 
\[\|v\|_\rho \le \sum|a_i| \cdot \max_{1\le i \le \dim(H_1(M,\mathbb Q))} \{\|e_i\|_\rho, \|-e_i\|_\rho\}=  \|v\|_1K, \]
where $\|v\|_1$ is the $1$-norm associated to the chosen basis, and $K$ is a constant.
From this Lipschitz continuity property, one deduces that $\|\cdot \|_\rho$ extends continuously to a homogeneous function on $H_1(M,\R)$ that satisfies the triangle inequality.  

The only outstanding assertion is that $\|v \|_\rho = 0$ if and only if $v =0$.
\begin{claim}\label{clm:infinity1}
    Let $\omega$ be a smooth closed $1$-form on $M$.  Then 
    \[\langle \omega, v\rangle\le \|v\|_\rho\cdot \|\omega\|_\infty\]
\end{claim}
\begin{proof}[Proof of the claim]
    First we show the claim for $c \in H_1(M,\Z)$.
    Let $\gamma_n \subset \Gamma$ be such that $\|c\|_\rho = \inf \frac{\lambda_1^\rho (\gamma_n)}{n}$ and $\gamma_n /n \in [c]$.
    Then 
    \[\langle \omega, c\rangle  = \frac 1n \int_{\gamma_n} \omega \le \|\omega\|_\infty \frac{\lambda_1^\rho(\gamma_n)}n \]
    holds for all $n$.  In particular, we have 
    \[\langle \omega, c\rangle \le \|\omega\|_\infty \cdot \|c\|_\rho. \]
    The claim follows for arbitrary $v \in H^1(M,\R)$ by homogineity and continuity of $\|\cdot \|_\rho$.  
\end{proof}
So, if $v \in H_1(M,\R)$ has $\|v\|_\rho = 0$, then by the claim, $\langle \omega, v\rangle =0$ for all closed one forms $\omega$.  Since closed $1$-forms span the dual space to $H_1(M,\R)$ (de Rham's Theorem), this implies  that $v =0$.
\end{proof}

That the stable norm does not deviate too far from $f$ on $H_1(M,\Z)$ is the content of the next result.
The idea for this theorem goes back to Burago \cite{Burago:periodic}, who proved the result for Riemannian metrics on tori.  
The proof given in Ammann and Löh is for all compact Riemannian manifolds.  A careful reading of their proof shows that it applies in our setting as well, without modifications.  
\begin{theorem}[Propositions 3.3 and 4.5 of \cite{AL:stablenorm}]\label{thm: doubling}
    There is a constant $D$ such that 
     \[f([\gamma]) \le \|[\gamma]\|_\rho +D\]
    holds for all $\gamma \in \Gamma$.
\end{theorem}

For $\varphi \in H^1(M,\R)$, recall the definition of $\|\varphi\|_\rho$ from \eqref{eqn: def costable}.  The following lemma asserts that this quantity is dual to the stable norm.  Its proof is not difficult, and is left as an exercise to the reader.  
\begin{lemma}\label{lem: co-stable is dual}
    We have an equality
    \[\|\varphi\|_\rho = \max_{v\not=0\in H_1(M,\R)} \frac{\varphi(v)}{\|v\|_\rho}.\]
\end{lemma}
\noindent The quantity $\|\cdot \|_\rho$ on $H^1(M,\R)$ is called the \textbf{co-stable norm}.

\subsection{Best forward Lipschitz maps and tension loci}

The feature of the co-stable norm that is most interesting for us is that it identifies the best Lipschitz constant among forward Lipschitz maps satisfying an equivariance condition.  It proves that the set over which the intersection is taken in item 2 of Theorem \ref{thm: co-stable norm and optimal best maps} is non-empty.
The argument given here is adapted from \cite[Proposition 2.2]{BIK:stablenorm}.
\begin{theorem}\label{thm: best Lipschitz achieved}
    Given any $\varphi\in \partial \mathcal B_\rho\subset H^1(M,\R)$, there is a $\varphi$-equivariant 1-forward Lipschitz map $\Omega \to \R$. 
\end{theorem}

First, we relate the stable norm to the large scale geometry of a certain covering space of $M$.  Let $\widehat M\to M$ be the cover with deck group $H_1(M,\Z)$ with corresponding metric $\widehat d_\rho$, i.e., $\widehat M$ is the \textbf{universal abelian cover} and the projection is locally isometric.
The main ingredient of the following result is Theorem \ref{thm: doubling}; see also \cite[Proposition 4.7]{AL:stablenorm}.

\begin{proposition}\label{prop: metric close to stable norm}
    There is a constant $C$ such that for any point $o\in \widehat M$, and any $g,h \in H_1(M,\mathbb Z)$, we have 
    \[\left| \widehat d_\rho(g.o, h.o) -\|h-g\|_\rho\right|<C.\]
\end{proposition}

\begin{proof}
    There is a directed path joining $g.o$ to $h.o$ of length $\widehat d_\rho(g.o,h.o)$, which projects to a closed curve in $M$ in the homology class of $h-g$, giving an upper bound for $f(h-g)$.  Since $\|h-g\|_\rho\le f(h-g)$ follows from the definitions, we have 
    \[\|h-g\|_\rho \le \widehat d_\rho(g.o,h.o).\]

    Now, let $\gamma$ be a geodesic in $M$ satisfying 
    \[f(h-g) = \lambda_1^\rho(\gamma).\]
    There is a lift $\widehat \gamma \subset \widehat M$ of $\gamma$ that passes within $D_\rho$ of $g.o$.  Moreover, since $[\gamma] \in h-g$, there is a directed segment of length $\lambda_1^\rho(\gamma)$ that starts within $D_\rho$ of $g.o$ and ends within $D_\rho$ of $h.o$.  Thus 
    \[\widehat d_\rho(g.o,h.o)\le 2D_\rho+f(h-g)\le 2D_\rho +D+\|h-g\|_\rho,\]
    by Theorem \ref{thm: doubling}, and the Proposition is proved with $C = 2D_\rho+D$.
\end{proof}

\begin{proof}[Proof of Theorem \ref{thm: best Lipschitz achieved}]
  Choose $o \in \widehat M$, and define a function
\[F(p) = \liminf_{h\in H_1(M,\Z), ~ \|h\|_\rho \to \infty} \widehat d_\rho(o, h.p) - \varphi(h)\]
For every $h$, the map $F_h: p\mapsto  \widehat d_\rho(o, h.p) - \varphi(h)$ is $1$-forward Lipschitz as a consequence of the triangle inequality.
If $F$ is finite, then it is a pointwise limit  of $1$-forward Lipschitz maps, hence itself $1$-forward Lipschitz.

Now we would like to show that $F$ is finite.
Given $p\in \widehat M$, choose $g \in H_1(M,\Z)$ so that $g.o$ is close to $p$.  Then 
\[\max\{\widehat d_\rho ((h+g).o, h.p), \widehat d_\rho (h.p, (h+g).o)\} \le D_\rho.\]
By the triangle inequality, we have 
\[\widehat d_\rho(o,(h+g).o) - \widehat d_\rho(h.p, (h+g).o) \le \widehat d_\rho(o,h.p) \le \widehat d_\rho(o,(h+g).o) + \widehat d_\rho((h+g).o, h.p)\]
or in other words
\[\widehat d_\rho(o,hp) =\widehat d_\rho(o,(h+g).o) +O(1).\]
By Proposition \ref{prop: metric close to stable norm}, we learn that 
\[F_h(p) = \|h+g\|_\rho -\varphi(h) +O(1).\]

In any (asymmetric) normed space, we have 
\[\|h\|\le \|h+g-g\| \le \|h+g\| +\|-g\|,\]
from which we conclude that $\|h+g\| - \|h\| \ge -\|-g\|$.

Since $\|\varphi\|_\rho = 1$, Lemma \ref{lem: co-stable is dual} gives $\varphi(h)\le \|h\|_\rho$, and so 
\begin{align*}
    F_h(p) = \|h+g\|_\rho -\varphi(h) +O(1) &\ge \|h+g\|_\rho - \|h\|_\rho +O(1).\\
    &\ge -\|-g\|_\rho +O(1).
\end{align*}
Since $g$ depended only on $p$, we see that $F_h(p)$ is bounded below, independently of $h$.

Now we'd like to see that $F_{h_n}(p)$ remains bounded above along some sequence $\{h_n\}\subset H_1(M,\Z)$.
For this, find $c \in H_1(M,\R)$ such that $\varphi(c) =\|c\|_\rho = 1$, which is possible according to Lemma \ref{lem: co-stable is dual}.
Since the integer points form a lattice in $H_1(M,\R)$, we can find $\delta_n \in H_1(M,\R)$ satisfying 
\[h_n = nc+\delta_n \in H_1(M,\Z) \text{ and } \max\{\|\delta_n\|, \|-\delta_n\|\}\le K, \]
for some $K$.
In particular, we have $\varphi(h_n) = \|h_n\|+O(1)$.

Then for the same $g$ depending only on $p$, as above, we have
\begin{align*}
    \widehat d_\rho(o, h_np) -\varphi(h_n) &=  \|g+h_n\| -\varphi(h_n)+O(1)\\
    & = \|g+h_n\| - \|h_n\| +O(1)\le \|g\|+O(1).
\end{align*}
Since $F(p) \le \liminf F_{h_n}(p) \le \|g\|+O(1)$, we conclude that $F(p)$ is finite, hence $1$-forward Lipschitz.

That $F$ is $\varphi$-equivariant follows by computation:  for any $g \in H_1(M,\Z)$, 
\begin{align*}
    F(g.p) &= \liminf_{\|h\|\to \infty} \widehat d_\rho (o,(h+g).p) -\varphi(h) \\
    & = \liminf_{\|h\|\to \infty} \widehat d_\rho (o,(h+g).p) -\varphi(h +g) +\varphi(g) \\
    & = \liminf_{\|h+g\|\to \infty} \widehat d_\rho (o,(h+g).p) -\varphi(h +g) +\varphi(g) = F(p)+\varphi(g).
\end{align*}    
This completes the proof of the Theorem.
\end{proof}

Any $\varphi$-equivariant $1$-forward Lipschitz map $F: \widehat M\to \R$ lifts to a locally (hence globally) $1$-forward Lipschitz map $ \Omega \to \R$.  Thus, as an immediate corollary of Theorem \ref{thm: best Lipschitz achieved} and the discussion surrounding \eqref{eqn: norm less than lipschitz}, we obtain:
\begin{corollary}\label{cor: co-norm is smallest Lipschitz}
    $\|\varphi\|_\rho$ is the smallest forward Lipschitz constant of any $\varphi$-equivariant map $\Omega \to \R$.
\end{corollary}

\begin{remark}\label{rmk: approximation by smooth maps}
    For a closed $1$-form $\omega$ representing $\varphi$, the function $f_\omega : \Omega \to \R$ had $\vec \Lip (f_\omega)\le \|\omega\|_\infty$.  By a standard smoothing argument, there is a sequence of equivariant smooth functions converging to any best forward Lipschitz map whose forward Lipschitz constants converge, as well.  Thus $\|\varphi\|_\rho = \inf_{\omega \in \varphi} \|\omega\|_\infty$, as well.
\end{remark}

Let $\varphi \in \partial \mathcal B_\rho$, and let $F: \Omega \to \R$ be any $\varphi$-equivariant (best) $1$-forward Lipschitz map.
Our next goal is to show that $\Tension (F)\not=\emptyset$.  
For this, it will be convenient to consider a function  $L_\delta : T^1_\rho M \to \R$.
By abuse of notation, $F: T^1_\rho \Omega \to \R$ is the pullback via the projection $\pi: T^1\Omega \to \Omega$.
Let $\delta>0$ be smaller than the injectivity radius of $M$.  For any $x \in T^1_\rho M$, choose a lift $\tilde x$ to $T_1^\rho \Omega$ of $x$ and define 
\begin{equation}\label{eqn: def L_delta}
    L_\delta(x) = \frac{F(g_\delta \tilde x) - F(\tilde x)}{\delta},\
\end{equation}
where $g_t$ is the geodesic flow.
By $\varphi$-equivariance, this definition does not depend on the choice of $\tilde x$.  
Since $F$ is $1$-forward Lipschitz and the geodesic flow is continuous, $L_\delta$ is continuous and $L_\delta \le 1$.
Indeed, if $L_\delta(x) =1$, then $\pi(g_{[0,\delta]}x)$ is a directed segment in $M$ which is mapped isometrically to $\R$ with the correct orientation; in particular,  $\pi(g_{[0,\delta]}x)$ is contained in $\Tension(F)/\rho$. 

The following is similar to an argument found in \cite[\S3]{FLM:I}.  Its proof yields item 3 of Theorem \ref{thm: co-stable norm and optimal best maps}; see also Theorem \ref{thm: failure of anosov along lamination} and its proof.  
\begin{lemma}\label{lem: Tension is non-empty}
    The set $L_\delta\inverse(1)$ is non-empty and closed.  
\end{lemma}
\begin{proof}
    Let $\gamma_n \subset M$ be a sequence of (the unique) geodesic representatives of oriented curves in $M$ for which 
    \[\|\varphi\|_\rho = \lim_{n\to \infty} \frac{\varphi(\gamma_n)}{\lambda_1^\rho(\gamma_n)}=1.\]
    Let $\mu_n$ be the $g_t$-invariant Borel probability measures supported on the canonical lifts of $\gamma_n$ to $T^1_\rho M$. 
    Then 
    \[\int L_\delta ~d\mu_n = \frac{1}{\lambda_1^\rho(\gamma_n)}\int_0^{\lambda_1^\rho(\gamma_n)} L_\delta(\gamma'_n(t))~dt = \frac{1}{\lambda_1^\rho(\gamma_n)}\int_0^{\lambda_1^\rho(\gamma_n)} \frac{F(g_\delta \tilde \gamma_n'(t)) - F(\tilde\gamma_n'(t))}{\delta} ~dt. \]
    
    By compactness of $M$, there is a constant $K$ such that for all $p,q \in \Omega$, $d_\rho(p,q)\le Kd_\rho(q,p)$.
    Thus any $1$-forward Lipschitz function is $K$-Lipschitz for the symmetrized metric.
    For any $x\in T_\rho^1M$, the function $[-T, T] \to \R$ defined by $t\mapsto F(g_tx)$ is therefore $K$-Lipschitz, hence absolutely continuous, differentiable almost everywhere, and the integral of the derivative is the change in function values.  This gives,
    \[ F(g_\delta\tilde \gamma_n'(t))- F(\tilde\gamma_n'(t)) = \int_{0}^\delta \left.\frac{d}{du}\right|_{u=0} F(g_{s+u}\tilde \gamma_n'(t))~ds.\]
    Using the previous line and Fubini, we get
    \[\int L_\delta~d\mu_n = \frac{1}{\lambda_1^\rho(\gamma_n)}\frac{1}{\delta}\int_{0}^\delta\int_0^{\lambda_1^\rho(\gamma_n)} \left.\frac{d}{du}\right|_{u=0}F(g_{s+u}\tilde \gamma_n'(t))~dt ~ds. \]
    But for each fixed $s$, the integral 
    \[\int_0^{\lambda_1^\rho(\gamma_n)} \left.\frac{d}{du}\right|_{u=0}F(g_{s+u}\tilde \gamma_n'(t))~dt = \varphi(\gamma_n), \]
    because $F$ is $\varphi$-equivariant, and the two arguments differ by translation by $\gamma_n$.
    Thus,
    \[\int L_\delta~d\mu_n = \frac{\varphi(\gamma_n)}{\lambda_1^\rho(\gamma_n)}\frac{1}{\delta}\int_{0}^\delta ds = \frac{\varphi(\gamma_n)}{\lambda_1^\rho(\gamma_n)}\to 1.\]
    Thus, for any weak-$*$ accumulation point $\mu$ of the sequence $\mu_n$ in the space of (flow invariant) probability measures on $T^1_\rho M$ and by continuity of $L_\delta$, we have 
    \[\int L_\delta ~d\mu =1.\]
    Since $L_\delta\le1$ and $\mu$ is a probability measure, we conclude that $L_\delta =1$ $\mu$-almost everywhere.  Since the support of $\mu$ is closed, we get $L_\delta =1$ on the support of $\mu$.  This concludes the lemma and also the proof of item 3 of Theorem \ref{thm: co-stable norm and optimal best maps}.
\end{proof}

\subsection{Optimal best Lipschitz maps}

We have shown that for every $\varphi\not=0 \in H^1(M,\R)$, there is a $\varphi$-equivariant $\|\varphi\|$-forward Lipschitz map 
\[ F: \widetilde M \to \R,\]
and that the tension locus of $F$ is non-empty. 
For $\|\varphi\|=1$, the \emph{optimal tension locus} is
\[\Tension (\varphi) : =\bigcap_{F\in \varphi}\Tension (F),\]
where $F\in \varphi$ means that $F$ is a $\varphi$ equivariant  $1$-forward Lipschitz map, as above.

In this section, we adapt arguments from \cite{GK:stretch} and \cite{FLM:II} to prove that the optimal tension locus is highly structured.
\begin{theorem}\label{thm: optimal map}
    Whenever $\varphi \in H^1(\Gamma,\R)$ and $\|\varphi\|_\rho = 1$, then $\Tension (\varphi)$ is a non-empty, oriented, chain recurrent geodesic lamination and there is an \textbf{optimal} $1$-forward Lipschitz $\varphi$-equivariant map 
    $\tau: \Omega \to \R$, i.e.,
    \[\Tension(\tau) = \Tension (\varphi).\]
\end{theorem}

Note that there is not a unique optimal best Lipschitz map, as one may make small perturbations away from $\Tension(\varphi)$. 

First, we produce an optimal map by averaging.  The argument comes directly from \cite{GK:stretch} with no substantive  modifications. 
\begin{lemma}\label{lem: optimal maps exist}
    There is an optimal $\varphi$-equivariant $1$-forward Lipschitz map $\tau: \Omega \to \R$.
\end{lemma}

\begin{proof}
    This is a simpler version of \cite[Lemma 4.13]{GK:stretch}.
    The set $\Tension (\varphi)$ is closed, as it is the intersection of closed sets.
    For every $p \in \Omega\setminus \Tension (\varphi)$, there is a $\varphi$-equivariant $1$-Lipschitz map $f_p$, an open neighborhood $U_p$ of $p$ in $\Omega\setminus \Tension (\varphi)$, and a $\delta_p>0$ such that 
    \[\sup_{q_1, q_2}\frac{f_p(q_2)-f_p(q_1)}{d(q_1, q_2)} = 1-\delta_p<1.\]
    Since $\Omega\setminus \Tension(\varphi)$ can be exhausted by compact sets, there is a countable collection $\{p_i\}$ such that $\cup U_{p_i} = \Omega\setminus \Tension (\varphi)$.  Let $\{a_i\}$ be a set of positive numbers summing to $1$ and define $\tau = \sum a_if_{p_i}$.
    It is not difficult to see that $\tau$ is $\varphi$-equivariant and 
    \[\vec \Lip_q(\tau) \le \sum a_i \vec\Lip(f_{p_i}),\]
    for all $q \in \Omega$.
    So, if $q \in U_{p_j}$, then 
    \[\vec \Lip_p(\tau) \le \sum_{i\not=j} a_i + (1-\delta_{p_j})a_j = 1-a_j\delta_{p_j}<1.\]
    This proves that $\Tension (\tau) \subset \Tension (\varphi)$, with the other inclusion being trivially satisfied.  
\end{proof}
We observe in the following Lemma that points in the taut locus are contained in at most one oriented geodesic.
\begin{lemma}\label{lem: lines dont cross}
    If $p \in \Tension (f)$ and if there are directed line segments $[x,p]$ and $[p,y]$ contained in $ \Tension(f)$, then $p$ is on the directed line segment $[x,y]$.
\end{lemma}

\begin{proof}
    Suppose for contradiction that $[x,y]$ were not in the union $[x,p]\cup[p,y]$. Then the three distinct points $x,y,p$ form a non-degenerate triangle in $\Omega$, and as a result $d_\rho(x,y)<d_\rho(x,p)+d_\rho(p,y)$ by Theorem \ref{thm: rhodist} (that the triangle inequality is strict for points not on a line). 
    
    However, since $[x,p]\cup[p,y]\subset \Tension(f)$, $f(y)-f(x)=f(y)-f(p)+f(p)-f(x)= d_\rho(x,p)+d_\rho(p,y)$. This contradicts the assumption that $f$ is $1$-forward Lipschitz.
\end{proof}

\begin{lemma}\label{lem: exists line}
    Suppose $p \in \Tension (\varphi)$, then there is a (non-degenerate) directed line segment $[x,y]\subset \Tension (\varphi)$ containing $p$.
\end{lemma}

\begin{proof}
    The structure of our proof follows \cite[Lemma 5.2]{GK:stretch}.
    By Lemma \ref{lem: optimal maps exist}, there is an optimal $1$-forward Lipschitz $\varphi$-equivariant map $\tau: \Omega \to \R$.
    Let $p \in \Tension (\tau)$.  Let $r$ be smaller than the injectivity radius of $M$ at $p$ (in both directions).  

    Define a closed \textbf{forward $r$-ball}
    \[\vec B_r(p) = \{ x: d(x,p)\le r\},\]
     a closed \textbf{backward $r$-ball}
    \[\cev B_r(p) = \{ x: d(p,x)\le r\},\]
    and let
        \[B = \vec B_r(p)\cup \cev B_r(p).\]
     Without loss of generality, we may assume that  
    \[0 = \min \tau(\vec B_r(p)) \text{ and } 2m= \max \tau(\cev B_r(p)).\]
    Note that $0< \tau(p) <2m\le 2r$, where the last inequality is immediate from the $1$-forward Lipschitz property of $\tau$.

    Our goal is to prove that $m=r$.
    Assuming that this is the case, we first claim that this implies that $\tau(p)=r$.
    Otherwise, without loss of generality $\tau(p)>r$, and there is a point $x\in \partial \vec B_r(p)$ such that $\tau(x)=0$. The pair $(x,p)$ contradicts the fact that $\tau$ is $1$-forward Lipschitz.
    Furthermore, there are a pair of points $(x,y)\in \partial \vec B_r(p)\times \partial \cev B_r(p)$ such that $\tau(x)=0$ and $\tau(y)=2r$. We claim that the segments $[x,p]$ and $[p,y]$ are necessarily isometrically mapped into $\R$ by $\tau$. To see this, consider the derivative $d\tau$, which is an almost everywhere defined function. Let $\gamma$ be a unit-speed parametrization of $[x,p]$ so that
    \[\tau(p) - \tau(x)=\int_{0}^rd\tau(\gamma'(t))dt =r.\]
    Since $\frac{d\tau\circ \gamma}{dt}$ is almost everywhere bounded above by $1$, we conclude that it is almost everywhere equal to $1$ on $\gamma'$. This shows that $\tau$ is an oriented isometry on $[x,p]$ and a similar argument implies the same for $[p,y]$. Lemma \ref{lem: lines dont cross} then shows that $p \in [x,y]$, completing the proof (assuming $m=r$).
    \medskip

    Suppose for the sake of contradiction that $m<r$, and let $0<\epsilon <r-m$. The strategy will be to produce a $1$-forward Lipschitz $\varphi$-equivariant map $f: \Omega \to \R$ that takes a neighborhood of $p$ to a constant, so that $p$ is not in the tension locus of $f$.  Unpacking the definitions, this implies that $p\not\in\Tension(\varphi) = \Tension(\tau)$, supplying a contradiction.  

    Now we construct $f$ in two steps.
    Let $B_\epsilon(p) = \vec B_\epsilon(p) \cap \cev B_\epsilon(p)$ and define
    \begin{align*}
        g: \partial B \cup B_\epsilon(p) \to &\R\\
        x \mapsto & \begin{cases}
            \tau(x), & x \in \partial B\\
            m, & x \in B_\epsilon(p).
        \end{cases}
    \end{align*}
    \begin{claim}
        $g$ is $1$-forward Lipschitz.
    \end{claim}
    \begin{proof}[Proof of the Claim.]
    
    \begin{figure}[h]
    \centering
    \includegraphics[width=.7\linewidth]{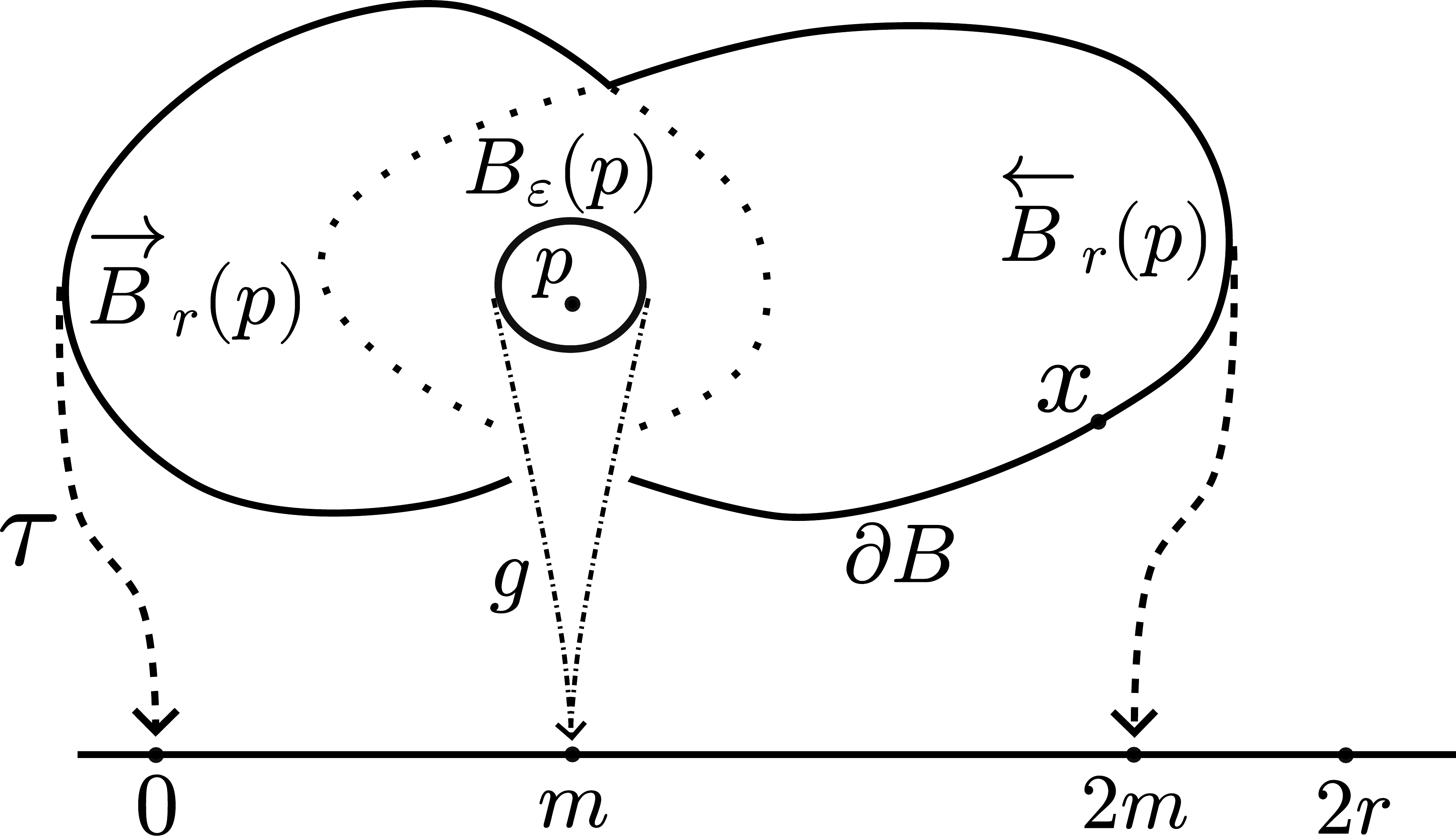}
    \caption{A schematic of the definition of $g$ in terms of $\tau$.}  
    \label{fig: lipschitz balls}
    \end{figure}

    We need to check four pairs of inequalities, but thankfully one inequality from each  pair is trivial.  Let $x \in \partial B$ and $p' \in B_\epsilon(p)$.  There are four cases:
    
    \begin{enumerate}
        \item $x \in \partial \vec B_r(p)$ and $\tau(x)\le m$.  In this case, we have
        \[0\le g(p') - g(x)= m-\tau(x)\le m <r-\epsilon = d(x,p)-\epsilon\le d(x,p'),\]
        and so $g(x)-g(p')\le 0 <d(p',x)$ holds trivially.
        \item $x \in \partial \cev B_r(p)$ and $\tau(x)\le m$.  The line segment $[p,x]$ meets $\partial \vec B_r(p)$ in a unique point $x'$.  Since the line segments for the domain shape metric are projective line segments (Theorem \ref{thm: rhodist}), we know that the geodesic line segments $[p,x]$ and $[x,p]$ coincide.  In particular, we have $d(x,x')+d(x',p) = d(x,p)$.  Now we estimate
        \begin{align*}
            0\le g(p') - g(x) &= g(p') - g(x')+g(x') - g(x) \\
            & \le  m - \tau(x')+d(x,x') \le m +d(x,x')\\
            &  <r-\epsilon +d(x,x')= d(x',p) - \epsilon +d(x,x')\\
            & =d(x,p)-\epsilon \le d(x,p').
        \end{align*}
        \item $x \in \partial \cev B_r(p)$ and $\tau(x)\ge m$.  This is similar to the first case.  Indeed,
        \[0\le g(x) - g(p') = \tau(x)-m\le m<r-\epsilon=d(p,x)-\epsilon\le d(p',x),\]
        while $g(p')-g(x)\le0<d(x,p')$.
        \item $x \in \partial \vec B_r(p)$ and $\tau(x)\ge m$.  This is similar to case 2.  We omit the details.  
    \end{enumerate}
    Other inequalities are trivially satisfied: if $x, y \in \partial B$, then $g(x)-g(y) = \tau(x)-\tau(y) \le d(y,x)$ and if $p', p''\in B_\epsilon(p)$, then $g(p')-g(p'')=0\le d(p'',p')$.
    Thus $g$ is 1-forward Lipschitz.  
    \end{proof}

     By McShane's extension Theorem \ref{thm: extension}, $g$ extends to a $1$-forward Lipschitz map $G: B \to \R$.  Furthermore, since $g$ agrees with $\tau$ on $\partial B$, the map 
    \begin{align*}
        f: \Omega \to &\R \\
        x\mapsto & \begin{cases}
            \tau(x), & x\not\in \Gamma B \\
            G(\gamma\inverse x)+\varphi(\gamma), & x\in \gamma B \text{ for some } \gamma \in \Gamma.
        \end{cases}
    \end{align*}
    is continuous and $\varphi$-equivariant.  Finally, $f$ is $1$-forward Lipschitz, because it is locally $1$-forward Lipschitz.  
    Since $p$ is clearly not in $\Tension(f)$, we have supplied the contradiction that completes the proof of the lemma.
    \end{proof}

\begin{proof}[Proof of Theorem \ref{thm: optimal map}]
    An optimal $1$-forward Lipschitz $\varphi$-equivariant map $\tau: \Omega \to \R$ exists by Lemma \ref{lem: optimal maps exist}.
    The tension locus is a non-empty closed set by upper semi-continuity of the local forward Lipschitz constant and Lemma \ref{lem: Tension is non-empty}.  
    By Lemma \ref{lem: exists line}, every point $p\in\Tension(\tau)$ has a directed line segment whose interior passes through $p$.
    By Lemma \ref{lem: lines dont cross}, no other line segment in $\Tension(\tau)$ passes through $p$.
    This proves that $\Tension(\varphi)=\Tension (\tau)$ is a non-empty geodesic lamination.
    
    The proof of \cite[Theorem 1.12]{FLM:II} found in Appendix A of that paper runs essentially verbatim (using Theorem \ref{thm: extension} in place of McShane's Extension Theorem for Lipschitz maps) to prove that $\Tension(\varphi)$ is chain recurrent in this more general setting. The orientation comes from the orientation of $\R$, completing the proof of the theorem.
\end{proof}

All claims of Theorem \ref{thm: co-stable norm and optimal best maps} have now been addressed.

\bibliography{biblio-2}{}
\bibliographystyle{amsalpha.bst}
\Addresses
\end{document}